\documentclass[3p,times,procedia]{elsarticle}
\usepackage{makeidx}  
\usepackage{amsmath,amssymb,amsthm}
\usepackage{color}
\usepackage[all]{xy}
\usepackage{graphicx}
\usepackage{fancybox}
\usepackage{mathrsfs}
\usepackage{enumerate}
\usepackage{stmaryrd}
\usepackage{mathtools}
\usepackage{url}
\newtheorem{theorem}{Theorem}
\newtheorem{proposition}{Proposition}
\newtheorem{lemma}{Lemma}
\newtheorem{corollary}{Corollary}

\theoremstyle{definition}
\newtheorem{definition}{Definition}
\newtheorem{remark}{Remark}
\newtheorem{notation}{Notation}
\newtheorem{example}{Example}
\newtheorem{case}{Case}
\newcommand{\fsets}{{\bf Sets}_{\omega}}
\newcommand{\sets}{{\bf Sets}}
\newcommand{\Sh}{\mathrm{Sh}}
\newcommand{\pifsets}{\Cl_f \pi_1}

\newcommand{\mfsets}{\mathscr{B}_f M}
\newcommand{\Adfa}{\textrm{$A$-$\dfa$}}
\newcommand{\profmon}{\bf{Prof.Mon}}
\newcommand{\semigalois}{\bf{Semi\mathchar`-Galois}}
\newcommand{\freemon}{\bf{Free.Mon}}
\newcommand{\bialg}{\bf{BiAlg}}
\newcommand{\Et}{{\bf \textbf{{\'E}}t}}
\newcommand{\bool}{\bf{Bool}}

\newcommand{\Cl}{\mathscr{B}}

\newcommand{\dfa}{{\bf DFA}}
\newcommand{\sgc}{\langle \C, \F \rangle}

\newcommand{\rep}{\mathscr{B}_f}
\newcommand{\colim}{\mathrm{colim}}

\newcommand{\D}{{\bf C}}
\newcommand{\C}{\mathscr{C}}
\newcommand{\E}{\mathscr{E}}
\newcommand{\ff}{\mathscr{F}}
\newcommand{\G}{\mathscr{G}}

\newcommand{\varreg}{\mathfrak{R}}
\newcommand{\varmon}{\mathfrak{M}}
\newcommand{\vardfa}{\mathfrak{D}} 

\newcommand{\F}{\mathrm{F}}
\newcommand{\rr}{\mathrm{R}}

\newcommand{\V}{\mathcal{V}}

\newcommand{\vv}{{\bf V}}
\newcommand{\va}{\mathscr{V}}
\newcommand{\End}{\mathrm{End}}
\newcommand{\Auto}{\mathrm{Aut}}

\newcommand{\Hom}{\mathrm{Hom}}
\newcommand{\spec}{\mathrm{Spec}}

\newcommand{\uu}{{\bf u}}
\newcommand{\limmon}{\lim_{\Gamma} \End(\Gamma)}

\newcommand{\id}{\mathrm{id}}
\newcommand{\field}{\mathbb{F}}
\newcommand{\reg}{\mathrm{Reg}}
\renewcommand{\spec}{\mathrm{Spec}}

\usepackage{amssymb}

\usepackage[figuresright]{rotating}

\begin{document}

\begin{frontmatter}



\title{Semi-galois Categories I:\\ The Classical Eilenberg Variety Theory}


\author{Takeo Uramoto}

\address{Research Center of Pure and Applied Mathematics,\\ Graduate School of Information Sciences, Tohoku University}
\begin{abstract}
This paper is an extended version of our proceedings paper \cite{Uramoto16} announced at LICS'16; in order to complement it, this version is written from a different viewpoint including topos-theoretic aspects of \cite{Uramoto16} that were not discussed there. Technically, this paper introduces and studies the class of \emph{semi-galois categories}, which extend galois categories and are dual to profinite monoids in the same way as galois categories are dual to profinite groups; the study on this class of categories is aimed at providing an axiomatic reformulation of \emph{Eilenberg's theory of varieties of regular languages}--- a branch in formal language theory that has been developed since the mid 1960s and particularly concerns systematic classification of regular languages, finite monoids and deterministic finite automata. In this paper, detailed proofs of our central results announced at LICS'16 are presented, together with topos-theoretic considerations. The main results include (I) a proof of the duality theorem between profinite monoids and semi-galois categories, extending the duality theorem between profinite groups and galois categories; based on this results on semi-galois categories we then discuss (II) a reinterpretation of Eilenberg's theory from a viewpoint of duality theorem; in relation with this reinterpretation of the theory, (III) we also give a purely topos-theoretic characterization of classifying topoi $\Cl M$ of profinite monoids $M$ among general coherent topoi, which is a topos-theoretic application of (I). This characterization states that a topos $\E$ is equivalent to the classifying topos $\Cl M$ of some profinite monoid $M$ if and only if $\E$ is (i) coherent, (ii) noetherian, and (iii) has a surjective coherent point $p: \sets \rightarrow \E$. This topos-theoretic consideration is related to the logical and geometric problems concerning Eilenberg's theory that we addressed at LICS'16, which remain open in this paper. 
\end{abstract}
\begin{keyword}
semi-galois category \sep profinite monoid \sep Eilenberg variety theory \sep classifying topos of profinite monoid


\end{keyword}

\end{frontmatter}
\section{Introduction}
\label{s1}
\noindent
This paper is an extended version of our proceedings paper \cite{Uramoto16} announced at the 31st Annual ACM/IEEE symposium on \emph{Logic in Computer Science} (LICS'16), where we discussed an axiomatic reformulation of \emph{Eilenberg's theory of varieties of regular languages} \cite{Eilenberg} (cf.\ \S \ref{s1s1}). The current paper gives detailed proofs of the central results announced there; and in order to complement the discussions in \cite{Uramoto16}, we elaborate on them with topos-theoretic considerations and new results (cf.\ \S \ref{s1s2}). In particular, as an application of our result in \cite{Uramoto16}, we provide a simple characterization of classifying topoi $\Cl M$ of profinite monoids $M$ in a purely topos-theoretic terminology: That is, a topos $\E$ is equivalent to the classifying topos $\Cl M$ of some profinite monoid $M$ if and only if $\E$ is (i) \emph{coherent}, (ii) \emph{noetherian}, and (iii) \emph{has a surjective coherent point} $p: \sets \rightarrow \E$ (cf.\ \S \ref{s6}). The motivation of proving this characterization is related to the logical / geometric problems concerning Eilenberg's theory that we addressed at LICS'16 (cf.\ \S \ref{s7}); but this result itself will be of independent interest. 

Throughout this paper, the reader is assumed to be familier with basic terminlogies in automata theory, finite semigroup theory, and category theory including topos theory. The reader who is not familiar with these fields is refered to e.g.\ \cite{Hopcroft_Ulman,Pin} for automata and regular languages; \cite{Almeida94,Rhodes_Steinberg} for finite (and profinite) semigroups; and \cite{MacLane,MacLane_Moerdijk,Johnstone,Johnstone_elephant} for categories and topoi. Other necessary references will be mentioned below at the corresponding places. 
 
The rest of the current section (\S \ref{s1}) is devoted to a brief overview of research context (\S \ref{s1s1}) and our contribution (\S \ref{s1s2}) as well as other related works (\S \ref{s1s3}). The technical argument starts from the next section (\S \ref{s2}) and ends in the sixth section (\S \ref{s6}). The last section (\S \ref{s7}) is devoted to a discussion on some future directions, which are formulated based on the results developed in this paper.

\paragraph{Acknowledgements}
We are grateful to Naohiko Hoshino, Shinya Katsumata and the annonimous reviewers of our proceedings paper \cite{Uramoto16} for their helpful comments, suggestions and careful readings, which improved the original version of that paper. We are also grateful to Samuel van Gool and Mai Gehrke for their fruitful comments and discussions at LICS'16 on the logical problem that we addressed there; in particular, Gool suggested to us considering topoi over semi-galois categories in relation with this problem. We also thank Hisashi Aratake and Yoshihiro Maruyama for discussions on classifying topoi and coherent topoi; in particular, Aratake suggested to us several references on the topic and the discussion with him motivated us to consider our result in \S \ref{s6}. This work was done during we were at Research Institute for Mathematical Sciences, Kyoto University, to whom we are grateful for their hospitality. This work was supported by JSPS Kakenhi Grant Number JP16K21115.

\subsection{Research context}
\label{s1s1}
\noindent
The current paper, as well as our previous one \cite{Uramoto16}, is related to recent reconsiderations \cite{Gehrke_Grigorieff_Pin,Rhodes_Steinberg,Adamek_general,Chen_Urbat,Bojanczyk,Adamek_category} on \emph{Eilenberg's variety theorem} \cite{Eilenberg}, which claims a bijective correspondence between \emph{varieties of regular languages} and \emph{pseudo-varieties of finite monoids} (\S \ref{s5}). This theorem was first proved in 1975 and is rather classical, but was recently refleshed \cite{Gehrke_Grigorieff_Pin,Rhodes_Steinberg,Adamek_general,Chen_Urbat,Bojanczyk,Adamek_category} in the light of Stone-type duality theorems. The current work also belongs in this research line in that we also develop a reformulation of Eilenberg's theorem from a viewpoint of duality theorem, but in a slightly different way so that our formulation of the theory is coherent with a version of Galois theory (cf.\ \S \ref{s1s2}).

Technically, on the one hand, a \emph{variety of regular languages} is defined as a class of regular languages that is closed under taking Boolean combinations, quotients by finite words, and inverse images of monoid homomorphisms; on the other hand, a \emph{pseudo-variety of finite monoids} is a class of finite monoids that is closed under taking quotients, submonoids, and finite products (cf.\ \S \ref{s5}). Although these structures are irrelevant at least a priori, Eilenberg's variety theorem indicates that there is a canonical isomorphism between the lattice consisting of varieties of regular languages and that of pseudo-varieties of finite monoids. Regardless of its ostensible abstractness, this theorem played a fundamental role in clarifying a principle behind the traditional method based on finite monoids of proving several decision problems on regular languages, particularly including the decidability results of several fragments of \emph{B\"uchi's monadic second-order logic over finite words} \cite{Buchi}. (See e.g.\ \cite{survey_logic,Pin} for more background on logic over finite words.)

One of the major motivations to reconsider this rather classical result is that Eilenberg's theorem is a good model of systematic classification of classes of formal languages; therefore, it is natural to seek a right direction to extend the theory beyond regular languages. In a sense, the recent works \cite{Gehrke_Grigorieff_Pin,Rhodes_Steinberg, Adamek_general,Chen_Urbat,Bojanczyk,Adamek_category} came to grips with this project, where the authors particularly reviewed the theory from the viewpoint of Stone-type duality theorems.

The work of Gehrke, Grigorieff and Pin \cite{Gehrke_Grigorieff_Pin}, to our knowledge, is a watershed for this project, which complemented the early ideas due to Almeida \cite{Almeida94} and Pippenger \cite{Pippenger}. The authors gave a reinterpretation of Eilenberg's theorem (to be more precise, its extended variant due to Pin \cite{Pin_var}) based on Priestley duality theorem \cite{Priestley}; independently, Rhodes and Steinberg gave in their monograph \cite{Rhodes_Steinberg} on finite semigroup theory a similar (but more compatible with commutative algebra) review on Eilenberg's original theorem based on Stone duality theorem. Their insight immediately promoted the later works due to Ad{\'a}mek, Millius, Myers and Urbat \cite{Adamek_general,Adamek_category}, Chen and Urbat \cite{Chen_Urbat}, and Boja{\'n}czyk \cite{Bojanczyk}, where they discussed generic frameworks unifying several existing variants of Eilenberg's theorem that were studied after Eilenberg but proved in somewhat technically independent ways. At least to our knowledge, it is Ad{\'a}mek et al.\ \cite{Adamek_general,Adamek_category} and Chen et al.\ \cite{Chen_Urbat} who first showed it possibile to unify some variants of Eilenberg's theorem on regular languages (e.g.\ \cite{Pin,Polak,Reutenauer}) based on the idea of Gehrke et al.\ \cite{Gehrke_Grigorieff_Pin} and Rhodes et al.\ \cite{Rhodes_Steinberg}, where they adopted the framework of universal algebra and coalgebra in particular--- by that, they also proved a new variant of Eilenberg's theorem; Boja{\'n}czyk \cite{Bojanczyk} made the first step to a uniform generalization of Eilenberg's theory for several sorts of languages including tree and infinite-word languages based on the concept of monads. This line of extensions is being continued by several authors, e.g.\ \cite{Gehrke_ultrafilter,Millius_profinite_monad,Chen_Urbat_product}, and represents one of leading research trends in the context of algebraic language theory.

\subsection{Our contribution}
\label{s1s2}
\noindent
In this paper, we develop a slightly different framework so that Eilenberg's variety theory becomes coherent with Grothendieck's formulation of Galois theory \cite{SGA}. This yet another formulation of Eilenberg's theory is (I) to develop a new approach to classical problems in the theory, as well as (II) to indicate yet another direction to extend this theory; a more detailed backgound motivation of the current study will be discussed in the last section (\S 7) with some discussions on further problems based on the results developed in this paper. Basically, the current paper is committed to developing particularly fundamental components of our intended framework. 

Technically speaking, this paper focuses on the study of the general structure of \emph{semi-galois categories} (\S \ref{s2}) in relation with a review of Eilenberg's varierty theory as well as its topos-theoretic consideration. This paper, except for the last discussion section (\S \ref{s7}), consists of five main sections (\S 2 - \S 6), which is divided roughly into three parts: That is, (\S 2 - \S 4), (\S 5), and (\S 6), in view of their respective subjects. 

\paragraph{1. General study of semi-galois categories}
In the first part (\S 2 - \S 4), we introduce the class of \emph{semi-galois categories} and investigate their general structures. The axiom of semi-galois categories is provided in \S \ref{s2}, where we also give some proto-typical examples of semi-galois categories (\S \ref{s2s2}). The main goal here is to prove the duality theorem between profinite monoids and semi-galois categories (Theorem \ref{duality}, \S \ref{s4}), which extends the duality theorem between profinite groups and galois categories \cite{SGA}. Several concepts necessary to this goal are studied in \S \ref{s3}, including \emph{galois objects} (\S \ref{s3s1}) and \emph{fundamental monoids} (\S \ref{s3s2}) of semi-galois categories. The technical results developed in this first part will be used for the consideration of the second part (\S \ref{s5}) and the third part (\S 6). 

\paragraph{2. Review of Eilenberg's variety theory}
The general study of semi-galois categories will be first used in our review on Eilenberg's variety theory \cite{Eilenberg}; the second part of this paper (\S \ref{s5}) is devoted to this consideration. Technically speaking, the central theorem in this theory is \emph{Eilenberg's variety theorem}; and our review is essentially about a reformulation of this theorem.

To be more precise, we actually review two variants of Eilenberg's variety theorem due to Straubing \cite{Straubing} and Chaubard, Pin and Straubing \cite{Chaubard_Pin_Straubing} rather than original Eilenberg's version; the reason of this specific choice of our starting point will be discussed in \S \ref{s5} as well. Briefly speaking, these theorems state canonical bijective correspondences between certain classes of (i) regular languages, (ii) finite monoids, and (iii) deterministic finite automata (DFAs) (\S \ref{section var one}); in relation to our reinterpretation of Eilenberg's theory, we give yet another proof of these theorems based on appropriate duality theorems (\S \ref{section var two}). 
This proof is not intended to simplify the original proofs; but instead, to gain a more conceptual understanding of these theorems by highlightening a duality principle behind them. This reinterpretation of the variety theorems then gives us a reason to get concerned with the general structure of semi-galois categories; in this relation, we proceed in this paper to specify the class of topoi equivalent to the class of semi-galois categories in order to obtain a topos-based classification of (local) varieties of regular languages. (See below.)

\paragraph{3. Topos Representation}
Technically speaking, the subject of the third part (\S \ref{s6}) is to give a fragment of the duality between \emph{pretopoi} and \emph{coherent topoi} \cite{Johnstone}. On the one hand, pretopoi are those categories which satisfy certain exactness conditions; clearly, the class of semi-galois categories is a proper subclass of that of pretopoi. On the other hand, coherent topoi are those topoi which are equivalent to sheaf topoi $\Sh(\C,J_\C)$ over those sites $(\C,J_\C)$ which satisfy certain finitary condition. In the literature on topos theory, it has been known that there is a duality between pretopoi and coherent topoi (cf.\ \S \ref{s6s1}). In the case of galois categories (hence pretopoi), the coherent topoi dual to them are exactly the classifying topoi $\Cl G$ of profinite groups $G$; interestingly, this class of topoi has several purely-topos theoretic characterizations: For instance, a topos $\E$ is equivalent to the classifying topos $\Cl G$ of some profinite group $G$ if and only if $\E$ is a \emph{hyperconnected pointed topos with proper diagonal} \cite{Moerdijk_proper}. This gives a ``galois-categorical fragment'' of the duality between pretopoi and coherent topoi, in that it specifies the class of coherent topoi exactly dual to galois categories among general coherent topoi.

We give in \S \ref{s6} a semi-galois categorical counterpart to this fact, applying our duality theorem between profinite monoids and semi-galois categories. As mentioned above, the class of semi-galois categories is a proper subclass of pretopoi; and it can be shown that the coherent topoi dual to semi-galois categories are exactly the classifying topoi $\Cl M$ of profinite monoids $M$ as in the case of galois categories. We then see that this class of coherent topoi also admits a purely topos-theoretic characterization: That is, a topos $\E$ is equivalent to the classifying topos $\Cl M$ of some profinite monoid $M$ if and only if $\E$ is (i) coherent, (ii) \emph{noetherian}, and (iii) \emph{has a surjective coherent point $p: \sets \rightarrow \E$} (cf.\ \S \ref{s6s2}). This characterization gives a semi-galois categorical fragment of the duality between pretopoi and coherent topoi. Together with the result in \S \ref{s5}, this provides a topos-based classification of (local) varieties of regular languages in the precise sense of canonical bijective correspondence described in \S \ref{s5} and \S \ref{s6}. 

\subsection{Other related work}
\label{s1s3}
\noindent
The current work also belongs in the context of studies on several duality theorems, particularly, between several classes of categories (such as galois categories \cite{SGA}, tannakian categories in the sense of Deligne and Milne \cite{Deligne_Milne}, and their several variants) and corresponding classes of algebras (such as profinite groups, algebraic groups etc.). The most general result in this direction is, to the best of our knowledge, the work due to Sch\"appi \cite{Schappi}, where he developed a generic framework of duality between tannakian-type enriched categories and comonoids in symmetric monoidal categories; the opposite of our duality theorem is closely related to his duality in that opposite of our duality concerns the duality between (opposite of) semi-galois categories and comonoids in the category $\bool$ of Boolean algebras. Nevertheless, an unignorable gap still exists between our duality theorem and the one due to Sch\"appi maily because of the difference of the basic setting. (The opposites of semi-galois categories are still those enriched in $\sets$, rather than in $\bool$.) As far as we could find, it is Alain Brugui{\`e}res who first claimed in his 2013 slide \cite{Bruguieres} a generic framework that aims to unify several existing duality theorems for tannakian-type categories as well as their opposite analogue (i.e.\ the one for galois categories); importantly, his claim of duality theory subsumes an essential part of our duality theorem. Although, according to him \cite{Bruguieres_personal}, his preprint still remains in preparation, his work should be recognized here as well.


Provided here is a proof of the duality theorem between profinite monoids and semi-galois categories in the full form of a contravariant equivalence between suitable categories. Our proof is given intentionally as a natural extension of an elementary proof of the duality between profinite groups and galois categories. This elementary proof will be valuable in its own right mostly for those who are particularly concerned with profinite monoids and with comparison with the case of profinite groups: The structure of profinite monoids is still mysterious in general, while their analysis plays a fundamental role in Eilenberg's theory \cite{Almeida}. We shall not proceed to generalize this duality theorem itself in this paper, since we instead proceed to another direction in \S \ref{s6}: We consider another dual (coherent topoi) of semi-galois categories and see that they are axiomatized exactly as coherent noetherian topoi with coherent surjective points. This consideration serves yet another (topos-theoretic) viewpoint on the classical Eilenberg variety theory. 

\section{Semi-galois Categories}
\label{s2}
\noindent
Here we prepare basic concepts and constructions on semi-galois categories (\S \ref{s2s1} -- \S \ref{s2s2}). The axiom of semi-galois categories is described in the first subsection \S \ref{s2s1}, where we also fix some general notations. Several proto-typical examples of semi-galois categories as well as their corresponding topoi are discussed in the second subsection \S \ref{s2s2}. Our detailed study on the general structure of semi-galois categories will start from the next section \S \ref{s3}; our first goal is to prove in \S \ref{s4} the duality theorem for semi-galois categories (Theorem \ref{duality}). 

\subsection{Axiom and notations}
\label{s2s1}
\noindent
\begin{notation}
Throughout this paper, we denote by $\sets$ the category of sets and maps; by $\fsets$ the full subcategory of $\sets$ consisting of finite sets. The empty set is denoted by $\emptyset$, while the singleton set $\{ \emptyset \}$ is denoted by $1$. Although $\sets$ can be replaced by more general topoi in the following argument, we assume that categories, sites, and topoi in this paper are all over $\sets$ in order to avoid inessential abstractness of the theory. 
\end{notation}

\emph{Semi-galois categories} are defined as follows, which constitute the class of categories of our central concern in this paper. 

\begin{definition}[Semi-galois category]
A \emph{semi-galois category} is a pair $\sgc$ of (i) an (essentially) small category $\C$ and (ii) a functor $\F:\C \rightarrow \fsets$ satisfying the following six axioms: 
\begin{description}
 \item[$C_0$)] $\C$ has the initial object $\emptyset_\C$ and the final object $1_\C$;
 \item[$C_1$)] $\C$ has finite pullbacks and finite pushouts;
 \item[$C_2$)] every arrow $f:X \rightarrow Y$ in $\C$ factors as $f = j_f \circ \pi_f$ such that $\pi_f: X \twoheadrightarrow Z$ is an epimorphism, and $j_f: Z \hookrightarrow Y$ is a monomorphism;
      \begin{equation*}
        \xymatrix{ X \ar[rr]^f \ar@/_/@{->>}[rd]_{\pi_f} && Y \\
                   & Z \ar@/_/@{^{(}->}[ru]_{j_f} & 
                  }
      \end{equation*}
 \item[$F_0$)] $\F(\emptyset_\C)=\emptyset$ and $\F(1_\C)=1$;
 \item[$F_1$)] $\F$ preserves finite pullbacks and finite pushouts;
 \item[$F_2)$] $\F$ reflects isomorphisms.
\end{description}
The functor $\F$ is called a \emph{fiber functor} on $\C$; sometimes we call $\C$ itself a semi-galois category if it forms a semi-galois category $\sgc$ with some fiber functor $\F$ on it. 
\end{definition}

\begin{notation}
\label{composition of arrows}
Here, the composition of arrows $f: X \rightarrow Y$, $g: Y \rightarrow Z$ in a semi-galois category is denoted $g \circ f: X \rightarrow Z$ and sometimes by $gf$ for short. Also, for each $f: X \rightarrow Y$ in $\C$, the corresponding map $\F(f): \F(X) \rightarrow \F(Y)$ in $\fsets$ is denoted simply by $f_*: \F(X) \rightarrow \F(Y)$, and its action on elements $\xi \in \F(X)$ is \emph{from the left}, i.e.\ denoted $f_* \xi \in \F(X)$. (Note that this and the next notations are reverse to those in our proceedings paper \cite{Uramoto16}.)
\end{notation}

\begin{notation}
\label{composition of natural transformation}
 To the contrary, for notational reasons, natural transformations $\phi: \F \Rightarrow \F$ on a fiber functor $\F$ are assumed to act from the right on each element $\xi \in \F(X)$ at each object $X \in \C$. So, we denote by $\xi \phi_X \in \F(X)$ the action of a natural transformation $\phi: \F \Rightarrow \F$ on $\xi \in \F(X)$ at $X \in \C$; the composition of $\phi$ and $\psi: \F \Rightarrow \F$ (first $\phi$; second $\psi$) is written as $\phi \cdot \psi$ or simply $\phi \psi$. (Thus one has $\xi (\phi \cdot \psi)_X = (\xi \phi_X) \psi_X$ for each $X \in \C$, $\xi \in \F(X)$ and $\phi, \psi: \F \Rightarrow \F$.)
\end{notation}

\begin{notation}
For symplicity, the initial object $\emptyset_\C$ and the final object $1_\C$ of $\C$ are abusively denoted by $\emptyset$ and $1$ respectively. Also, for each $f: X \rightarrow Y$ in $\C$, we say \emph{the image of $f$} to mean an object $Z$ such that $f:X \rightarrow Y$ factors as $X \twoheadrightarrow Z \hookrightarrow Y$ as in the axiom $C_2$ above. Since such an object $Z$ is unique up to isomorphism, we shall denote $Z =: \mathrm{Im}(f)$ and say ``the'' image of $f$. 
\end{notation}



\subsection{Examples}
\label{s2s2}
\noindent

\subsubsection*{Example 1: The category of finite $M$-sets}
\label{classifying topoi of profinite monoids}
\noindent
Given a profinite monoid $M$, we can construct the category of (finite) \emph{right $M$-sets} and \emph{$M$-equivariant maps}, which provides prototypes of topoi (resp.\ semi-galois categories) studied in this paper. Formally, these structures are defined as follows:

\begin{definition}[right $M$-set]
Let $M$ be a profinite monoid. A \emph{right $M$-set} (or simply, \emph{$M$-set}) is a pair $\langle S, \rho \rangle$ of (i) a set $S \in \sets$ and (ii) a continuous map $\rho: S \times M \rightarrow S$ with respect to the discrete topology on $S$ that satisfies (ii-a) $\rho (\rho(\xi,m), n) = \rho (\xi, mn)$ and (ii-b) $\rho (\xi,1) = \xi$ for every $\xi \in S$ and $m,n \in M$. (Here $1 \in M$ denotes the identity of $M$.) An $M$-set $\langle S, \rho \rangle$ is called \emph{finite} if $S$ is a finite set. 
\end{definition}

\noindent
We use capital romans $X,Y,Z \cdots$ to denote $M$-sets; and for each $M$-set $X$, we use symbols $S_X$ and $\rho_X$ to mean $X=\langle S_X, \rho_X \rangle$. That is, $S_X$ is the \emph{underlying set} of the $M$-set $X$, while $\rho_X: S_X \times M \rightarrow S_X$ is the \emph{action map} for $X$. Given an $M$-set $X$, we write $\xi \cdot m := \rho_X (\xi, m)$ for each $\xi \in S_X$ and $m \in M$.

\begin{definition}[$M$-equivariant map]
 Let $X, Y$ be $M$-sets. An \emph{$M$-equivariant map} $f:X \rightarrow Y$ from $X$ to $Y$ is a map $f_*: S_X \rightarrow S_Y$ between their underlying sets that \emph{commutes with $M$-actions}, that is, $f_* (\xi \cdot m) =  (f_* \xi) \cdot m$ for every $\xi \in S_X$ and $m \in M$. 
\end{definition}

\noindent
These structures form a category:

\begin{definition}[The categories $\Cl M$ and $\Cl_f M$]
The category $\Cl M$ is defined as the one whose objects are (not necessarily finite) right $M$-sets and arrows are $M$-equivariant maps between them; also, the full subcategory of $\Cl M$ consisting of finite right $M$-sets is denoted $\Cl_f M$. 
\end{definition}

\noindent
The latter category, $\Cl_f M$, is equipped with the canonical functor $\F_M: \Cl_f M \rightarrow \fsets$ that assigns to each $X \in \Cl_f M$ the underlying set $S_X$; and to each $f: X \rightarrow Y$ the map $f_*: S_X \rightarrow S_Y$. Then the pair $\langle \Cl_f M, \F_M \rangle$ forms a semi-galois category. 

\begin{remark}[left and right]
 These facts and the following arguments hold even when right $M$-sets are replaced by \emph{left} $M$-sets, although one needs to shift notations given in \S \ref{s2s1} from right to left and vice versa. In this paper we deal with right $M$-sets (unlike we dealt with left $M$-sets in \cite{Uramoto16}) in order to be compatible with a convention in automata theory. 
\end{remark}

\begin{remark}[$M$-sets for profinite semigroups]
 One can define right $M$-sets and $M$-equivariant maps for profinite \emph{semigroups} $M$ in a similar way to the case of profinite monoids, except for (ii-b); and the category of finite $M$-sets and $M$-equivariant maps also forms a semi-galois category. 
\end{remark}

\begin{remark}
 While the semi-galois category $\Cl_f M$ itself is not a topos in general (indeed, just a \emph{pretopos}), the category $\Cl M$ on the other forms a topos. In fact, $\Cl M$ is the \emph{classifying topos} of the profinite monoid $M$. As known in topos theory, \emph{coherent topoi} (including $\Cl M$) are dually equivalent to pretopoi (including $\Cl_f M$) in a certain definite way; and under this equivalence, the semi-galois category $\Cl_f M$ corresponds to the classifying topos $\Cl M$ (cf.\ \S \ref{s6}). The study on semi-galois categories (\S \ref{s2} -- \S \ref{s4}) will be used in \S \ref{s6} to give a characterization of the topoi of the form $\Cl M$ (i.e.\ exactly the class of classifying topoi of profinite monoids) among the class of general coherent topoi in the same manner as that classifying topoi of profinite groups could be characterized in several ways. 
\end{remark}

\subsubsection*{Example 2: The category of deterministic finite automata}
\noindent
The semi-galois category $\Adfa$ whose objects are \emph{deterministic finite automata} (Definition \ref{dfa}) is of central concern in \S \ref{s5}. As briefly mentioned below (cf.\ Remark \ref{dfas are simplest}), this semi-galois category is in some sense the simplest example of semi-galois categories, but will play a central role in the reinterpretation of classical concepts in Eilenberg's theory. For the basic concepts concerning DFAs and regular languages, the reader is refered to e.g.\ \cite{Hopcroft_Ulman,Pin}. 

\begin{notation}
First let us fix some general notations and terminology on finite words. We call \emph{alphabets} any finite non-empty sets; symbols $A, B \cdots$ are used to denote alphabets. Elements of an alphabet $A$ are called \emph{letters}; finite sequences $w=a_1 a_2 \cdots a_n$ of letters $a_i$ in $A$ are called \emph{finite words over $A$}, where $n$ is the \emph{length} of $w$ and denoted $n=|w|$. The \emph{empty word} is the one of length $0$ and denoted $\varepsilon$. The set of all finite words over $A$ is denoted $A^*$, which forms a monoid (free over $A$) with multiplication of finite words $u, v \in A^*$ given by the concatenation $uv \in A^*$; the empty word $\varepsilon \in A^*$ is the unit of this monoid $A^*$. 
\end{notation}

\begin{definition}[deterministic finite automaton]
\label{dfa}
 Let $A$ be an alphabet. A \emph{deterministic finite automaton} (DFA) over $A$ is a pair $\langle S,\rho \rangle$ of (i) a finite set $S$ and (ii) a map $\rho: S \times A \rightarrow S$. Elements in $S$ are called \emph{states} of the DFA, while $\rho$ is called its \emph{transition function}. 
\end{definition}

\begin{notation}
\label{transition function}
The transition function $\rho: S \times A \rightarrow S$ can be extended to (and so, identified with) a map $\hat{\rho} : S \times A^* \rightarrow S$ by induction on the length of finite words: That is, given $\rho: S \times A \rightarrow S$, we can define $\hat{\rho}: S \times A^* \rightarrow S$ by $\hat{\rho}(\xi,\varepsilon) := \xi$ and $\hat{\rho}(\xi,wa) := \hat{\rho}(\hat{\rho}(\xi,w), a)$ for $w \in A^*$ and $a \in A$. We denote $\hat{\rho}$ simply by $\rho: S \times A^* \rightarrow S$; and write $\rho(\xi,u) =: \xi \cdot u$ for each $\xi \in S$ and $u \in A^*$. Moreover, similarly to the case of $M$-sets, we use capital romans $X,Y,Z \cdots$ to denote DFAs; and if $X$ is a DFA, symbols $S_X$ and $\rho_X$ represent respectively the set of states in $X$ and its transition function, i.e.\ $X = \langle S_X,\rho_X \rangle$. 
\end{notation}

\begin{remark}[DFA as graph]
\label{dfa as graph}
 DFAs over an alphabet $A$ are often writen as a kind of \emph{directed graphs} whose edges are labeled by letters in $A$; to see this convention and to use it below, a few remarks and notation are in order here. (See also Notation \ref{transition function graphically}.)

 Let $A$ be an alphabet and let $X=\langle S ,\rho \rangle$ be a DFA over $A$. For states $\xi, \xi' \in S$ and a letter $a \in A$, we write $\xi \xrightarrow{a} \xi'$ if and only if $\xi' =\xi \cdot a$, i.e.\ $\xi' = \rho(\xi, a)$. In more graph-theoretic terminology, this says that a DFA $X$ defines a directed finite graph $\Gamma_X$ whose vertices are states $\xi \in S$, and edges are labeled by letters $a \in A$ as $\xi \xrightarrow{a} \xi'$ in $\Gamma$. Conversely, a directed finite graph $\Gamma$ with edges labeled by $A$ arises from a DFA exactly when for each vertex $\xi$ in $\Gamma$ and $a \in A$ there is a \emph{unique} vertex $\xi'$ such that $\xi \xrightarrow{a} \xi'$ in $\Gamma$. In other words, under this correspondence, DFAs are equivalent structures to such directed finite graphs. 
\end{remark}

\begin{notation}
\label{transition function graphically}
The above notation (Notation \ref{transition function}) is rephrased in this graph-theoretic terminology. Let $X=\langle S,\rho \rangle$ be a DFA over $A$; $u = a_1 a_2 \cdots a_n \in A^*$ be a finite word over $A$; and $\xi \in S$ be a state in $X$. Then we have $\xi' = \xi \cdot u$ (i.e.\ $\xi' = \hat{\rho}(\xi,u)$) if and only if there is a \emph{path from $\xi$ to $\xi'$ labeled by $u$}, i.e.\ there are adjacent edges $\xi \xrightarrow{a_1} \xi_1 \xrightarrow{a_2} \xi_2 \cdots \xrightarrow{a_n} \xi_n = \xi'$ starting from $\xi$ and ending at $\xi'$ in the associated graph $\Gamma_X$. When this is the case, we write $\xi \xrightarrow{u} \xi'$ in what follows, representing $\xi \cdot u = \xi'$. 
\end{notation}

The morphisms between DFAs are those called \emph{transition-preserving maps}: 

\begin{definition}[transition-preserving map]
 Let $A$ be an alphabet and $X, Y$ be DFAs over $A$. A \emph{transition-preserving map} $f: X \rightarrow Y$ from $X$ to $Y$ is a map $f_*: S_X \rightarrow S_Y$ such that $(f_* \xi) \cdot u = f_* (\xi \cdot u)$ for every $u \in A^*$ and $\xi \in S_X$. In other words, a transition-preserving map $f: X \rightarrow Y$ is a map of directed graphs $\Gamma_X \rightarrow \Gamma_Y$ that preserves labels on edges.
\end{definition}

\begin{definition}[The category of DFAs]
 Let $A$ be an alphabet. Then the category $\Adfa$ is defined as the one whose objects are DFAs over $A$ and arrows are transition-preserving maps between them. 
\end{definition}

This category $\Adfa$ is equipped with a functor $\F_A: \Adfa \rightarrow \fsets$ that assigns to each $X \in \Adfa$ the set of states $S_X$; and to each $f: X \rightarrow Y$ the map $f_*: S_X \rightarrow S_Y$. Then the pair $\langle \Adfa, \F_A \rangle$ forms a semi-galois category. 

\begin{remark}
\label{dfas are simplest}
 It is readily seen that the category $\Adfa$ is canonically equivalent to the category $\Cl_f\widehat{A^*}$ (cf.\ \emph{Example 1}) for the free profinite monoid $\widehat{A^*}$. This is simply because DFAs over $A$, say $\rho: S \times A^* \rightarrow S$, are essentially the same as finite $\widehat{A^*}$-sets, say $\rho: S \times \widehat{A^*} \rightarrow S$; note that the monoid $A^*$ is dense in $\widehat{A^*}$. In this sense, the category $\Adfa$ of DFAs is the simplest example of semi-galois categories. 
\end{remark}


\subsubsection*{Example 3: galois category}
\noindent
\emph{Galois categories} \cite{SGA} are always semi-galois categories, although this fact is not very apparent from the original axiom of galois categories given as follows. (The reader interested in galois categories is refered to more detailed references e.g.\ \cite{Lenstra,Szamuely,Tonini} on galois categories.)

\begin{definition}[galois category \cite{SGA}]
A \emph{galois category} is a pair $\sgc$ of (i) an (essentially) small category $\C$ and (ii) a functor $\F:\C \rightarrow \fsets$ satisfying the following seven axioms: 
\begin{description}
 \item[$C_0$)] $\C$ has the final object $1_\C$, finite pullbacks, and finite coproducts, including the initial object $\emptyset_\C$;
 \item[$C_1$)] every object $X \in \C$ has \emph{universal quotients $X/G$ by any subgroup} $G \leq \mathrm{Aut}(X)$ (cf.\ Remark \ref{quotients});
 \item[$C_2$)] every arrow $f:X \rightarrow Y$ in $\C$ factors as $f = j_f \circ \pi_f$ such that $\pi_f: X \twoheadrightarrow Z$ is an epimorphism, and $j_f: Z \hookrightarrow Y$ is a monomorphism;
 \item[$C_3$)] every monomorphism $j: X \hookrightarrow Y$ is a direct summand of a coproduct $X \sqcup X' = Y$; 
 \item[$F_0$)] $\F(1_\C)=1$, and $\F$ preserves finite pullbacks and finite coproducts;
 \item[$F_1$)] $\F$ preserves universal quotients and epimorphisms;
 \item[$F_2)$] $\F$ reflects isomorphisms.
\end{description}
Similarly the functor $\F$ is called a \emph{fiber functor} on $\C$; again, we call $\C$ itself a galois category if it is for some fiber functor on it. 
\end{definition}

\begin{remark}[universal quotient]
\label{quotients}
 Let $X$ be an object in $\C$; and $G \leq \Auto(X)$ be a subgroup of automorphisms on $X$. The \emph{universal quotient} of $X$ by $G$ is defined as an arrow $p: X \rightarrow Y$ such that:
 \begin{description}
  \item[$Q_0$)] the arrow $p:X \rightarrow Y$ \emph{annihilates} $G$, i.e.\ we have $p\circ g = p$ for every $g \in G$; 
  \item[$Q_1$)] if $q: X \rightarrow Z$ annihilates $G$, then there is a unique arrow $q': Y \rightarrow Z$ such that the following diagram commutes:
       \begin{eqnarray*}
         \xymatrix{ X \ar[r]^q \ar[d]_p & Z  \\
                    Y \ar@{.>}[ru]_{\exists! q'} &
                  }
       \end{eqnarray*}
 \end{description}
 If $p: X \rightarrow Y$ satisfies $Q_0$ and $Q_1$, then one can see that it is necessarily an epimorphism and unique up to canonical isomorphism; thus we write $p: X \twoheadrightarrow X/G$ to mean that $p$ is the universal quotient of $X$ by $G$; and if we denote $Y=X/G$, it means that $Y$ is the codomain of the universal quotient $p: X \rightarrow X/G$. The axiom $C_1$ says that, more precisely, there exists a universal quotient $p: X \rightarrow X/G$ for every $X \in \C$ and $G \leq \Auto(X)$. 
\end{remark}

\begin{remark}[galois and semi-galois]
 Note that the axiom of galois categories itself does not require the existence of e.g.\ all finite pushouts (more than just coproducts), while that of semi-galois categories does. So it is not apparent from the axiom itself that galois categories are always semi-galois, i.e.\ have all finite pushouts. This is, however, the case because it is known that every galois category $\C$ is equivalent to the category of the form $\Cl_f G$ (cf.\ \emph{Example 1}) for some profinite group $G$; and as mentioned in \emph{Example 1}, $\Cl_f G$ satisfies the axiom of semi-galois category. 

The only difference that characterizes galois categories among semi-galois categories is the axiom $C_3$ above, namely, \emph{``every monomorphism $j: X \hookrightarrow Y$ is a direct summand of a coproduct $X \sqcup X' = Y$''}. To be more precise, a semi-galois category $\sgc$ is a galois category if and only if it satisfies $C_3$; and not every semi-galois category satisfies $C_3$. In fact, the semi-galois category $\Adfa$ of DFAs (cf.\ Remark \ref{fundamental facts on galois categories}) is a typical example that fails to satisfy $C_3$. 

\end{remark}

\begin{remark}[Some references on galois categories]
\label{fundamental facts on galois categories}
 A fundamental result on galois categories is the duality theorem between (suitable categories of) profinite groups and galois categories (e.g.\ \cite{Tonini}). In particular this duality theorem implies that (i) every galois category $\sgc$ is in fact equivalent to that of the form $\langle \Cl_f G, \F_G \rangle$ for a profinite group $G$ called the \emph{fundamental group} of $\sgc$ denoted $\pi_1(\C,\F)$; and (ii) that, roughly speaking, there is a bijective correspondence between ``galois subcategories'' of $\sgc$ and closed subgroups of $\pi_1(\C,\F)$ (cf.\ \S \ref{s4}). The usage of galois categories in topology and algebraic geometry can be found in e.g.\ \cite{Lenstra,Szamuely}--- where (ii) above instanciates the classical Galois' fundamental theorems for field extensions and covering spaces; for a well-organized proof of these fundamental results on galois categories the reader is refered to e.g.\ \cite{Tonini} as well as \cite{Lenstra}. 

\end{remark}

\section{Fundamental Monoids}
\label{s3}
\noindent
From now on we are going to develop the general theory on the structure of semi-galois categories (\S \ref{s3} -- \S \ref{s4}). The duality theorem between profinite monoids and semi-galois categories is proved in Theorem \ref{duality}, \S \ref{s4}, extending that between profinite groups and galois categories. This duality theorem is used in this paper in two ways: First we will review Eilenberg's theorem in \S \ref{s5} based on the duality theorem; and motivated by this review of the theory, we then give a relatively simple characterization of classifying topoi $\Cl M$ of profinite monoids $M$. 

In the current section \S \ref{s3}, we prepare a necessary construction of \emph{fundamental monoids} $\pi_1(\C,\F)$ of semi-galois categories $\sgc$ (\S \ref{s3s3}), which are profinite monoids canonically associated to semi-galois categories. In the case where $\sgc$ is a galois category, the fundamental monoid $\pi_1(\C,\F)$ is equal to the classical fundamental group. The construction here is based on the concept of \emph{galois objects} in semi-galois categories $\sgc$ (\S \ref{s3s2}), naturally extending that of fundamental groups; before discussing this construction, we firstly make some elementary general remarks on semi-galois categories in \S \ref{s3s1}, which will be used freely throughout this paper without any further comments.

\subsection{Elementary remarks}
\label{s3s1}
\noindent
\begin{lemma}
 A semi-galois category has equalizers and coequalizers for every pair of parallel arrows $h,k: X \rightarrow Y$; and also, the fiber functor $\F$ preserves them. 
\end{lemma}
\begin{proof}
 This follows from the general fact that equalizers (resp.\ coequalizers) can be constructed using the initial object $\emptyset$ and finite pullbacks (resp.\ the final object $1$ and finite pushouts). See MacLane \cite{MacLane}. 
\end{proof}

\begin{lemma}
 Let $f: X \rightarrow Y$ be an arrow in a semi-galois category $\langle \C, \F \rangle$. 
 \begin{enumerate}
  \item $f$ is an epimorphism in $\C$ if and only if $f_*$ is a surjection in $\fsets$;
  \item $f$ is a monomorphism in $\C$ if and only if $f_*$ is an injection in $\fsets$.
 \end{enumerate}
\end{lemma}
\begin{proof}
 Firstly, recall that an arrow $f: X \rightarrow Y$ in a category is an epimorphism if and only if the following diagram is a pushout diagram:
\begin{eqnarray*}
 \xymatrix{
   X \ar[r]^f \ar[d]_f & Y \ar@{=}[d]^{\id_Y} \\
   Y \ar@{=}[r]_{\id_Y} & Y
} 
\end{eqnarray*}
The first claim follows from this fact together with the axiom $F_1$ that the fiber functor $\F: \C \rightarrow \fsets$ preserves finite pushouts and the axiom $F_2$ that $\F$ reflects isomorphisms. The second claim is similar. 
\end{proof}

\begin{lemma}
 The fiber functor $\F: \C \rightarrow \fsets$ is faithful: i.e.\ for each parallel arrows $h,k: X \rightarrow Y$, the equality $h_* = k_*$ implies $h=k$. 
\end{lemma}
\begin{proof}
 Let $e: Z \rightarrow X$ be an equalizer of $h,k: X \rightarrow Y$. If $h_* = k_*$, then $e_*: \F(Z) \rightarrow \F(X)$ is an isomorphism in $\fsets$. Thus $e: Z \rightarrow X$ itself is an isomorphism by the axiom $F_2$, which implies $h=k$. 
\end{proof}

\begin{corollary}
 A semi-galois category $\sgc$ is \emph{locally finite}: That is, for every objects $X,Y \in \C$, the set of arrows $\Hom_\C(X,Y)$ is a finite set. In particular, the monoid $\End_\C(X)$ of endomorphisms on $X$ is a finite monoid. 
\end{corollary}
\begin{proof}
 Immediate from the above lemma. 
\end{proof}

\begin{notation}
 In what follows we sometimes ommit the subscripts in $\Hom_\C(X,Y)$ and $\End_\C(X,Y)$ if it is clear from the context that $X,Y$ are objects in $\C$. 
\end{notation}

\noindent
As in the study of galois categories, we need the construction of universal quotients $p: X \twoheadrightarrow X/G$ of objects $X \in \C$. However, in the case of semi-galois categories, we actually use a more flexible construction of quotients; as one will see in arguments in \S \ref{s4}, this is roughly because quotients $M \twoheadrightarrow N$ of profinite monoids are given by \emph{congruences} $M \twoheadrightarrow M/E$ unlike quotients $G \twoheadrightarrow H$ of profinite groups could be given by (normal) subgroups $G \twoheadrightarrow G/N$. 

For this reason we need the concept of \emph{universal quotient $p: X \twoheadrightarrow X/E$ of an object $X$ by a right congruence $E$ on the monoid $\End(X)$}; this concept in fact makes sense for an arbitrary binary relation $E \subseteq \End(X) \times \End(X)$ on endomorphisms on $X$, as defined as follows (but always isomorphic to universal quotients by right congruences, cf.\ Lemma \ref{universal quotient is by right congruence}). 

\begin{definition}[universal quotient]
 Let $X \in \C$ be an object and $E \subseteq \End(X) \times \End(X)$ be an arbitrary relation on $\End(X)$. The \emph{universal quotient of $X$ by $E$} is defined as an arrow $p: X \rightarrow Y$ such that:
\begin{description}
 \item[$Q_0$)] the arrow $p:X \rightarrow Y$ \emph{coequalizes} $E$, i.e.\ we have $p \circ h = p \circ k$ for every $(h,k) \in E$;
 \item[$Q_1$)] if an arrow $q: X \rightarrow Z$ coequalizes $E$, there is a unique arrow $q': Y \rightarrow Z$ such that the following diagram commutes:
       \begin{eqnarray*}
         \xymatrix{ X \ar[r]^q \ar[d]_p & Z  \\
                    Y \ar@{.>}[ur]_{\exists! q'} &
                  }
       \end{eqnarray*}
\end{description}
If $p: X \rightarrow Y$ satisfies $Q_0$ and $Q_1$, then one can see that $p$ is necessarily an epimorphism and unique up to canonical isomorphism; thus we write $p: X \twoheadrightarrow X/E$ to mean that $p$ is the universal quotient of $X$ by $E$; and if we denote $Y=X/E$, it means that $Y$ is the codomain of the universal quotient $p: X \rightarrow X/E$.
\end{definition}

Universal quotients are essentially those by right congruences: 

\begin{lemma}
\label{universal quotient is by right congruence}
 Let $E \subseteq \End(X) \times \End(X)$ be any relation; and $\hat{E} \subseteq \End(X) \times \End(X)$ be the smallest right congruence on the monoid $\End(X)$ containing $E$. Then $p_E$ and $p_{\hat{E}}$ are isomorphic. In other words, every universal quotient is the one by a right congruence on $\End(X)$. 
\end{lemma}
\begin{proof}
 By definition, it is clear that $p_{\hat{E}}$ factors through $p_E$. To see the converse, let $E' := \{ (h,k) \in \End(X) \times \End(X) \mid p_E \circ h= p_E \circ k \}$ that is clearly a right congruence containing $E$. Thus, $\hat{E} \subseteq E'$. This means that $p_E$ coequalizes $\hat{E}$. Hence, $p_E$ factors through $p_{\hat{E}}$. 
\end{proof}

\begin{lemma}
\label{existence of universal quotients}
 A semi-galois category $\C$ has arbitrary universal quotients. 
\end{lemma}
\begin{proof}
 Let $X \in \C$ and $E \subseteq \End(X) \times \End(X)$, for which we construct a universal quotient $X \twoheadrightarrow X/E$. For each member $(h,k) \in E$, consider its coequalizer $q_{(h,k)}: X \twoheadrightarrow X_{(h,k)}$. Then one can take in $\C$ a finite pushout diagram for the family $\{ q_{(h,k)} : X \twoheadrightarrow X_{(h,k)} \}_{(h,k) \in E}$. 
\begin{equation}
\label{pushout diagram constructing universal quotient}
 \xymatrix{
  X \ar@{->>}[r]^{q_{(h,k)}} \ar@{->>}[d]_{q_{(h',k')}} \ar@{..>>}[rd]^{\pi} &  X_{(h,k)} \ar[d]^{l_{(h,k)}} \\
  X_{(h',k')} \ar[r]_{l_{(h',k')}}  &  W 
}
\end{equation}
Note that every $l_{(h,k)}$ is an epimorphism; and thus we can define an epimorphism $\pi: X \twoheadrightarrow W$ so that the above diagram commutes. 

We prove that $\pi$ is a universal quotient of $X$ by $E$. By definition, it is obvious that $\pi$ coequalizes $E$. Let $g: X \rightarrow Z$ coequalize $E$. Then, $g: X \rightarrow Z$ factors through the coequalizer $q_{(h,k)}: X \twoheadrightarrow X_{(h,k)}$ for each $(h,k) \in E$: that is, there exists $g_{(h,k)}: X_{(h,k)} \rightarrow Z$ such that $g= g_{(h,k)} \circ q_{(h,k)} $ for each $(h,k) \in E$. Since the diagram (\ref{pushout diagram constructing universal quotient}) is a pushout diagram, one obtains a unique arrow $g':W \rightarrow Z$ such that $g_{(h,k)} = g' \circ l_{(h,k)}$ for each $(h,k) \in E$. Then, one can see that $g=g' \circ \pi$. Now, since $\pi$ is an epimorphism, such a $g'$ is unique, which concludes that $\pi$ is indeed a universal quotient of $X$ by $E$. 
\end{proof}

Note that $\fsets$ together with the identity functor $\id_{\fsets}: \fsets \rightarrow \fsets$ is also a semi-galois category; thus, in the following, saying ``universal quotients in $\fsets$'' also make sense as well. 

\begin{corollary}
\label{preserving universal quotients}
 The fibre functor $\F: \C \rightarrow \fsets$ preserves universal quotients. That is, if $p:X \twoheadrightarrow X/E$ is a universal quotient of $X \in \C$ by $E\subseteq \End(X) \times \End(X)$, then the corresponding arrow $p_*: \F(X) \twoheadrightarrow \F(X/E)$ in $\fsets$ is a universal quotient by $\F(E) := \{(h_*,k_*) \mid (h,k) \in E\} \subseteq \End(\F(X)) \times \End(\F(X))$. 
\end{corollary}
\begin{proof}
 Immediate from the above construction of universal quotients and the fact that $\F: \C \rightarrow \fsets$ preserves finite colimits. 
\end{proof}

\begin{remark}
\label{on universal quotients}
More explicitly, the set $\F(X/E)$ can be described as the quotient $\F(X)/ \equiv$ of the set $\F(X)$ by the smallest equivalence relation $\equiv$ on $\F(X)$ that contains the following relation: 
\begin{eqnarray*}
 \{(h_*\xi, k_*\xi) \in \F(X) \times \F(X) \mid (h,k) \in E \hspace{0.1cm} \wedge \hspace{0.1cm} \xi \in \F(X) \}. 
\end{eqnarray*}
This remark will be used later (\S \ref{appendix b two}). 
\end{remark}

\subsection{Galois objects}
\label{s3s2}
\noindent

\begin{definition}[Covering]
 We say a \emph{(finite) covering} of an object $Y$ to mean a (finite) family $\{ \iota_i: Y_i \hookrightarrow Y\}_{i=1}^n$ of monomorphisms $\iota_i:Y_i \hookrightarrow Y$ ($1\leq i \leq n$) such that the induced morphism $\coprod_i \iota_i: \coprod_i Y_i \rightarrow Y$ is an epimorphism. 
\end{definition}

\begin{definition}[Rooted object]
 An object $X \in \C$ is called a \emph{rooted object} if, for every object $Y \in \C$, an arrow $f: X \rightarrow Y$ in $\C$ and a covering $\{\iota_i: Y_i\hookrightarrow Y\}_{i=1}^n$ of $Y$, the arrow $f$ factors through some component $\iota_k:Y_k \hookrightarrow Y$ of the covering:
\begin{eqnarray*}
  \xymatrix{
         &  Y_k \ar@<-0.15em>@{^{(}->}[d]^{\iota_k} \\
    X \ar[r]_f \ar@{..>}[ur]^{\exists f_k}& Y
}
\end{eqnarray*}
That is, there is an arrow $f_k: X \rightarrow Y_k$ such that $f= \iota_k \circ f_k$. 
\end{definition}

\begin{proposition}
\label{images of rooted objects}
 An object $X \in \C$ is rooted if and only if, for every covering $\{\iota_i: X_i \hookrightarrow X\}_{i=1}^n$ of $X$, there exists an index $i$ such that $\iota_i: X_i \hookrightarrow X$ is in fact an isomorphism. 
\end{proposition}
\begin{proof}
  The only-if part is trivial from the definition. Conversely, let $X$ be an object satisfying the condition, $f: X \rightarrow Y$ be an arrow, and $\{\iota_k: Y_k \hookrightarrow Y\}$ be a covering of $Y$. It suffices to show that $f$ factors through some of $\iota_k$. For this aim, consider the pullback $\{f^*\iota_k: f^*Y_k \hookrightarrow X\}$, which is easily proved to be a covering of $X$. By the assumption, there is an index $k$ such that $f^*\iota_k: f^*Y_k \rightarrow X$ is an isomorphim, whence $f$ factors through $\iota_k: Y_k \hookrightarrow Y$. This proves that $X$ is rooted. 
\end{proof}

\begin{proposition}
 Let $X$ be a rooted object and $f: X \twoheadrightarrow Y$ be an epimorphism. Then $Y$ is also rooted. 
\end{proposition}
\begin{proof}
Let $\{\iota_k: Y_k \hookrightarrow Y\}$ be a covering of $Y$. Since $X$ is now rooted, $f$ factors through some component $\iota_k: Y_k \hookrightarrow Y$.
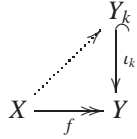
\begin{eqnarray*}
 \xymatrix{
   & Y_k \ar@<-0.15em>@{^{(}->}[d]^{\iota_k} \\  
 X \ar@{..>}[ru] \ar@{->>}[r]_f & Y
}
\end{eqnarray*}
Since $f$ is an epimorphism, so is $\iota_k: Y_k \hookrightarrow Y$. Thus, $\iota_k$ is an isomorphism. 
\end{proof}

\begin{example}[Rooted $M$-set]
 Let $M$ be a profinite monoid. A finite $M$-set $X \in \Cl_f M$ is rooted if and only if there exists $\xi \in S_X$ such that every $\xi' \in S_X$ can be expressed as $\xi' = \xi \cdot m$ for some $m \in M$. 
\end{example}



\noindent
Abstractly, let $X \in \C$ be an arbitrary object. Since $\C$ has pullbacks, we can take intersections $X_1 \cap X_2 \hookrightarrow X$ of two monomorphisms $X_1 \hookrightarrow X$ and $X_2 \hookrightarrow X$. Since $\F(X)$ is a finite set, there exists a (unique) minimal subobject, denoting $X_{\xi} \hookrightarrow X$, for every element $\xi \in \F(X)$ such that $\xi \in \F(X_{\xi})$. By definition, it is not difficult to see that $X_{\xi}$ is rooted. In fact, one can also show the converse.
\begin{proposition}
\label{characterization of rooted objects}
 An object $X$ is rooted if and only if $X=X_{\xi}$ for some $\xi \in \F(X)$.
\end{proposition}
\begin{proof}
 First, let $X$ be rooted. Consider the family $\{X_{\xi} \hookrightarrow X\}_{\xi \in \F(X)}$, which is clearly a covering of $X$. By assumption, there is some component $X_{\xi} \hookrightarrow X$ that is an isomorphism. Thus $X=X_{\xi}$. Conversely, let $X=X_{\xi}$ and $\{Z_i \hookrightarrow X\}_{i=1}^n$ be a covering of $X$. Then there is an index $i$ such that $\xi \in \F(Z_i)$, whence $X=X_{\xi} \subseteq Z_i \subseteq X$. Thus $Z_i \hookrightarrow X$ is an isomorphism, which concludes that $X$ is rooted. 
\end{proof}

\begin{proposition}
\label{maps from rooted objects}
 Let $X \in \C$ be rooted and $\xi \in \F(X)$ be such that $X=X_{\xi}$. Then, for any $Y \in \C$, the map $\omega_{X,\xi}:\Hom(X,Y) \ni f \mapsto f_* \xi \in \F(Y)$ is injective.
\end{proposition}
\begin{proof}
Let $f, g: X \rightarrow Y$ be arrows such that $f_* \xi = g_* \xi$ and $m: Z \hookrightarrow X$ be the equalizer of $f$ and $g$. Then $\xi \in \F(Z)$. Let $X_{\xi}$ be the minimal subobject of $X$ containing $\xi$. By assumption $X=X_{\xi}$, and thus $X=X_{\xi} \subseteq Z \subseteq X$: i.e.\ $X=Z$. This proves that $f=g$. 
\end{proof}

\begin{definition}[Galois object]
 A pair $(X,\xi)$ of an object $X \in \C$ and $\xi \in \F(X)$ is said to be a \emph{galois object} if $X=X_{\xi}$ and the map $\omega_{X,\xi}: \End (X) \ni f \mapsto f_* \xi \in \F(X)$ is bijective. 
\end{definition}

\begin{remark}
 If $\sgc$ is a galois category, the concept of galois objects defined here coincides with the original one. This is because the rooted objects in a galois category $\C$ are exactly the connected objects; see e.g.\ \cite{Tonini}.
\end{remark}



\begin{proposition}
\label{cofinality of galois objects}
 For an arbitrary object $Y \in \C$, there exists a galois object $(X, \xi)$ such that $\Hom(X,Y) \ni \sigma \mapsto \sigma_* \xi\in \F(Y)$ is a bijection. 
\end{proposition}
\begin{proof}
  Put $\F(Y) = \{\eta_1,\cdots, \eta_n\}$ and let $\alpha: X \hookrightarrow Y^n$ be the minimal subobject containing $\xi = (\eta_1, \cdots, \eta_n) \in \F(Y^n)$. We first show that $(X, \xi)$ is a galois object. It is sufficient to prove that, for any $\gamma \in \F(X)$, there exists a unique $g\in \End(X)$ such that $g_* \xi = \gamma$. Let $\alpha_* \gamma = (\eta_{\tau(1)},\cdots, \eta_{\tau(n)})$ where $\tau$ is a transformation on $\{1,\cdots,n\}$. Define $\hat{\tau}: Y^n \rightarrow Y^n$ to be the endomorphism of $Y^n$ naturally induced from $\tau$. Then $ \hat{\tau}_*  \alpha_* \xi = \alpha_* \gamma$. Consider the pullback $\hat{\tau}^* X$ of $\alpha: X \hookrightarrow Y$ along $\hat{\tau}: Y^n \rightarrow Y^n$; then $\xi \in \F(\hat{\tau}^* X)$. By the minimality of $X$, this implies that $X \hookrightarrow \hat{\tau}^* X$; using this inclusion together with the pullback diagram for $\hat{\tau}^* X$, one can construct an endomorphism $g:X \rightarrow X$ such that the following diagram commutes:
\begin{equation}
   \xymatrix{ Y^n  \ar[r]^{\hat{\tau}} &  Y^n \\
              X \ar[u]^{\alpha} \ar@{..>}[r]_{\exists g} & X \ar[u]_{\alpha} 
            }
\end{equation}
Thus we have $\alpha_* g_* \xi = \hat{\tau}_*\alpha_*\xi = \alpha_* \gamma$. By the injectivity of $\alpha_*$, we consequently obtain $g_* \xi = \gamma$. This shows that the map $\omega_{X, \xi}: \End(X) \ni g \mapsto g_* \xi \in \F(X)$ is surjective. By Proposition \ref{maps from rooted objects}, it follows that $\omega_{X,\xi}$ is also injective and thus $(X, \xi)$ is a galois object. Finally, let $\pi_i: Y^n \rightarrow Y$ be the projection onto the $i$-th component and $\sigma^i$ be the composition $\sigma^i= \pi_i \alpha$. Then clearly $\sigma^i_* \xi = \eta_i$ for each $1 \leq i \leq n$. This proves that $\Hom(X,Y) \ni \sigma \mapsto  \sigma_* \xi \in \F(Y)$ is surjective. Again, by Proposition \ref{maps from rooted objects}, it is also injective and thus a bijection. 
\end{proof}

One can see that each object $X \in \C$ can be decomposed into a disjoint union $X = X_1 \sqcup X_2 \sqcup \cdots \sqcup X_n$ of connected components $X_i \hookrightarrow X$; and in the case where $\C$ is galois, this decomposition was used (e.g.\ in the proof that every galois category $\sgc$ is equivalent to one of the form $\langle \Cl_f G, \F_G \rangle$, cf.\ \cite{Tonini}) as a step of reducing several arguments on objects to the case for connected objects. In the case where $\C$ is not galois, however, we need a more subtle decomposition of objects in order to perform the same reduction (\S \ref{s4}). The concepts of \emph{optimal covering} and \emph{root degree} of objects are used for this purpose; and formally, defined as follows. 

\begin{definition}[Optimal covering and root degree]
 A covering $\{j_k: X_k \hookrightarrow X\}$ of $X$ is said to be \emph{optimal} if:
  \begin{enumerate}
    \item each $X_k$ is rooted;
    \item if $j_k$ can factor through $j_l$, then $k=l$. 
  \end{enumerate}
Moreover, the \emph{root degree} of an object $X \in \C$ (or simply, the \emph{degree}) is defined as the size $n$ of an optimal covering $\{\lambda_i: X_i \hookrightarrow X\}_{i=1}^n$ of $X$. 
\end{definition}

These concepts are general enough and well-defined in the following sense: 

\begin{proposition}
  Every object $X \in \C$ has an optimal covering, which is unique up to isomorphism. 
\end{proposition}
\begin{proof}
 First, the existence of optimal covering is proved by induction on $|\F(X)|$. For the base step, assume that $|\F(X)|=1$. Then the identity $\{\id_X: X \rightarrow X\}$ is an optimal covering of $X$. For the induction step, let $X$ be such that $|\F(X)| = n+1$. If $X$ is rooted, then $\id_X: X \rightarrow X$ is an optimal covering of $X$. Otherwise, there is a covering $\{j_k: X_k \hookrightarrow X\}$ such that no $j_k: X_k \hookrightarrow X$ is an isomorphism, i.e.\ $|\F(X_k)| < n+1$ in particular. By induction hypothesis, each $X_k$ has an optimal covering. Using them, it is not difficult to construct an optimal covering of $X$. 

 Finally, for the proof of uniqueness, let $\{\lambda_i: X_i \hookrightarrow X\}_{i=1}^n$ and $\{\lambda_k': X_k' \hookrightarrow X \}_{k=1}^m$ be optimal coverings of $X$. Since $X_i$ is rooted and $\{ \lambda_k': X_k' \hookrightarrow X \}$ is a covering of $X$, $\lambda_i: X_i \hookrightarrow X$ factors through some $\lambda_k': X_k' \hookrightarrow X$. One can see that such an index $k$ is unique for $i$. In fact, assume that $\lambda_i$ also factors through $\lambda_l'$. The inverse argument shows that $\lambda_k'$ factors through some $\lambda_j$, which implies that $\lambda_i$ factors through $\lambda_j$ and thus $j=i$ by the optimality of $\{\lambda_i\}$. From the fact that $\lambda_k'$ factors through $\lambda_i$ and $\lambda_i$ factors through $\lambda_l'$, $\lambda_k'$ factors through $\lambda_l'$ and thus $k=l$. Therefore we can define a map $\sigma: \{1,\cdots,n\} \rightarrow \{1,\cdots, m\}$ so that $\lambda_i$ factors through $\lambda_{\sigma(i)}'$. The above argument in fact proves that $\sigma$ is an isomorphism (i.e.\ $n=m$ in particular), and that each of $\lambda_i$ and $\lambda_{\sigma(i)}'$ factor through each other, i.e.\ isomorphic. 
\end{proof}

\begin{remark}
 If $\C$ is galois, an optimal covering $\{X_k \hookrightarrow X\}$ of an object $X$ is equal to one given by the decomposition $X= X_1 \sqcup X_2 \sqcup \cdots \sqcup X_n$ into connected components $X_i \hookrightarrow X$; in other words, components of optimal coverings $X_i$ are mutually disjoint in the case of galois categories. However, if $\C$ is not galois, components of an optimal covering are not necessarily disjoint. Thus we denote such a decomposition as $X = X_1 \cup X_2 \cup \cdots \cup X_n$ to indicate that the components $X_i$ may intersect. 
\end{remark}

\subsection{Fundamental monoids}
\label{s3s3}
\noindent
The monoid $\End(\F)$ of natural endomorphisms $\F \Rightarrow \F$ can be equipped with a canonical profinite topology so that $\End(\F)$ forms a profinite monoid. In this section, we construct a profinite monoid using an inverse system based on galois object; and prove that it is canonically isomorphic to $\End(\F)$. This representation of $\End(\F)$ will be used in the proof of the duality theorem between profinite monoids and semi-galois categories (\S \ref{s4}).

To construct a necessary inverse system, several lemmas are in order:
\begin{lemma}
\label{arrows between galois objects}
 Let $(X,\xi)$ and $(X',\xi')$ be galois objects. Then there is at most one arrow $\lambda: X \rightarrow X'$ such that $ \lambda_*\xi = \xi'$; and if exists, the arrow $\lambda$ is an epimorphism.
\end{lemma}
\begin{proof}
 The first claim follows from Proposition \ref{maps from rooted objects}. To prove the second claim, let $\lambda: X \rightarrow X'$ be such that $ \lambda_*\xi = \xi'$. Then, the image $\mathrm{Im}(\lambda) \subseteq X'$ contains $\xi'$. By the fact that $X'= X'_{\xi'}$, one obtains $\mathrm{Im}(\lambda) = X'$ because $X'=X'_{\xi'} \subseteq \mathrm{Im}(\lambda) \subseteq X'$. This shows that $\lambda$ is an epimorphism.
\end{proof}
\noindent
By this lemma, a pre-order $\leq$ on galois objects can be defined by $(X,\xi) \leq (X', \xi')$ if and only if there exists a \emph{pointed arrow} $\lambda: (X',\xi') \rightarrow (X,\xi)$, i.e.\ an arrow $\lambda: X' \rightarrow X$ such that $ \lambda_*\xi' = \xi$. In what follows, this pre-ordered set of galois objects is denoted $\G$ that turns out to be \emph{cofiltered} in the following sense:

\begin{lemma}
\label{cofilteredness of galois objects}
 For an arbitrary pair of galois objects $(X, \xi), (X',\xi') \in \G$, there is a galois object $(X'',\xi'') \in \G$ such that $(X,\xi) \leq (X'',\xi'')$ and $(X',\xi') \leq (X'',\xi'')$.
\end{lemma}
\begin{proof}
 Let $(X,\xi)$ and $(X',\xi') \in \G$ be galois objects. By Proposition \ref{cofinality of galois objects}, there are a galois object $(X'',\xi'') \in \G$ and $\sigma: X'' \rightarrow X\times X'$ such that $ \sigma_*\xi'' = (\xi,\xi') \in \F(X \times X')$. Let $\pi:X\times X' \rightarrow X$ and $\pi': X \times X' \rightarrow X'$ be canonical projections. Then it is easy to see that $\pi \sigma: X'' \rightarrow X$ and $\pi' \sigma: X'' \rightarrow X'$ give the pre-ordering $(X,\xi) \leq (X'',\xi'')$ and $(X',\xi') \leq (X'',\xi'')$.
\end{proof}

\noindent
The cofiltered pre-ordered set $\G$ is used below as the index set of an inverse system of finite monoids whose limit is isomorphic to $\End(\F)$ as profinite monoids. The inverse system consists of finite monoids $\End(X)$ where $X$ ranges over all galois objects $(X,\xi) \in \G$. More precisely, the construction of the inverse system is based on the following two lemmas:

\begin{lemma}
\label{endomorphisms of galois objects}
 Let $(X, \xi)$ and $(X', \xi')$ be galois objects with $\lambda: (X,\xi) \rightarrow (X', \xi')$. Then, for every $u \in \End(X)$, there exists a unique $u' \in \End(X')$ such that the following diagram commutes:
\begin{equation*}
  \xymatrix{
    X \ar[r]^u  \ar@{->>}[d]_{\lambda} & X \ar@{->>}[d]^{\lambda} \\
    X' \ar@{..>}[r]_{\exists ! u'} & X'
}
\end{equation*}
\end{lemma}
\begin{proof}
 The uniqueness follows from the fact that the pointed arrow $\lambda:(X,\xi) \rightarrow (X',\xi')$ between galois objects must be an epimorphism (cf.\ Lemma \ref{arrows between galois objects}). For the existence, consider the element $\lambda_* u_* \xi \in \F(X')$. By the bijectivity of $\End(X') \ni v \mapsto v_* \xi' \in \F(X')$, there exists $u' \in \End(X')$ such that $\lambda_* u_* \xi  = u'_* \xi'$, whence $ \lambda_* u_* \xi = u'_* \lambda_* \xi$. By Proposition \ref{maps from rooted objects}, we obtain $\lambda u = u' \lambda$.  
\end{proof}

\begin{lemma}
 Let $(X,\xi)$ and $(X',\xi')$ be galois objects with $\lambda: (X,\xi) \rightarrow (X',\xi')$. Then the above correspondence $\End(X) \ni u \mapsto u' \in \End(X')$ is a surjective homomorphism of finite monoids. 
\end{lemma}
\begin{proof}
 The uniqueness of $u'$ shown in Lemma \ref{endomorphisms of galois objects} proves that the correspondence $ \End(X) \ni u \mapsto u' \in \End(X')$ is a monoid homomorphism. To see the surjectivity, it is sufficient to show that, for every $u' \in \End(X')$, there exists some $u\in \End(X)$ satisfying $\lambda u = u' \lambda$. For this aim, let $u' \in \End(X)$ be an arbitrary endomorphism on $X'$, and consider the element $u'_* \xi' \in \F(X')$. By the surjectivity of $\lambda_* : \F(X) \twoheadrightarrow \F(X')$, there exists an element $\xi_1 \in \F(X)$ such that $\lambda_* \xi_1 = u'_* \xi'$. Also, by the bijectivity of $\End(X) \ni v \mapsto v_* \xi \in \F(X)$, there is $u_1 \in \End(X)$ such that $\xi_1 = u_{1*}\xi$. Then, $ \lambda_* u_{1*} \xi =  u'_* \lambda_*\xi$. Again, by Proposition \ref{maps from rooted objects}, one obtains $\lambda u_1 = u' \lambda$. This proves the claim.  
\end{proof}

\noindent
For simplicity, we reserve capital Greeks $\Gamma, \Gamma' \cdots$ to name galois objects, e.g.\ $\Gamma = (X,\xi)$. Also, if we denote $\End(\Gamma)$ for $\Gamma = (X, \xi) \in \G$, we mean $\End(\Gamma) := \End(X)$. For galois objects $\Gamma=(X,\xi)$ and $\Gamma'=(X',\xi')$ with $\Gamma' \leq \Gamma$, the induced surjective homomorphism $\End(\Gamma) \twoheadrightarrow \End(\Gamma')$ is denoted $\rho^{\Gamma}_{\Gamma'}$. 
\begin{definition}
 The correspondence $\G \ni \Gamma \mapsto \End(\Gamma) \in {\bf Fin.Mon}$ defines an inverse system of finite monoids, where ${\bf Fin.Mon}$ denotes the category of finite monoids and homomorphisms. 
\end{definition}

\noindent
The limit of this inverse system $\{\End(\Gamma)\}_{\Gamma \in \G}$ is denoted $\displaystyle \lim_{\leftarrow \Gamma} \End(\Gamma)$ or simply $\lim_{\Gamma} \End(\Gamma)$. More explicitly, $\lim_{\Gamma} \End(\Gamma)$ consists of families ${\bf u}=\{ u_{\Gamma}\}_{\Gamma \in \G}$ of endomorphisms $u_{\Gamma} \in \End(\Gamma)$ such that $\rho^{\Gamma}_{\Gamma'} (u_{\Gamma}) = u_{\Gamma'}$ for all $\Gamma' \leq \Gamma \in \G$. The multiplication of ${\bf u} = \{u_{\Gamma}\}_{\Gamma \in \G}$ and ${\bf v} = \{v_{\Gamma}\}_{\Gamma \in \G}$ is given by:
\begin{eqnarray}
  {\bf uv} &=& \{ u_{\Gamma} v_{\Gamma} \}_{\Gamma \in G}.
\end{eqnarray}
The rest of this subsection is devoted to a proof of the following theorem:

\begin{theorem}
\label{delta is isomorphism}
 There exists a canonical isomorphism of profinite monoids:
 \begin{equation}
   \delta : \lim_{\leftarrow \Gamma} \End(\Gamma) \rightarrow \End(\F). 
 \end{equation}
\end{theorem}
\begin{proof}
 Let $\uu = \{u_{\Gamma}\} \in \limmon$, from which we would like to define a natural endomorphism $\delta \uu: \F \Rightarrow \F \in \End(\F)$. For this aim, we first define a map $\delta \uu _Y: \F(Y) \rightarrow \F(Y)$ for each object $Y \in \C$; then, we will show that the family $\delta \uu = \{\delta \uu _Y\}$ forms a natural endomorphism on $\F$. Take $Y \in \C$ arbitrarily and let $\eta \in \F(Y)$. Then, by Proposition \ref{cofinality of galois objects}, there exists an arrow $\sigma: X \rightarrow Y$ from a galois object $\Gamma = (X,\xi)$ such that $\sigma_* \xi = \eta$. We define $\eta \cdot \delta \uu _Y \in \F(Y)$ as follows. (Recall Notation \ref{composition of natural transformation}, \S \ref{s2s1} for the notation of the action of $\delta \uu: \F \Rightarrow \F$ on $\eta \in \F(Y)$ at $Y \in \C$.)
\begin{eqnarray}
 \eta \cdot \delta \uu_Y &:=& \sigma_* u_{\Gamma *} \xi. 
\end{eqnarray}
This is well-defined. That is:
\begin{lemma}
 The value of $\eta \cdot \delta \uu_Y \in \F(Y)$ does not depend on the choice of the galois object $\Gamma = (X,\xi)$ and the arrow $\sigma: X \rightarrow Y$ with $\sigma_* \xi = \eta$. 
\end{lemma}
\begin{proof}
 Let $\Gamma' = (X',\xi')$ and $\sigma': X' \rightarrow Y$ be another choice. We need to prove the following equality:\begin{eqnarray} 
\label{well-definedness of delta}
\sigma_* u_{\Gamma *} \xi &=&  \sigma'_* u_{\Gamma' *} \xi'. 
\end{eqnarray}
By Lemma \ref{cofilteredness of galois objects}, there exists a galois object $\Gamma''= (X'',\xi'')$ that has pointed arrows $\lambda: (X'',\xi'') \rightarrow (X,\xi)$ and $\lambda': (X'',\xi'') \rightarrow (X',\xi')$. Then, we have $\xi= \lambda_* \xi''$ and $\xi'= \lambda_*\xi''$. So the required equation (\ref{well-definedness of delta}) is reduced to the following equality: 
\begin{eqnarray} 
\label{reduced well-definedness of delta}
 \sigma_* u_{\Gamma *} \lambda_* \xi'' &=&  \sigma'_* u_{\Gamma' *} \lambda'_* \xi''. 
\end{eqnarray}
Since $\uu \in \lim_{\Gamma} \End(\Gamma)$, we have $\rho^{\Gamma''}_{\Gamma} (u_{\Gamma''}) = u_{\Gamma}$ and $\rho^{\Gamma''}_{\Gamma'} (u_{\Gamma''}) = u_{\Gamma'}$. By the definition of $\rho^{\Gamma}_{\Gamma'}$, this means that $ \lambda u_{\Gamma''} = u_{\Gamma} \lambda$ and $ \lambda' u_{\Gamma''} = u_{\Gamma'} \lambda'$. Thus, the left and right hand sides of (\ref{reduced well-definedness of delta}) are equal to $ \sigma_* \lambda_* u_{\Gamma'' *} \xi''$ and $ \sigma'_* \lambda'_* u_{\Gamma'' *}\xi''$ respectively. Finally, notice that $ \sigma \lambda= \sigma' \lambda'$ because $ \sigma_* \lambda_*\xi'' = \sigma_*\xi = \eta = \sigma'_* \xi' = \sigma'_* \lambda'_* \xi''$. Thus:
\begin{eqnarray*} 
 \sigma_* u_{\Gamma *} \lambda_*\xi'' &=& \underline{\sigma_*\lambda_*} u_{\Gamma'' *} \xi'' \\
&=& \underline{\sigma'_*\lambda'_* } u_{\Gamma'' *} \xi'' \\
&=&  \sigma'_* u_{\Gamma' *}\lambda'_*  \xi''. 
\end{eqnarray*}
This proves the equation (\ref{reduced well-definedness of delta}). 
\end{proof}

\begin{lemma}
\label{naturality of delta}
 The family $\delta \uu = \{ \delta \uu_Y \} _{Y\in \C}$ is a natural endomorphism $\F \Rightarrow \F$. 
\end{lemma}
\begin{proof}
 Let $Y, Z$ be arbitrary objects in $\C$ and $f: Y \rightarrow Z$ be an arrow. We need to see that $(f_* \eta) \cdot\delta \uu_Z = f_* (\eta \cdot \delta \uu_Y)$. Let $(X,\xi)$ be a galois object with $\sigma: X \rightarrow Y$ such that $\sigma_*\xi = \eta$, whence $f_* \sigma_* \xi = f_* \eta$. Thus, by the definition of $\delta \uu_Z$ and $\delta \uu_Y$:
\begin{eqnarray*}
  (f_* \eta) \cdot \delta \uu_Z &=& (f \sigma)_* u_{\Gamma *} \xi\\
  &=&  f_* \underline{ \sigma_* u_{\Gamma *}\xi}\\
  &=& f_* \underline{(\eta \cdot \delta \uu_Y)}.
\end{eqnarray*}
This proves the naturality of $\delta \uu$. 
\end{proof}
\noindent
By these lemmas, we obtain a map $\delta: \lim_{\Gamma} \End(\Gamma) \rightarrow \End(\F)$. In fact, $\delta$ is not only a map but also a homomorphism of monoids. 
\begin{lemma}
\label{delta is homomorphism}
 The map $\delta: \lim_{\Gamma} \End(\Gamma) \rightarrow \End(\F)$ is a homomorphism. 
\end{lemma}
\begin{proof}
 It is easy to see that $\delta$ preserves the identity; so we show that $\delta$ preserves multiplication. Let $\uu = \{ u_{\Gamma} \}, \uu' = \{ u'_{\Gamma}\} \in \lim_{\Gamma} \End(\Gamma)$. We want to show that $\delta (\uu' \uu) = \delta (\uu') \cdot \delta (\uu)$, which is done by object-wise argument. Let $Y$ be an arbitrary object in $\C$ and $\eta \in \F(Y)$. Also, choose a galois object $\Gamma=(X,\xi)$ and an arrow $\sigma: X \rightarrow Y$ with $\sigma_* \xi = \eta$. By definition of $\delta \uu'$, we have:
\begin{eqnarray}
  \eta \cdot \delta \uu'_Y &=& \sigma_* u'_{\Gamma *} \xi,
\end{eqnarray}
which we denote $\eta' \in \F(Y)$. Since $\eta' = (\sigma u'_{\Gamma})_* \xi$ and $\sigma u'_{\Gamma} : X \rightarrow Y$ is an arrow from the galois object $(X,\xi)$, the definition of $\delta \uu_Y$ implies:
\begin{eqnarray*}
   \eta' \cdot \delta \uu_Y  &=& (\sigma u'_{\Gamma })_*  u_{\Gamma *} \xi\\
   &=& \sigma_* (u'_{\Gamma }  u_{\Gamma *})_*  \xi.
\end{eqnarray*}
Since $\uu' \uu = \{u'_{\Gamma} u_{\Gamma}\}$ and $\eta = \sigma_* \xi$, we obtain:
\begin{eqnarray*}
   \eta \cdot (\delta \uu'_Y \cdot \delta \uu_Y)  &=& \eta' \cdot \delta \uu_Y \\
   &=& \sigma_* (u'_{\Gamma }  u_{\Gamma *})_*  \xi\\
   &=& \eta \cdot \delta (\uu' \uu)_Y . 
\end{eqnarray*}
Since we took $Y\in \C$ and $\eta \in \F(Y)$ arbitrarily, this proves $\delta (\uu' \uu) = \delta \uu' \cdot \delta \uu$.
\end{proof}
\noindent
Finally, we prove that the homomorphism $\delta: \lim_{\Gamma} \End(\Gamma) \rightarrow \End(\F)$ is an isomorphism. First we show the injectivity of $\delta$; and then the surjectivity of $\delta$. 
\begin{lemma}
\label{injectivity of delta}
 The homomorphism $\delta: \lim_{\Gamma} \End(\Gamma) \rightarrow \End(\F)$ is injective. 
\end{lemma}
\begin{proof}
Let $\uu = \{ u_{\Gamma}\}, \uu' = \{ u'_{\Gamma}\} \in \limmon$ be such that $\delta \uu = \delta \uu'$. From this, we would like to show that $u_{\Gamma} = u'_{\Gamma}$ for every $\Gamma = (X, \xi) \in \G$. To this end, it is sufficient to see $ u_{\Gamma *}\xi = u'_{\Gamma *} \xi$ (cf.\ Proposition \ref{maps from rooted objects}). Since $\xi = \mathrm{id}_{X*} \xi$ and by definition of $\delta \uu, \delta \uu'$, we have $u_{\Gamma *}\xi =\xi \cdot \delta \uu_X $ and $u'_{\Gamma *} \xi = \xi \cdot  \delta \uu'_X$. But since we now have $\delta \uu = \delta \uu'$, this implies $ u_{\Gamma *}\xi = u'_{\Gamma *} \xi$ as requested. 
\end{proof}
\begin{lemma}
\label{surjectivity of delta}
 The homomorphism $\delta: \lim_{\Gamma} \End(\Gamma) \rightarrow \End(\F)$ is surjective.   
\end{lemma}
\begin{proof}
 Let $\phi: \F \Rightarrow \F \in \End(\F)$ be an arbitrary natural endomorphism on $\F$, from which we now construct $\uu = \{ u_{\Gamma}\}_{\Gamma \in \G} \in \limmon$ so that $\delta \uu = \phi$. For this aim, take $u_{\Gamma} \in \End(X)$ for each $\Gamma =(X, \xi) \in \G$ so that $u_{\Gamma *}\xi  = \xi \cdot \phi_X$. (Note that such $u_\Gamma$ exists because $(X,\xi)$ is galois.) We first see that $\uu := \{u_{\Gamma}\}_{\Gamma} \in \limmon$; and second, we see that $\delta \uu = \phi$. 

First, let $\Gamma = (X, \xi)$ and $\Gamma'=(X',\xi') \in \G$ be galois objects such that $\Gamma' \leq \Gamma$. Then, there exists an arrow $\lambda: X \rightarrow X'$ with $\lambda_* \xi = \xi'$. By the naturality of $\phi$, we have $  \lambda_* (\xi\cdot \phi_X) = ( \lambda_* \xi) \cdot \phi_{X'} $, i.e.\ $ \lambda_* u_{\Gamma *}\xi =   u_{\Gamma' *} \lambda_* \xi$ which proves $\lambda u_{\Gamma} = u_{\Gamma'} \lambda$. By definition of $\rho^{\Gamma}_{\Gamma'}$, this means that $\rho^{\Gamma}_{\Gamma'} (u_{\Gamma}) = u_{\Gamma'}$. In other words, $\uu = \{u_{\Gamma}\}_{\Gamma}$ is indeed a member of $\limmon$. Second, we see that $\delta \uu = \phi$, i.e.\ $\delta \uu_Y = \phi_Y$ for every object $Y \in \C$. Let $Y \in \C$, $\eta \in \F(Y)$, and $\sigma: X \rightarrow Y$ be an arrow from a galois object $\Gamma = (X, \xi)$ such that $\eta = \sigma_* \xi$. By definition of $\delta \uu$, we have:
\begin{eqnarray*}
  \eta \cdot \delta \uu_ Y  &=& \sigma_* u_{\Gamma *} \xi\\
  &=&   \sigma_* (\xi \cdot \phi_X).
\end{eqnarray*}
By the naturality of $\phi$, the right-most is equal to $( \sigma_* \xi) \cdot \phi_Y  = \eta \cdot \phi_Y $. From this, it follows that $\eta \cdot \delta \uu_Y  =\eta \cdot \phi_Y $. Thus, $\delta \uu_Y = \phi_Y$. 
\end{proof}
\noindent
Consequently, all these lemmas show that $\delta: \limmon \rightarrow \End(\F)$ is an isomorphism: the theorem is proved. 
\end{proof}

\begin{definition}[Fundamental monoids] 
 Let $\langle \C,\F \rangle$ be a semi-galois category. Then, the monoid $\End(\F)$ forms a profinite monoid with respect to the topology induced by the isomorphism $\delta: \limmon \rightarrow \End(\F)$. We call the profinite monoid $\End(\F)$ the \emph{fundamental monoid} of $\langle \C,\F \rangle$, and denote $\pi_1(\C,\F)$. 
\end{definition}

\section{The Duality Theorem}
\label{s4}
\noindent
This section proves the opposite equivalence between (i) the category of profinite monoids and continuous homomorphisms, denoted $\profmon$; and (ii) the category of semi-galois categories and exact functors, denoted $\semigalois$:
\begin{eqnarray*}
  \profmon &\simeq& \semigalois^{op}. 
\end{eqnarray*}
The proof consists of three parts: In the first part (\S \ref{appendix b one}), every profinite monoid $M$ is proved isomorphic to the fundamental monoid $\pi_1(\mfsets, \F_M)$ of the semi-galois category $\langle \mfsets, \F_M \rangle$ of finite $M$-sets and $M$-equivariant maps; in the second part (\S \ref{appendix b two}), proved is that every semi-galois category $\langle \C,\F \rangle$ is canonically equivalent to the semi-galois category $\langle \Cl_f M, \F_M \rangle$ with $M = \pi_1(\C,\F)$; and in the last part (\S \ref{appendix b three}), we show the target opposite equivalence $\profmon \simeq \semigalois^{op}$, using the results proved in the first two parts \S \ref{appendix b one} -- \S \ref{appendix b two}. 

\subsection{Reconstruction of profinite monoids}
\label{appendix b one}
\noindent
We start with the first part of the proof. Let $M$ be a profinite monoid and $\langle \mfsets, \F_M \rangle$ be the semi-galois category of finite $M$-sets and $M$-equivariant maps. The purpose of this section is to prove the following:
\begin{theorem}[The Reconstruction Theorem]
\label{reconstruction of profinite monoid}
 There is a canonical isomorphism of profinite monoids: 
\begin{equation}
\lambda_M: M \xrightarrow{\simeq} \pi_1(\mfsets,\F_M). 
\end{equation}
\end{theorem}

\noindent
Let $Y = \langle S,\rho \rangle \in \mfsets$. Then, by definition of finite $M$-sets, $\rho: M \rightarrow\End(S) = \End(\F_M(Y))$ is a monoid homomorphism; and for each $m \in M$, defines an element $\rho(m) \in \End(\F_M(Y))$, i.e.\ a map $\rho(m) : \F_M(Y) \rightarrow \F_M(Y)$. Consider the following family: 
\begin{eqnarray}
\label{family}
 \lambda_M(m) &:=& \bigl\{\rho(m): \F_M(Y) \rightarrow \F_M(Y)\bigr\}_{Y=\langle S,\rho \rangle \in \mfsets}.
\end{eqnarray}
Since $M$-equivariant maps (i.e.\ arrows in $\mfsets$) preserve $M$-actions (i.e.\ $\rho(m)$), the family (\ref{family}) defines a natural transformation $\F_M \Rightarrow \F_M$. Thus a map $\lambda_M: M \rightarrow \pi_1(\mfsets,\F_M)$ is defined by mapping $m\in M$ to $\lambda_M(m) \in \End(\F_M) = \pi_1(\mfsets,\F_M)$. We now show that $\lambda_M$ is a continuous isomorphism of profinite monoids. 

For this purpose, it is sufficient to prove that the following composition is an isomorphism:
\begin{equation}
\label{lambda delta}
   M \xrightarrow{\lambda_M} \End(\F_M) \xrightarrow{\delta^{-1}} \lim_{\leftarrow \Gamma} \End(\Gamma).
\end{equation}
Here $\Gamma$ ranges over all galois objects in $\mfsets$; and $\delta: \lim_{\Gamma} \End(\Gamma) \rightarrow \End(\F_M)$ is the isomorphism that we constructed in the previous section (cf.\ Theorem \ref{delta is isomorphism}). As shown in Lemma \ref{surjectivity of delta}, the inverse $\delta^{-1}: \End(\F_M) \rightarrow \lim_{\Gamma} \End(\Gamma)$ assigns to each $\phi \in \End(\F_M)$ the family $\{ u_{\Gamma}\} \in \lim_{\Gamma} \End(\Gamma)$ such that $\xi \cdot\phi_X = u_{\Gamma *} \xi$ for each galois object $\Gamma = (X,\xi)$ in $\mfsets$. This means that the composition (\ref{lambda delta}) assigns to each $m \in M$ the family $\{u^m_{\Gamma}\} \in \lim_{\Gamma} \End(\Gamma)$ such that $\xi \cdot m = u^m_{\Gamma *} \xi$ for each $\Gamma = (X,\xi)$. 

To prove that (\ref{lambda delta}) is an isomorphism, we first characterize galois objects in $\mfsets$. 

\begin{lemma}
\label{characterization of galois objects}
 Let $\Gamma = (X,\xi)$ be a galois object in $\mfsets$. Then, the equivalence $m \equiv_{\Gamma} m'$ given by $\xi \cdot m = \xi \cdot m'$ is a congruence on $M$ of finite index. 
\end{lemma}
\begin{proof}
 The relation $\equiv_\Gamma$ is obviously a right congruence of finite index. To see that $\equiv_\Gamma$ is also a left congruence, let $m \equiv_\Gamma m$ and $n \in M$; we show that $n m \equiv_\Gamma n m'$, i.e.\ $\xi \cdot nm = \xi \cdot nm'$. Since $\Gamma=(X,\xi)$ is a galois object, there exists $v \in \End(\Gamma)$ such that $\xi \cdot n = v_* \xi$. Thus:
\begin{eqnarray*}
   \xi \cdot nm &=& (v_* \xi) \cdot m \\
                &=& v_* (\xi \cdot m) \hspace{0.5cm} (\because \textrm{$v$ is an $M$-equivariant map.})\\
                &=& v_* (\xi \cdot m') \\ 
                &=& (v_* \xi) \cdot m \\
                &=& \xi \cdot nm'.
\end{eqnarray*}
Thus, $nm \equiv_\Gamma nm'$. 
\end{proof}

\noindent
Hence, the quotient $H_\Gamma := M / \equiv_\Gamma$ forms a finite monoid; and, by the continuity of $M$-action on the set $\F_M(X)$, the canonical projection $p_\Gamma: M \twoheadrightarrow H_\Gamma$ is continuous. 

In general, each finite quotient $p: M \twoheadrightarrow H$ of $M$ onto a finite monoid $H$ induces a finite $M$-set $X_H = \langle S_H, \rho_H \rangle \in \mfsets$, which is formally defined as $\F_M(X_H) = S_H := H$ and $\rho_H: S_H \times M \rightarrow S_H$ is given by, for each $m \in M$ and $h \in H$,
\begin{eqnarray}
\label{definition of action}
 \rho_H (h,m)  &:=& h \cdot p(m).
\end{eqnarray}
Let $e_H \in H = \F_M (X_H)$ denote the identity of $H$. 

\begin{lemma}
\label{galois objects in mfsets}
  $\Gamma_H = (X_H, e_H)$ is a galois object in $\mfsets$.
\end{lemma}
\begin{proof}
  First, $X_H$ is rooted with $e_H \in \F_M(X_H)$ its root because every $h \in \F_M(X_H)$ is of the form $h= p(m)$ for some $m \in M$; and so, $h = e_H \cdot m$ by (\ref{definition of action}). Second, we see that $\omega: \End(X_H) \ni u \mapsto  u_* e_H \in \F_M (X_H)$ is bijective. Since $X_H$ is rooted, the map $\omega$ is injective. To see the surjectivity, let $h \in H = \F_M(X_H)$ be an arbitrary element. Notice that the left action $H \ni k \mapsto hk \in H$ defines an $M$-equivariant map $u_h: X_H \rightarrow X_H$, i.e.\ commutes with right $M$-action (\ref{definition of action}); and also, $ u_{h*} e_H = h $. This shows that $\omega$ is indeed bijective; and thus, $\Gamma_H = ( X_H,e_H )$ is a galois object in $\mfsets$. 
\end{proof}

\begin{lemma}
 Let $\Gamma = (X, \xi)$ be a galois object in $\mfsets$ and $p_\Gamma: M \twoheadrightarrow H_\Gamma$ be the finite quotient defined above. Then, $\Gamma$ and $\Gamma_{H_\Gamma}$ are isomorphic as $M$-sets. In other words, galois objects in $\mfsets$ are exactly those of the form $\Gamma_H$ for some finite quotient $M \twoheadrightarrow H$. 
\end{lemma}
\begin{proof}
 Define an arrow $X_{H_\Gamma} \xrightarrow{f} X$ by $\F_M(X_{H_\Gamma}) \ni [m] \mapsto m \cdot \xi \in \F_M(X)$, where $[m]$ denotes the $\equiv_\Gamma$-equivalence class containing $m \in M$. By definition of $\equiv_\Gamma$, this arrow is well-defined; and bijective by that $\xi$ is a root of $X$. Also, by definition of $\Gamma_{H_\Gamma}$, it is easy to see that $f$ preserves $M$-actions and maps the root $e_H = [e_M]$ to $\xi$.  
\end{proof}

\begin{lemma}
\label{h is isomorphic to end}
 $H \simeq \End(\Gamma_H)$. 
\end{lemma}
\begin{proof}
 The above construction $h \mapsto u_h$ of Lemma \ref{galois objects in mfsets} defines a homomorphism $H \rightarrow \End(\Gamma_H)$. The injectivity is clear. To see the surjectivity, let $u \in \End(\Gamma_H)$ and $h := u_* e_H$. Since $e_H$ is a root of $X_H$ and $ u_* e_H = h = u_{h*} e_H$, we have $u = u_h$. This shows that $H \ni h \mapsto u_h \in \End(\Gamma_H)$ is an isomorphism of finite monoids. 
\end{proof}

\noindent
Let $p: M \twoheadrightarrow H$ be a finite quotient and $q_H: H \xrightarrow{\simeq} \End(\Gamma_H)$ be the isomorphism of Lemma \ref{h is isomorphic to end}. Also, denote by $\gamma: M \rightarrow \End(\Gamma_H)$ the composition:
\begin{equation}
 \gamma: M \xrightarrow{\lambda_M} \End(\F_M) \xrightarrow{\delta^{-1}} \lim_{\Gamma} \End(\Gamma) \xrightarrow{\pi_{\Gamma_H}} \End(\Gamma_H)
\end{equation}
where $\pi_{\Gamma_H}: \lim_{\Gamma} \End(\Gamma) \twoheadrightarrow \End(\Gamma_H)$ is the canonical projection. 
\begin{lemma}
\label{two inverse systems are equivalent}
 The following diagram commutes:
 \begin{eqnarray}
  \xymatrix{
   M \ar@{->>}[r]^p \ar@{->>}[dr]_{\gamma} & H \ar[d]^{q_H}\\
   & \End(\Gamma_H) 
 }
 \end{eqnarray}
\end{lemma}
\begin{proof}
 Recall that the composition $M \xrightarrow{\lambda_M} \End(\F_M) \xrightarrow{\delta^{-1}} \lim_{\Gamma} \End(\Gamma)$ assigns to each $m \in M$ the family $\{u_\Gamma^m\} \in \lim_{\Gamma} \End(\Gamma)$ such that $\xi \cdot m = u_{\Gamma *}^m \xi$ for every $\Gamma =(X,\xi)$. Since, by definition, $\gamma(m) \in \End(\Gamma_H)$ is the component $u^m_{\Gamma_H}$ of $\{u_\Gamma^m\}$ at the galois object $\Gamma_H = (X_H,e_H)$, it is characterized by the property that $p(m) = \gamma(m)_* e_H$, which also characterizes $q_H \circ p (m) \in \End(\Gamma_H)$ by definition of $q_H$. Thus, we have $\gamma(m) = q_H \circ p(m)$ for each $m \in M$, i.e.\ $\gamma = q_H \circ p$. 
\end{proof}

\noindent
Let $\ff = \{p_i: M \twoheadrightarrow H_i\}$ be the family of all finite quotients of $M$, which forms an inverse system of finite monoids by defining a pre-order $i \leq j$ if and only if there exists a (necessarily unique) homomorphism $\tau^j_i: H_j \rightarrow H_i$ such that $p_i = \tau^j_i \circ p_j$. By taking the inverse limit of this system $\ff$, we obtain a profinite monoid $\lim_i H_i$ and a canonical homomorphism $p: M \rightarrow \lim_i H_i$, which assigns to each $m \in M$ the family $\{p_i(m)\} \in \lim_i H_i$. It is well-known that this homomorphism $p: M \rightarrow \lim_i H_i$ is in fact an isomorphism of profinite monoids. 

Thus we have successive isomorphisms of profinite monoids:
\begin{eqnarray*}
 M &\xrightarrow{\simeq}& \lim_i H_i \hspace{0.5cm} (\textrm{induced by $p$}) \\
   &\xrightarrow{\simeq}& \lim_i \End(\Gamma_{H_i}) \hspace{0.5cm} (\textrm{induced by $q_{H_i}$}) \\
   &\xrightarrow{\simeq}& \lim_\Gamma \End(\Gamma) \hspace{0.5cm} (\because \textrm{by Lemma \ref{characterization of galois objects}}) 
\end{eqnarray*}
Then a careful chase of this sequence, using Lemma \ref{two inverse systems are equivalent}, shows that the composition of these isomorphisms is in fact equal to the composition $M \xrightarrow{\lambda_M} \End(\F_M) \xrightarrow{\delta^{-1}} \lim _{\Gamma} \End(\Gamma)$. Thus, $\delta^{-1} \circ \lambda_M$ and also $\lambda_M$ as well are isomorphisms of profinite monoids. This completes the proof of Theorem \ref{reconstruction of profinite monoid}. 

\subsection{Representation of semi-galois categories}
\label{appendix b two}
\noindent
In this section, we prove the following:

\begin{theorem}[Representation theorem]
\label{representation of semi-galois category}
 For an arbitrary semi-galois category $\langle \C, \F \rangle$, there exists a canonical equivalence of categories $\Phi: \C \xrightarrow{\simeq} \Cl_f \pi_1(\C,\F)$ such that the following diagram commutes:
\begin{equation}
\label{diagram for comparison functor}
  \xymatrix{
   \C \ar[rr]^{\Phi} \ar[dr]_{\F}& & \Cl_f \pi_1(\C,\F) \ar[dl]^{\F_{\pi_1(\C,\F)}} \\
   & \fsets & 
}
\end{equation}
\end{theorem}
\begin{proof}
 Throughout this proof, we denote $\pi_1:= \pi_1(\C,\F)$ for simplicity. Firstly, we define a functor $\Phi: \C \rightarrow \Cl_f \pi_1$ so that the diagram (\ref{diagram for comparison functor}) commutes. After that, we prove (i) the \emph{fully faithfulness} of $\Phi$ and (ii) the \emph{essentially surjectivity} of $\Phi$. 

Now we define a functor $\Phi: \C \rightarrow \pifsets$. Let $Y \in \C$ be an object. Then the fundamental monoid $\pi_1$ acts on the set $\F(Y)$ from the right by:
\begin{equation*}
  \F(Y) \times \pi_1 \ni (\eta, \phi) \longmapsto \eta \cdot \phi_Y \in \F(Y).
\end{equation*}
We denote this right action by $\rho_Y: \F(Y) \times \pi_1 \rightarrow \F(Y)$. To see that $\rho_Y \in \pifsets$, we need to see the continuity of $\rho_Y$. 
\begin{lemma}
 The action $\rho_Y$ is continuous with respect to the topology of $\pi_1$.
\end{lemma}
\begin{proof}
 To see this, it is sufficient to show that, for every $\eta, \eta' \in \F(Y)$, the following set is open in $\pi_1 = \End(\F)$:
\begin{eqnarray}
 U_{\eta,\eta'} &=& \{\phi \in \pi_1 \mid \eta \cdot \phi_Y = \eta' \}.
\end{eqnarray}
Recall that the topology of $\pi_1=\End(\F)$ is induced from the isomorphism $\delta: \lim_{\Gamma} \End(\Gamma) \rightarrow \End(\F)$. Under this isomorphism, $U_{\eta,\eta'}$ corresponds to the subset of $\lim_{\Gamma} \End(\Gamma)$ given by:
\begin{equation}
\label{open set in limend}
 \bigl\{ \uu=\{u_{\Gamma}\} \in \lim_{\leftarrow \Gamma} \End(\Gamma) \hspace{0.1cm} \mid \hspace{0.1cm} \sigma_{0 *} u_{\Gamma_0 *}\xi_0 = \eta' \bigr\},
\end{equation}
where $\sigma_0: X_0 \rightarrow Y$ is a fixed arrow from a galois object $\Gamma_0 = (X_0,\xi_0) \in \G$ with $\xi_0 \sigma_{0 *} = \eta$. By definition of the topology of $\lim_{\Gamma} \End(\F)$, the set (\ref{open set in limend}) is clearly open in $\lim_{\Gamma} \End(\Gamma)$. 
\end{proof}
\noindent
The assignment $\C \ni Y \mapsto \langle \F(Y), \rho_Y \rangle \in \pifsets$ gives the object portion of the functor $\Phi: \C \rightarrow \pifsets$. The arrow portion is given by, for an arrow $f: Y \rightarrow Z$ in $\C$, 
\begin{eqnarray}
 \Phi (f)  &:=& f_*: \F(Y) \rightarrow \F(Z)
\end{eqnarray}
It is obvious that $\Phi (f) : \F(Y) \rightarrow \F(Z)$ is indeed $\pi_1$-equivariant (i.e.\ commutes with right actions of $\pi_1$) by the naturality of members in $\pi_1 = \End(\F)$ with respect to arrows in $\C$. Thus $\Phi (f) = f_*$ is an arrow from $\Phi(Y) = \langle \F(Y), \rho_Y \rangle$ to $\Phi(Z) = \langle \F(Z), \rho_Z \rangle$ in $\pifsets$. By this construction, it is clear that the diagram (\ref{diagram for comparison functor}) commutes. 

The proof of the theorem is established if we prove (i) the fully faithfulness of $\Phi$; and (ii) the essentially surjectivity of $\Phi$. To do so, we need a bit more effort.
\begin{proposition}
\label{fully faithfulness}
 The functor $\Phi$ is fully faithful. 
\end{proposition}
\begin{proof}
 By definition, $\Phi$ is clearly faithful. So we prove the fullness of $\Phi$. It is sufficient to show that, for arbitrary objects $Y, Z$ in $\C$ and arrow $\sigma: \Phi(Y) \rightarrow \Phi(Z)$ in $\pifsets$, there exists an arrow $h: Y \rightarrow Z$ in $\C$ such that $\sigma = \Phi (h) = h_*$. We prove this claim by induction on the degree of $Y$. Before the proof, we need a lemma concerning epimorphisms from galois objects.

\begin{lemma}
\label{epi from galois objects}
 Let $(X,\xi) \in \G$ and $f: X \twoheadrightarrow Y$ be an epimorphism. Then, $f$ is the universal quotient of $X$ by the right congruence $E:=\{(h,k) \in \End(X) \times \End(X) \mid f h=f k \}$. 
\end{lemma}
\begin{proof}
 Let $p_E: X \twoheadrightarrow X/E$ be a universal quotient of $X$ by $E$. Since $f$ coequalizes $E$, $f$ factors through $p_E$, i.e.\ there is an arrow $t: X/E \rightarrow Y$ such that $f= t p_E$. Since $f$ is an epimorphism, so is $t$. Thus, to complete the proof, it is sufficient to see that $t_*$ is injective. 

We now have the following two diagrams in $\C$ and in $\fsets$ respectively:
\begin{eqnarray}
  \xymatrix{
   X \ar@{->>}[d]_{p_E} \ar[r]^f & Y  \\
   X/E \ar[ru]_t &  
} & \hspace{1cm} & 
  \xymatrix{
   \F(X) \ar@{->>}[d]_{p_{E*}} \ar[r]^{f_*} & \F(Y)  \\
   \F(X/E) \ar[ru]_{t_*} &  
}
\end{eqnarray}
By Lemma \ref{preserving universal quotients}, the map $p_{E*}: \F(X) \twoheadrightarrow \F(X/E)$ is a universal quotient of $\F(X)$ by $\F(E)$ in $\fsets$. Let us denote by $[\xi'] \in \F(X/E)$ the equivalence class containing $\xi' \in \F(X)$ (with respect to the equivalence described in Remark \ref{on universal quotients}, \S \ref{s3s1}). To prove that $t_*$ is injective, assume that $ t_* [\xi_1] = t_*[\xi_2]$, i.e.\ $f_* \xi_1 = f_* \xi_2$. Since $(X,\xi)$ is a galois object, there exist $h,k \in \End(X)$ such that $\xi_1 = h_*\xi$ and $\xi_2 = k_*\xi$. Thus, we obtain $f_* h_*\xi = f_* k_* \xi$. By Proposition \ref{maps from rooted objects}, we have $fh= fk$ which implies $(h,k) \in E$. Since $(\xi_1, \xi_2) = (h_* \xi, k_* \xi)$ with $(h,k) \in E$, this proves that $[\xi_1] = [\xi_2]$ in $\F(X/E)$ (cf.\ Remark \ref{on universal quotients}). Thus $t_*$ is indeed injective. 
\end{proof}

\begin{lemma}
 Let $Y \in \C$ have degree one, and $\sigma: \Phi(Y) \rightarrow \Phi(Z)$ be an arrow in $\pifsets$. Then there exists $h: Y \rightarrow Z$ in $\C$ such that $\sigma = h_*$. 
\end{lemma}
\begin{proof}
 Since $Y$ is a rooted object, Proposition \ref{characterization of rooted objects} shows that there exists $\eta \in \F(Y)$ such that $Y=Y_{\eta}$. By Proposition \ref{cofinality of galois objects} and the fact that $\G$ is cofiltered, it is easy to see that there exist a galois object $(X,\xi) \in \G$ and an arrow $f: X \rightarrow Y$ such that (i) $f_* \xi = \eta$ and (ii) $\Hom(X,Z) \ni g \mapsto g_* \xi \in \F(Z)$ is bijective. Then, there exists $g \in \Hom(X,Z)$ such that $g_* \xi = \sigma \eta = \sigma f_* \xi$. Now, we have arrows as in the following diagrams in $\C$ and in $\pifsets$ respectively:
\begin{eqnarray}
  \xymatrix{
   X \ar@{->>}[d]_f \ar[rd]^g \\
   Y & Z 
} & \hspace{1cm} & 
  \xymatrix{
  \F(X) \ar@{->>}[d]_{f_*} \ar[rd]^{g_*} & \\
  \F(Y) \ar[r]_{\sigma} & \F(Z)
}
\end{eqnarray}
Notice that, if we could prove that the right diagram commutes in $\pifsets$, i.e.\ $g_*= \sigma f_*$, then it follows that there exists $h: Y \rightarrow Z$ in $\C$ such that $\sigma = h_*$ as requested. In fact, by Lemma \ref{epi from galois objects}, the epimorphism $f: X \twoheadrightarrow Y$ is the universal quotient $X \twoheadrightarrow X/ E$ by some left congruence $E \subseteq \End(X) \times \End(X)$. Then, for every $(u,v) \in E$, we have:
\begin{eqnarray*}
    g_* u_* &=& (\sigma f_*) u_* \hspace{0.5cm} (\because \textrm{we now assume $g_*=f_* \sigma$})\\
           &=& ( \sigma f_*) v_* \hspace{0.52cm} (\because \textrm{$f$ coequalizes $\forall (u,v) \in E$})\\
           &=&  g_* v_*.
\end{eqnarray*}
Thus, $gu=gv$. Since $f: X \twoheadrightarrow Y$ is universal with respect to coequalizers of $E$, it follows that $g$ factors through $f$, i.e.\ $g= hf$ for some $h: Y \rightarrow Z$. Then, $h_*f_* = g_* = \sigma f_*$ and the surjectivity of $f_*$ imply that $\sigma = h_*$. Hence, our task is reduced to proving that $g_* = \sigma f_*$. For this aim, notice that, since $(X,\xi) = \Gamma$ is galois, (i) $\End(X) \ni k \mapsto k_* \xi \in \F(X)$ is a bijection; and (ii) the canonical projection $\pi_1 \ni \delta \{u_{\Gamma}\} \mapsto u_{\Gamma} \in \End(X)$ is a sujection. From this, it follows that, for every $\xi' \in \F(X)$, there is $\phi \in \pi_1$ such that $\xi' = \xi \cdot \phi_X$. By this and the fact that $\sigma: \F(Y) \rightarrow \F(Z)$ is $\pi_1$-equivariant (i.e.\ commutes with $\forall \phi \in \pi_1$), we have, for every $\xi' = \xi \cdot \phi_X \in \F(X)$:
\begin{eqnarray*}
  g_* \xi' &=& (g_* \xi) \cdot \phi_Z  \hspace{0.7cm} (\because \textrm{$g_*$ is $\pi_1$-equivariant})\\
           &=& ( \sigma f_*\xi) \cdot \phi_Z \hspace{0.5cm} (\because \textrm{definition of $g$})\\
           &=&  \sigma f_* (\xi \cdot \phi_X) \hspace{0.55cm} (\because \textrm{$f_*$ and $\sigma$ are $\pi_1$-equivariant}) \\
           &=& \sigma f_* \xi'.
\end{eqnarray*}
Since $\xi' \in \F(X)$ is choosen arbitrarily, this proves that $g_* = \sigma f_*$. 
\end{proof}

\begin{lemma}
 If the above lemma holds for objects $Y \in \C$ with degree less than or equal to $n \geq 1$, then it also holds for all objects $Y \in \C$ with degree $n + 1$. 
\end{lemma}
\begin{proof}
 Let $Y$ be an object with degree $n+1$. Then there is a covering $\{j_1: Y_1 \hookrightarrow Y, j_2: Y_2 \hookrightarrow Y\}$ of $Y$ such that $Y_1$ has degree one and $Y_2$ has degree $n$. Let $\sigma: \F(Y) \rightarrow \F(Z)$ be an arrow in $\pifsets$. Then, by induction hypothesis, there exist $f_1:Y_1 \rightarrow Z$ and $f_2: Y_2 \rightarrow Z$ with $f_{1*} =\sigma j_{1*}$ and $f_{2*} = \sigma j_{2*}$. Now, we have the following diagrams in $\C$ and in $\pifsets$: 
\begin{eqnarray}
  \xymatrix{
   Y_1 \ar@/^/[rrd]^{f_1} \ar[rd]_{j_1} & \\
    & Y & Z \\
   Y_2 \ar@/_/[rru]_{f_2} \ar[ru]^{j_2} & 
} & \hspace{1cm} &
  \xymatrix{
   \F(Y_1) \ar@/^/[rrd]^{f_{1*}} \ar[rd]_{j_{1*}} & \\
    & \F(Y) \ar[r]^{\sigma} & \F(Z) \\
   \F(Y_2) \ar@/_/[rru]_{f_{2*}} \ar[ru]^{j_{2*}} & 
}
\end{eqnarray}
Consider a pullback diagram of $j_1$ and $j_2$ in the left diagram in $\C$:
\begin{eqnarray}
\label{a pullback diagram in C}
  \xymatrix{
    & Y_1 \ar@/^/[rrd]^{f_1} \ar[rd]_{j_1} & & \\
    Y_1 \cap Y_2 \ar[ru] \ar[rd] & & Y & Z \\
    & Y_2 \ar@/_/[rru]_{f_2} \ar[ru]^{j_2} & &
}
\end{eqnarray}
and its image under $\Phi: \C \rightarrow \pifsets$, which is also a pullback diagram:
\begin{eqnarray}
\label{a pullback diagram in pifsets}
  \xymatrix{
    & \F(Y_1) \ar@/^/[rrd]^{f_{1*}} \ar[rd]_{j_{1*}} & & \\
    F(Y_1 \cap Y_2) \ar[ru] \ar[rd] & & \F(Y) \ar[r]^{\sigma} & \F(Z) \\
    & \F(Y_2) \ar@/_/[rru]_{f_{2*}} \ar[ru]^{j_{2*}} & &
}
\end{eqnarray}
Since the out-most square of (\ref{a pullback diagram in pifsets}) in $\pifsets$ commutes, so does the out-most square of (\ref{a pullback diagram in C}) in $\C$ by the faithfulness of $\Phi$. Also, it is easy to see that the pullback square in the first diagram in $\C$ is also a pushout square. Thus, there exists $h:Y \rightarrow Z$ that makes the diagram commutes. This arrow is the required one: i.e.\ $\sigma=h_*$. 
\end{proof}
\noindent
By these lemmas, the proof of the fully faithfulness of $\Phi$ is done. 
\end{proof}
\noindent
Finally, we prove the essentially surjectivity of $\Phi$, which concludes the proof of the representation theorem. 
\begin{proposition}
 The functor $\Phi: \C \rightarrow \pifsets$ is essentially surjective. 
\end{proposition}
\begin{proof}
 Let $\langle S, \rho \rangle \in \pifsets$ be a finite $\pi_1$-set. We need to prove that there exists $Y \in \C$ such that $\Phi(Y)$ is isomorphic to $S$ as $\pi_1$-sets. The proof is by induction on the size $|S|$ of $S$. The case of $|S|=1$ is trivial. So assume that the claim is true for those $S \in \C$ having $|S| \leq n$. In order to proceed the induction, we first prepare a lemma:
\begin{lemma}
 Assume that there exists a surjection $\sigma: \Phi(X) \twoheadrightarrow S$ in $\pifsets$ with $\Gamma = (X, \xi) \in \G$. Then there is a right congruence $E \subseteq \End(X) \times \End(X)$ such that $S \simeq \Phi(X/E)$. 
\end{lemma}
\begin{proof}
 Define $E:= \{(h,k) \in \End(X) \times \End(X) \mid  \sigma h_*= \sigma k_* \}$. Then $E$ is clearly a right congruence on $\End(X)$. Consider a universal quotient $p: X \twoheadrightarrow X/E$ in $\C$. By the universality of $p_*$, there exists $\bar{\sigma}: \Phi(X/E) \rightarrow S$ such that the following diagram commutes:
\begin{equation}
  \xymatrix{
   \Phi(X) \ar@{->>}[d]_{p_*} \ar@{->>}[r]^{\sigma} & S \\
   \Phi(X/E) \ar@{..>}[ru]_{\bar{\sigma}} &
}
\end{equation}
We construct an inverse of $\bar{\sigma}: \Phi(X/E) \rightarrow S$. For this aim, notice that $\Phi(X/E)$ is equal to the quotient $\Phi(X) / \equiv$, where the equivalence relation $\equiv$ on $\F(X)$ is the smallest $\pi_1$-equivalence among those containing the following subset of $\F(X) \times \F(X)$:
\begin{equation}
\label{pi equivalence}
  \bigl\{ (\xi_1,\xi_2) \mid \exists (h,k) \in E.  \hspace{0.1cm}\exists \zeta \in \F(X). \hspace{0.1cm} \xi_1 = h_*\zeta \wedge \xi_2 = k_* \zeta \bigr\}.
\end{equation}
Since $\sigma: \Phi(X) \twoheadrightarrow S$ is surjective, every $s \in S$ can be represented as $s= \sigma \xi'$ for some $\xi' \in \F(X)$. We now see that the assignment $S \ni \xi' \sigma \mapsto p_* \xi' \in \Phi(X/E)$ is well-defined. Since this assignment is, if well-defined, an inverse of $\bar{\sigma}$ in $\pifsets$, it concludes the proof of this lemma. 

Let $s \in S$ be represented in two ways as $s = \sigma \xi_1 = \sigma \xi_2$ with $\xi_1, \xi_2 \in \F(X)$. Then, we need to show that $ p_*\xi_1 = p_* \xi_2$. By the assumption that $\Gamma = (X, \xi)$ is a galois object, there exist $h, k \in \End(X)$ such that $\xi_1 = h_* \xi$ and $\xi_2 = k_* \xi$. Thus, from $ \sigma\xi_1 = s = \sigma \xi_2$, we obtain $ \sigma h_*\xi  = \sigma k_* \xi$. Since $(X, \xi)$ is a galois object, every $\xi' \in \F(X)$ can be represented as $\xi' = \xi \cdot \phi_X$ for some $\phi \in \pi_1$. Therefore, for every $\xi' \in \F(X)$:
\begin{eqnarray*}
 \sigma h_* \xi' &=&  \sigma h_*(\xi \cdot \phi_X) \\
                  &=& (\sigma h_*\xi) \cdot \phi \hspace{0.5cm} (\because \textrm{$h_*, \sigma$ are $\pi_1$-equivariant})\\
                  &=&  (\sigma k_*\xi) \cdot \phi \\
                  &=& \sigma k_*(\xi \cdot \phi_X) \hspace{0.5cm} (\because \textrm{$k_*, \sigma$ are $\pi_1$-equivariant})\\
                  &=& \sigma  k_* \xi'
\end{eqnarray*}
This implies $ \sigma h_* =\sigma k_*$, i.e.\ $(h,k) \in E$. So, the pair $(\xi_1,\xi_2) = (h_*\xi,k_* \xi)$ belongs to the set (\ref{pi equivalence}). This proves that $p_*\xi_1  = p_* \xi_2$ in $\Phi(X/E)$ and thus the assignment $S \in \xi' \sigma \mapsto p_* \xi' \in \Phi(X/E)$ is well-defined. This establishes an isomorphism $S \simeq \Phi(X/E)$. 
\end{proof}

\noindent
To continue the induction, let $S \in \pifsets$ be such that $|S|=n+1$. The proof is divided into two cases. 
\begin{case}
 $S= a \cdot \pi_1$ for some $a \in S$. 
\end{case}

\noindent
By the above lemma, it suffices to show that there exists a surjection $\Phi(X) \twoheadrightarrow S$ with $\Gamma = (X,\xi) \in \G$. Since $\kappa: \pi_1 \ni \phi \mapsto a \cdot \phi \in S$ is a continuous surjection onto a finite set $S$ and $\pi_1 \simeq \lim_{\Gamma} \End(\Gamma)$, there exist a galois object $\Gamma = (X,\xi) \in \G$ and a surjection $\lambda: \End(\Gamma) \rightarrow S$ such that the following commutes:
\begin{equation}
 \xymatrix{ 
  \pi_1 \simeq \lim_{\Gamma} \End(\Gamma) \ar@{->>}[d]_{r} \ar@{->>}[r]^>>>>>{\kappa} & S \\
  \End(\Gamma) \ar@{..>>}[ru]_{\lambda} &
}
\end{equation}
Combining with the bijection $\omega_{X,\xi}: \End(\Gamma) \ni h \mapsto h_* \xi \in \F(X)$, we obtain a surjective map $\omega_{X,\xi}^{-1} \lambda = \sigma: \F(X) \twoheadrightarrow S$. We show that $\sigma$ is in fact a $\pi_1$-equivariant map, i.e.\ a surjection in $\pifsets$. Recall that, since $(X,\xi)$ is a galois object, every $\xi' \in \F(X)$ can be represented as $\xi' =\xi \cdot \phi_X$ for some $\phi \in \pi_1$. Also notice that, by the construction, the map $\sigma: \F(X) \twoheadrightarrow S$ assigns to each $\xi'=\xi \cdot\phi_X $ the element $a \cdot \phi \in S$. Thus, for every $\psi \in \pi_1$ and $\xi'=\xi \cdot \phi_X \in \F(X)$:
\begin{eqnarray*}
   (\sigma \xi') \cdot \psi &= & (a\cdot \phi) \cdot \psi \\
                            &= & a \cdot (\phi \cdot \psi) \\
                            &= & \sigma ((\xi \cdot \phi_X ) \cdot \psi_X) \\
                            &= & \sigma (\xi' \cdot \psi). 
\end{eqnarray*}
This proves that the map $\sigma: \F(X) \twoheadrightarrow S$ gives a surjection from $\Phi(X)$ to $S$ in $\pifsets$ as requested.

\begin{case}
 $S \neq a \cdot \pi_1$ for any $a \in S$. 
\end{case}

\noindent
In this case, there exist subobjects $S_1, S_2$ of $S$ such that $|S_i| < |S|$ ($i=1,2$) and $\{S_1 \hookrightarrow S, S_2 \hookrightarrow S\}$ is a covering of $S$. By induction hypothesis, there exist $Y_1, Y_2 \in \C$ such that $\Phi(Y_i) \simeq S_i$ ($i=1,2$). Consider the following pullback diagram taken in $\pifsets$:
\begin{equation}
   \xymatrix{
     S_1 \cap S_2 \ar[r] \ar[d] & S_1 \ar[d] \\
     S_2 \ar[r] & S
}
\end{equation}
Since $|S_1 \cap S_2| < |S|$, there also exists $Z \in \C$ such that $\Phi(Z) \simeq S_1 \cap S_2$. By taking a pushout diagram of $Z \hookrightarrow Y_1$ and $Z \hookrightarrow Y_2$ in $\C$, we have diagrams in $\C$ and in $\pifsets$ as follows:
\begin{eqnarray}
  \xymatrix{
  Z \ar[r] \ar[d] & Y_1 \ar[d] \\
  Y_2 \ar[r] & Y
} & \hspace{1.5cm} & 
  \xymatrix{
  S_1 \cap S_2 \ar[r] \ar[d] & S_1 \ar[d] \\
  S_2 \ar[r] & S
}
\end{eqnarray}
Then, by the fact that both diagrams are pushouts and $\Phi$ preserves them, it follows that $\Phi(Y) \simeq S$. This completes the proof of the proposition. 
\end{proof}

\noindent
We have proved that $\Phi: \C \rightarrow \pifsets$ is fully faithful and essentially surjective. Thus, $\Phi$ is an equivalence of categories. 
\end{proof}

\subsection{The duality theorem}
\label{appendix b three}
\noindent
Finally, we prove that the category $\profmon$ of profinite monoids and continuous homomorphisms is opposite equivalent to the category $\semigalois$ of semi-galois categories and \emph{exact functors} (Definition \ref{exact functor}). The results in the previous subsections (\S \ref{appendix b one} -- \S \ref{appendix b two}) are then understood as the object portions of the target duality:
\begin{eqnarray}
\label{duality}
\profmon &\simeq& \semigalois^{op}. 
\end{eqnarray}
So, the last problem for obtaining this duality is to show that the correspondences on objects, i.e.\ $M \mapsto \langle \mfsets, \F_M \rangle$ and $\langle \C, \F \rangle \mapsto \pi_1(\C,\F)$, extend to functors $\profmon \rightarrow \semigalois^{op}$ and $\semigalois^{op} \rightarrow \profmon$; and prove that these are mutually inverse. 

To proceed to the proof, we first need to define arrows in $\semigalois$, i.e.\ arrows between semi-galois categories. Briefly speaking, they are defined as suitable isomorphism classes of functors between semi-galois categories that preserve fiber functors and finite limits and colimits, which we will call \emph{exact functors}. To be more precise, exact functors are functors that preserve fiber functors up to natural isomorphism, and formally, defined as follows:

\begin{definition}[exact functor]
\label{exact functor}
 Let $\langle \C,\F \rangle$ and $\langle \C', \F' \rangle$ be semi-galois categories. An \emph{exact functor} is a pair $\langle A,\sigma \rangle$ of a functor $A: \C \rightarrow \C'$ and a natural isomorphism $\sigma: \F' \circ A \Rightarrow \F$ such that $A$ preserves finite limits and colimits. 
\end{definition}

\noindent
The composition of composable exact functors $\langle A,\sigma \rangle: \langle \C,\F \rangle \rightarrow \langle \C',\F' \rangle$, $\langle A',\sigma' \rangle: \langle \C',\F' \rangle \rightarrow \langle \C'',\F'' \rangle$ is given by $\langle A' \circ A, \sigma'' \rangle:\langle \C,\F \rangle \rightarrow \langle \C'',\F'' \rangle$, where the natural isomorphism $\sigma'': \F'' \circ (A' \circ A) \Rightarrow \F$ is defined by:
\begin{equation}
\label{composition of exact functors}
 \sigma'': \F'' \circ A' \circ A \xRightarrow{\sigma' A} \F' \circ A \xRightarrow{\hspace{0.15cm}\sigma \hspace{0.15cm}} \F.
\end{equation}
It is straightforward to show that the composition of exact functors is indeed associative. Also, the pair of identities $\langle \id_{\C}, \id_{\F} \rangle : \langle \C, \F \rangle \rightarrow \langle \C, \F \rangle$ is obviously an exact functor that is unit with respect to the above composition of exact functors. However, to define $\semigalois$ so that it is opposite equivalent to $\profmon$, we need to identify exact functors $\langle A,\sigma \rangle$ with respect to a suitable equivalence relation. 

For this aim, let $\langle \C,\F \rangle$ and $\langle \C',\F' \rangle$ be semi-galois categories, and $\langle A,\sigma \rangle$ and $\langle B,\tau \rangle: \langle \C,\F \rangle \rightarrow \langle \C', \F' \rangle$ be parallel exact functors. We say that  $\langle A,\sigma \rangle$ and $\langle B,\tau \rangle$ are \emph{equivalent} and write $\langle A,\sigma \rangle \equiv \langle B, \tau \rangle$ if there exists a natural isomorphism $\lambda: A \Rightarrow B$ such that $F'\lambda \cdot \tau = \sigma$. It is not difficult to see that this equivalence $\equiv$ is a category congruence (i.e.\ preserved by compositions); and thus, we can define a category $\semigalois$ as follows:

\begin{definition}[The category of semi-galois categories]
 The category $\semigalois$ is defined as the category whose objects are semi-galois categories $\langle \C,\F \rangle$ and whose arrows are equivalence classes of exact functors $\langle A,\sigma \rangle$ with respect to the above equivalence relation $\equiv$. Denoting the equivalence class of $\langle A,\sigma \rangle$ by $[A,\sigma]$, the composition of $[A,\sigma]: \langle \C, \F \rangle \rightarrow \langle \C',\F' \rangle$ and $[A',\sigma']: \langle \C',\F' \rangle \rightarrow \langle \C'',\F'' \rangle$ is defined by $[A' \circ A, \sigma'']$, where $\sigma''$ is the one given in (\ref{composition of exact functors}). 
\end{definition}

\noindent
The rest of this subsection is devoted to a proof of the fact that $\semigalois$ and $\profmon$ are opposite equivalent to each other via the constructions $\langle \C,\F \rangle \mapsto \pi_1(\C,\F)$ of fundamental monoids; and $M \mapsto \langle \mfsets,\F_M \rangle$ of the category of finite $M$-sets. Formally:
\begin{theorem}[The duality theorem]
 There are equivalences of categories, $\pi_1: \semigalois^{op} \xrightarrow{\simeq} \profmon$ and $\rep: \profmon \xrightarrow{\simeq} \semigalois^{op}$ such that 
\begin{enumerate}
 \item $\pi_1$ assigns the fundamental monoid $\pi_1(\C,\F)$ to each semi-galois category $\langle \C,\F \rangle \in \semigalois$; 
 \item $\rep$ assigns the semi-galois category $\langle \mfsets, \F_M \rangle$ to each profinite monoid $M \in \profmon$; 
\end{enumerate}
and also, $\pi_1$ and $\rep$ are mutually inverse up to natural isomorphisms. 
\end{theorem}

\noindent
We firstly define the functor $\pi_1:\semigalois^{op} \rightarrow \profmon$; secondly the functor $\rep:\profmon \rightarrow \semigalois^{op}$; and finally, prove that they are mutually inverse. 
\begin{definition}
 Let $\langle \C,\F \rangle$ and $\langle \C', \F' \rangle$ be semi-galois categories and $\langle A,\sigma \rangle: \langle \C,\F \rangle \rightarrow \langle \C',\F' \rangle$ be an exact functor. Then, a map $\pi_1(A,\sigma) :\pi_1(\C',\F') \rightarrow \pi_1(\C,\F)$ between their fundamental monoids is defined as follows:
\begin{eqnarray*}
  \pi_1(\C',\F')  &\xrightarrow{\pi_1(A,\sigma)}& \pi_1(\C,\F) \\
  \phi &\longmapsto& \sigma^{-1} \cdot \phi A \cdot \sigma
\end{eqnarray*}
\end{definition}

\begin{lemma}
 The map $\pi_1(A,\sigma): \pi_1(\C',\F') \rightarrow \pi_1(\C,\F)$ is a continuous homomorphism of profinite monoids, i.e.\ defines an arrow in $\profmon$. 
\end{lemma}
\begin{proof}
 It is easy to see that $\pi_1(A,\sigma)$ is continuous and preserves the monoid unit. To show the rest of the claim, let $\phi,\psi \in \pi_1(\C',\F')$. Then:
 \begin{eqnarray*}
   \pi_1(A,\sigma) (\phi \cdot \psi) &=& \sigma^{-1} \cdot (\phi \cdot \psi) A \cdot \sigma \\
                                     &=& \sigma^{-1} \cdot \phi A \cdot \psi A \cdot \sigma \\
                                     &=& (\sigma^{-1} \cdot \phi A \cdot \sigma) \cdot (\sigma^{-1} \cdot \psi A \cdot \sigma) \\
                                     &=& \pi_1(A,\sigma) (\phi) \cdot \pi_1(A,\sigma) (\psi). 
 \end{eqnarray*}
Thus, $\pi_1(A,\sigma)$ is a continuous homomorphism. 
\end{proof}

\begin{lemma}
The equivalence $\langle A,\sigma \rangle \equiv \langle B,\tau \rangle$ implies $\pi_1(A,\sigma) = \pi_1(B,\tau)$. In other words, the assignment $[A,\sigma] \mapsto \pi_1(A,\sigma)$ is well-defined.
\end{lemma}
\begin{proof}
 By definition of the equivalence $\langle A,\sigma \rangle \equiv \langle B,\tau \rangle$, there exists a natural isomorphism $\lambda: A \Rightarrow B$ such that $\F'\lambda \cdot \tau = \sigma$. Then, for $\phi \in \pi_1(\C',\F')$:
\begin{eqnarray*}
  \pi_1(A,\sigma) (\phi) &=& \sigma^{-1} \cdot \phi A \cdot \sigma \\
                  &=& \tau^{-1} \cdot (\F' \lambda)^{-1} \cdot \phi A \cdot (\F' \lambda) \cdot \tau \\
                  &=& \tau^{-1} \cdot \phi B \cdot \tau \hspace{0.5cm} (\because \textrm{the naturality of $\phi$})\\
                  &=& \pi_1(B,\tau) (\phi)
\end{eqnarray*}
Thus, $\pi_1(A,\sigma) = \pi_1(B,\tau)$. 
\end{proof}

\begin{lemma}
 The assignment $[ A,\sigma ] \mapsto \pi_1(A,\sigma)$ preserves the identities and the composition of exact functors. 
\end{lemma}
\begin{proof}
 The first claim is obvious. To see the second, let $[ A,\sigma ]: \langle \C, \F \rangle \rightarrow \langle \C',\F' \rangle$ and $[ A',\sigma' ]: \langle \C',\F' \rangle \rightarrow \langle \C'',\F'' \rangle$ be composable arrows in $\semigalois$. Their composition is given by $[ A' \circ A, (\sigma' A') \cdot \sigma]$. Then, for every $\phi \in \pi_1(\C',\F')$:
\begin{eqnarray*}
 && \pi_1(A'\circ A, (\sigma' A) \cdot \sigma) (\phi) \\
 &=& \sigma^{-1} \cdot (\sigma' A)^{-1} \cdot \phi (A' \circ A) \cdot (\sigma' A) \cdot \sigma \\
 &=& \sigma^{-1} \cdot (\sigma'^{-1} \cdot \phi A' \cdot \sigma') A \cdot \sigma \\
 &=& \pi_1(A, \sigma) (\sigma'^{-1} \cdot \phi A' \cdot \sigma') \\
 &=& \pi_1(A,\sigma) \circ \pi_1(A',\sigma') (\phi)
\end{eqnarray*}
\end{proof}

\begin{definition}
The functor $\pi_1: \semigalois^{op} \rightarrow \profmon$ is defined as the assignments $\langle \C,\F \rangle \mapsto \pi_1(\C,\F)$ on objects; and $[A,\sigma] \mapsto \pi_1(A,\sigma)$ on arrows. 
\end{definition}

\noindent
To give its inverse, we define a functor $\rep: \profmon \rightarrow \semigalois^{op}$. Let $M, M'$ be profinite monoids and $f: M \rightarrow M'$ be a continuous homomorphism. We need to define an arrow $\rep(f): \langle \Cl_f M', \F_{M'} \rangle \rightarrow \langle \mfsets, \F_M \rangle$ of semi-galois categories. 

For each finite $M'$-set $\langle S,\rho \rangle \in \Cl_f M'$, canonically induced by $f$ is a finite $M$-set $\langle S,f^* \rho \rangle \in \mfsets$, where $f^*\rho: M \rightarrow \End(S)$ is the composition:
\begin{equation}
\label{pullback of m-sets}
  M \xrightarrow{f} M' \xrightarrow{\rho} \End(S).
\end{equation}
Also, if $h: \langle S, \rho \rangle \rightarrow \langle T,\tau \rangle$ is an arrow in $\Cl_f M'$, i.e.\ an $M'$-equivariant map, the map $h:S \rightarrow T$ obviously preserves $M$-actions on $S$ and $T$ given by $f^* \rho$ and $f^* \tau$ respectively. That is, the map $h$ is an $M$-equivariant map from $\langle S,f^*\rho \rangle$ to $\langle T,f^* \tau \rangle$. Moreover, it is obvious that the assignments $\langle S,\rho \rangle \mapsto \langle S,f^*\rho \rangle$ and $h \mapsto h$ define a functor $\Cl_f M' \rightarrow \mfsets$, which we denote by $\rr(f)$. 

\begin{lemma}
 The functor $\rr(f):\Cl_f M' \rightarrow \mfsets$ preserves finite limits and finite colimits, and also, strictly preserves fiber functors: i.e.\ $\F_M \circ \rr(f) = \F_{M'}$. 
\end{lemma}
\begin{proof}
 This is obvious from the definition of $\rr(f)$ and the construction of finite limits and finite colimits in $\mfsets$. 
\end{proof}

\noindent
Thus, the pair $\langle \rr(f), \id_{F_M} \rangle$ defines an exact functor from $\langle \Cl_f M', \F_{M'} \rangle$ to $\langle \mfsets, \F_M \rangle$; and represents an arrow $\rep(f):= [\rr(f), \id_{\F_M}]$ of semi-galois categories. Since the functor $\rr(f)$ is defined by the composition (\ref{pullback of m-sets}), it also follows that we have the strict identity $\rr(g \circ f) = \rr(f) \circ \rr(g)$ of functors, where $f: M \rightarrow M'$ and $g: M' \rightarrow M''$ are composable continuous homomorphisms. Finally, $\rr(\id_M) = \id_{\mfsets}$. 

\begin{definition}
 The functor $\rep: \profmon \rightarrow \semigalois^{op}$ is defined by the assigments $M \mapsto \langle \mfsets, \F_M \rangle$ on objects; and $f \mapsto \rep(f)$ on arrows. 
\end{definition}
Now, we have two functors $\pi_1: \semigalois^{op} \rightarrow \profmon$ and $\rep: \profmon \rightarrow \semigalois^{op}$. In the rest of this subsection, we check that these functors are mutually inverse:

\begin{proposition}
 $\pi_1 \circ \rep \simeq \id_{\profmon}$. 
\end{proposition}
\begin{proof}
 Recall the proof of Theorem \ref{reconstruction of profinite monoid}, where we constructed an isomorphism $M \xrightarrow{\lambda_M} \pi_1(\mfsets,\F_M)$ of profinite monoids. So, in order to prove the claim, it is sufficient to show that the following diagram commutes for every continuous homomorphism $f:M \rightarrow M'$:
\begin{eqnarray}
\label{naturality of pi and r}
 \xymatrix{
  M \ar[d]_f \ar[rr]^{\hspace{-0.7cm}\lambda_M} & & \pi_1(\mfsets,\F_M) \ar[d]^{\pi_1(\rr(f), \id_{\F_M})}\\
  M' \ar[rr]_{\hspace{-0.7cm}\lambda_{M'}} && \pi_1(\Cl_f M',\F_{M'})
}
\end{eqnarray}
For simplicity, put $\theta:=\pi_1(\rr(f), \id_{\F_M})$ throughout this proof. Let $m \in M$ and take $Y:=\langle S,\rho \rangle \in \Cl_f M'$ arbitrarily. Then, the $Y$-component of the natural transformation $\theta (\lambda_M (m)) \in \pi_1(\Cl_f M', \F_{M'})$ is given by: for each element $a \in S = \F_{M'} (Y)$, 
\begin{eqnarray*}
   a \cdot \bigl(\theta (\lambda_M (m)) \bigr) _Y &=& \bigl ( f^* \rho (m) \bigr) (a) \\
   &=& a \cdot f(m). 
\end{eqnarray*}
By definition of $\lambda_{M'}$, this is clearly equal to $a \cdot \bigl(\lambda_{M'} ( f(m) )\bigr)_Y$. Thus, the diagram (\ref{naturality of pi and r}) commutes. 
\end{proof}

\begin{proposition}
 $\rep \circ \pi_1 \simeq \id_{\semigalois}$.
\end{proposition}
\begin{proof}
 For this proof, recall Theorem \ref{representation of semi-galois category}, where we constructed a canonical functor $\Phi: \C \rightarrow \Cl_f \pi_1(\C,\F)$ that was shown to be an equivalence of categories for every semi-galois category $\langle \C,\F \rangle$. Moreover, $\Phi$ strictly preserves fiber functors, i.e.\ $\F_{\pi_1(\C,\F)} \circ \Phi = \F$. This means that $\langle \Phi, \id_{\F} \rangle$ represents an arrow $[\Phi,\id_{\F}]$ of semi-galois categories. In the current proof, we denote $\Phi$ by $\Phi_{\C,\F}$ when indicating the domain.

Let $\langle \C,\F\rangle, \langle \C',\F' \rangle$ be semi-galois categories; and $[A,\sigma]: \langle \C,\F \rangle \rightarrow \langle \C',\F' \rangle$ be an arrow between them. To prove the claim, it is sufficient to show that the following diagram commutes in $\semigalois$: 
\begin{eqnarray}
 \xymatrix{
  \langle \C,\F \rangle \ar[rr]^{\hspace{-0.2cm}[\Phi_{\C,\F}, \id_{\F}]} \ar[d]_{[A,\sigma]}&& \Cl_f \pi_1(\C,\F) \ar[d]^{\rep(\pi_1(A,\sigma))} \\
  \langle \C',\F' \rangle \ar[rr]_{\hspace{-0.2cm}[\Phi_{\C',\F'}, \id_{\F'}]} && \Cl_f \pi_1(\C',\F')   
}
\end{eqnarray}
In more elementary words, this is equivalent to showing the following equivalence of exact functors:
\begin{eqnarray}
\label{eqeq}
 \langle \Phi_{\C',\F'} \circ A, \sigma \rangle &\equiv& \langle \rr(\pi_1(A,\sigma)) \circ \Phi_{\C,\F}, \id_{\F} \rangle. 
\end{eqnarray}
We now see that the natural isomorphism $\sigma: \F'\circ A \Rightarrow \F$ itself plays the role of proving this equivalence. 

To save the notation, let us write $\rr:= \rr(\pi_1(A,\sigma))$, $\Phi:=\Phi_{\C,\F}$ and $\Phi':= \Phi_{\C',\F'}$. Also, for each $\langle S,\rho\rangle \in \Cl_f \pi_1(\C,\F)$, we denote the induced finite $\pi_1(\C',\F')$-set $\rr(\langle S,\rho \rangle) = \langle S, \pi_1(A,\sigma)^* \rho \rangle \in \Cl_f \pi_1(\C',\F')$ simply by $\langle S,\rr(\rho) \rangle$. Then, to prove the equivalence (\ref{eqeq}), we need to construct a natural isomorphism $\lambda: \Phi' \circ A \Rightarrow \rr \circ \Phi$ such that:
\begin{eqnarray}
\label{requiremarkent for lambda}
 \F_{\pi_1(\C',\F')} \lambda \cdot \id_{\F} = \sigma.
\end{eqnarray}
Notice that the requested natural isomorphism $\lambda$ should consist of isomorphisms $\lambda_Y: \langle \F'(A(Y)), \rho_{A(Y)} \rangle \rightarrow \langle \F(Y), \rr(\rho_Y) \rangle$ ($Y \in \C$) of finite $\pi_1(\C',\F')$-sets. That is, each map $\lambda_Y: \F'(A(Y)) \rightarrow \F(Y))$ is an isomorphism that preserves $\pi_1(\C',\F')$-actions. On the other hand, the requested identity (\ref{requiremarkent for lambda}) says that this map $\lambda_Y$ must be equal to $\sigma_Y: \F' (A(Y)) \rightarrow \F(Y)$--- in other words, finding a natural isomorphism $\lambda: \Phi' \circ A \Rightarrow \rr \circ \Phi$ so that it satisfies (\ref{requiremarkent for lambda}) is equivalent to proving that $\sigma_Y: \F'(A(Y)) \rightarrow \F(Y)$ is not just an isomorphism, but also preserves $\pi_1(\C',\F')$-actions on $\F'(A(Y))$ and $\F(Y)$ given by $\rho_{A(Y)}$ and $\rr(\rho_Y)$ respectively. 

To see this, we recall the definitions of $\rho_{A(Y)}$ and $\rr(\rho_Y)$. On the one hand, $\rho_{A(Y)}$ is the following monoid homomorphism, which defines a $\pi_1(\C',\F')$-action on the finite set $\F'(A(Y))$:
\begin{equation}
\label{rho a y}
  \rho_{A(Y)} : \pi_1(\C',\F') \ni \phi \mapsto \phi_{A(Y)} \in \End\bigl(\F'(A(Y))\bigr). 
\end{equation}
That is, for each $\phi \in \pi_1(\C',\F')$ and $\eta \in \F'(A(Y))$, the action $\eta \cdot \phi \in \F'(A(Y))$ of $\phi$ on $\eta$ is $\eta \cdot \phi := \eta \cdot  \phi_{A(Y)}$. On the other hand, as mentioned above, $\rr(\rho_Y)$ is defined as a monoid homomorphism $\pi_1(A,\sigma)^* \rho_Y$. That is:
\begin{equation}
  \rr(\rho_Y) : \pi_1(\C', \F') \xrightarrow{\pi_1(A,\sigma)} \pi_1(\C,\F) \xrightarrow{\rho_Y} \End(\F(Y)), 
\end{equation}
which maps each $\phi \in \pi_1(\C', \F')$ to $\rho_Y \bigl(\pi_1(A,\sigma) (\phi)\bigr) = \bigl(\sigma^{-1} \cdot \phi A \cdot \sigma \bigr)_Y$. Combining with (\ref{rho a y}), we have the following commutative diagram for each $\phi \in \pi_1(\C',\F')$:
\begin{eqnarray}
 \xymatrix{
  \F' (A(Y)) \ar[r]^{\hspace{0.2cm}\sigma_Y} \ar[d]_{\rho_{A(Y)}(\phi)} & \F (Y) \ar[d]^{\rr(\rho_Y)(\phi)} \\
  \F' (A(Y)) \ar[r]_{\hspace{0.2cm}\sigma_Y} & \F(Y) 
}
\end{eqnarray}
This proves that, as requested, the map $\sigma_Y: \F'(A(Y)) \rightarrow \F(Y)$ indeed preserves $\pi_1(\C',\F')$-actions on $\F'(A(Y))$ and $\F(Y)$ given by $\rho_{A(Y)}$ and $\rr(\rho_Y)$ respectively. This completes the proof. 
\end{proof}

\noindent
By these natural isomorphisms, $\pi_1 \circ \rep \simeq \id_{\profmon}$ and $\rep \circ \pi_1 \simeq \id_{\semigalois}$, we obtained the requested duality:
\begin{eqnarray}
\pi_1 : \semigalois^{op} \leftrightarrows \profmon : \rep
\end{eqnarray}

\section{Eilenberg's Variety Theory, Revisited}
\label{s5}
\noindent
So far we have studied the general structure of semi-galois categories (\S \ref{s2} -- \S \ref{s4}); in particular, semi-galois categories were proved dual to profinite monoids under the construction $\sgc \mapsto \pi_1(\C,\F)$ of fundamental monoids. This extends, to the case of profinite monoids, the classical duality between profinite groups and galois categories. While the duality for galois categories was used to unify several galois theories and to develop galois theory for schemes (cf.\ \S \ref{s2s2}), the first use of our duality theorem for semi-galois categories is to review Eilenberg's variety theory \cite{Eilenberg} from a more axiomatic standpoint. 

Technically speaking, we review here two variants of \emph{Eilenberg's variety theorem}--- the central theorem in Eilenberg theory--- due to Straubing \cite{Straubing} and Chaubard et al.\ \cite{Chaubard_Pin_Straubing} (cf.\ \S \ref{section var one}); and give yet another proof of these theorems based on appropriate duality theorems (\S \ref{section var two}). This proof is not intended to simplify the original proof; but instead, to highlight the duality principle behind the variety theorems. This rather conceptual proof of the variety theorems naturally gives rise to several questions concerning the relationship between Galois theory, Eilenberg theory and \emph{B\"uchi's monadic second-order logic over finite words} \cite{Buchi} MSO[$<$]; after pursueing general topos-theoretic aspects of semi-galois categories in the next section (\S \ref{s6}), this matter will be discussed in the last section (\S \ref{s7}), where we also discuss the original motivation of the variety theorems in order to clarify the motivation of our yet another conceptual framework of Eilenberg theory.

\subsection{The variety theorem}
\label{section var one}
\noindent
This subsection recalls necessary concepts and results concerning the variety theorems due to Straubing \cite{Straubing} and Chaubard et al.\ \cite{Chaubard_Pin_Straubing} for the sake of reader's convenience; no novel result is presented in this subsection. Our consideration on these theorems starts in the next subsection (\S \ref{section var two}). For more information, the reader is refered to the original source \cite{Straubing,Chaubard_Pin_Straubing}.

Firstly let us fix some general notations that are used throughout this section. We denote by $\freemon$ the category of free monoids (over \emph{finite} alphabets) and monoid homomorphisms; the symbols $\D, \D' \cdots$ denote those subcategories of $\freemon$ which contain all free monoids as objects (but not necessarily all morphisms, i.e.\ may not be full). If a homomorphism $f: A^* \rightarrow B^*$ belongs to some $\D \subseteq \freemon$, we denote it as $f \in \D$. 

The central concepts for the variety theorems of \cite{Straubing,Chaubard_Pin_Straubing} are \emph{$\D$-varieties} of (i) regular languages, (ii) \emph{finite stamps} (= finite monoids with fixed generators), and (iii) \emph{finite actions} (= DFAs in our terminology), to be recalled in detail below. Then the variety theorems claim canonical bijective correspondences between varieties of these three types. Following \cite{Adamek_general}, we also define \emph{local varieties}\footnote{The original definitions of $\D$-varieties are slightly different from those presented below, but actually equivalent. This equivalent modification of definition is just to make it straightforward to see the relationship between $\D$-varieties and local varieties.}. 
\begin{definition}[$\D$-varieties of regular languages]
\label{d varieties of regular languages}
 A class $\V$ of regular languages is called a \emph{$\D$-variety of regular languages} if, denoting $\V_A:= \{ L \subseteq A^* \mid L \in \V \}$, the following properties hold:
 \begin{description}
  \item[R1)] $\emptyset \in \V_A$ and $A^* \in \V_A$ for every alphabet $A$;
  \item[R2)] if $L, R \in \V_A$, then $L \cup R, L \cap R$ and $A^* \backslash L \in \V_A$; 
  \item[R3)] if $L \in \V_A$ and $w \in A^*$, then $w^{-1} L$ and $ L w^{-1} \in \V_A$;
  \item[R4)] if $L \in \V_B$ and $f:A^* \rightarrow B^* \in \D$, then $f^{-1}L \in \V_A$. 
 \end{description}
A class $\V_A$ of regular languages over $A$ is called a \emph{local variety of regular languages} if it satisfies {\bf R1)}, {\bf R2)} and {\bf R3)}.
\end{definition}
\noindent
Here, for a language $L \subseteq A^*$ and a word $w \in A^*$, we denote by $w^{-1}L \subseteq A^*$ the language $\{u \in A^* \mid wu \in L\}$ and call it the \emph{left quotient of $L$ by $w$}. The language $Lw^{-1} \subseteq A^*$ is defined similarly and called the \emph{right quotient of $L$ by $w$}. 

\begin{remark}
The original varieties of regular languages in Eilenberg's sense are exactly $\D$-varieties of regular languages with $\D=\freemon$ in particular. In this sense, the notion of $\D$-varieties of regular languages subsumes the original varieties.
\end{remark}

\begin{remark}
\label{why variety}
 In Eilenberg theory, a main motivation to focus on varieties of regular languages among other general classes of regular languages is that the membership in varieties of regular languages admits \emph{algebraic characterization} in terms of syntactic monoids: For instance, the class of those regular languages which are definable by the first-order fragment FO[$<$] of MSO[$<$] is a variety of regular languages; and a regular language $L$ is a menber of this variety if and only if its syntactic monoid $M(L)$ contains only trivial subgroups (i.e.\ is \emph{aperiodic}). (See e.g.\ \cite{survey_logic,Pin}.)

The practical motivation of this algebraic characterization of FO[$<$]-definability concerns its decidability: The aperiodicity of $M(L)$ is equivalent to the property that $M(L)$ satisfies $\forall x \in M(L). \hspace{0.1cm} x^\omega = x^{\omega + 1}$; and importantly, this equational property is decidable. From this and the fact that the construction $L \mapsto M(L)$ is effectively calculable, it follows that whether a given regular language is FO[$<$]-definable is decidable. In general, for a class of regular languages, forming a variety is a necessary and sufficient condition to admit such equational characterization in terms of syntactic monoids. (This is a consequence from Eilenberg's variety theorem and \emph{Reiterman's theorem} \cite{Reiterman}.)
\end{remark}

\begin{remark}
 As mentioned above, Straubing's variant \cite{Straubing} of the variety theorem considers $\D$-varieties of regular langugages, while Eilenberg's original one considers ($\freemon$-) varieties of regular languages in particular. Originally, this extension of variety concept was to deal with classes of regular languages of natural interest that are not varieties in the sense of Eilenberg but their menbership still admits similar (relaxed) algebraic characterization in terms of syntactic monoids. See \cite{Straubing}. 
\end{remark}

We choose Straubing's variant \cite{Straubing} of Eilenberg's variety theorem as our starting point of reconsideration of Eilenberg's theory because Eilenberg's original theorem itself does not follow naturally from the duality theorem of Rhodes et al.\ \cite{Rhodes_Steinberg}, while so does Straubing's variant (cf.\ \S \ref{section var two}) that also subsumes Eilenberg's one in a certain precise sense. This is simply because the definition of varieties of regular languages and that of pseudo-varieties of finite monoids are not symmetric; in Straubing's variant, pseudo-varieties of finite monoids were replaced with \emph{$\D$-varieties of finite stamps} that are defined so as to be symmetric with $\D$-varieties of regular languages. 

Informally, \emph{finite stamps} are finite monoids with fixed generators; formally, they are defined as follows:

\begin{definition}[Finite stamps]
\label{finite stamps}
 A \emph{finite stamp} over an alphabet $A$ is a surjective homomorphism $s: A^* \twoheadrightarrow M$ onto a finite monoid $M$. 
\end{definition}

\noindent
Given two finite stamps $s: A^* \twoheadrightarrow M$ and $t: A^* \twoheadrightarrow N$ over the same alphabet $A$, the \emph{product} of $s$ and $t$ is defined as $u: A^* \twoheadrightarrow U$, where $U$ is the image of the pairing homomorphism $\langle s,t \rangle: A^* \rightarrow M \times N$ and $u$ is the canonical factor of $\langle s,t \rangle$ through $U$. We denote by $s*t$ the product of $s$ and $t$. 
Finally, we denote by $1_A: A^* \twoheadrightarrow 1$ the stamp onto the trivial singleton monoid $1$.

\begin{definition}[$\D$-variety of finite stamps]
A class $\vv$ of finite stamps is called a \emph{$\D$-variety of finite stamps} if, denoting $\vv_A := \{s: A^* \twoheadrightarrow M \mid s \in \vv\}$, the following properties hold:
\begin{description}
  \item[M1)] $1_A: A^* \twoheadrightarrow 1$ is in $\vv_A$ for every alphabet $A$;
  \item[M2)] if $s:A^* \twoheadrightarrow M, t: A^* \twoheadrightarrow N \in \vv_A$, then $s*t \in \vv_A$;
  \item[M3)] if $s: A^* \twoheadrightarrow M \in \vv_A$ and $h: M \twoheadrightarrow N$ is surjective, then $h \circ s \in \vv_A$; \vspace{0.1cm}
  \item[M4)] if $t:B^* \twoheadrightarrow N \in \vv_B$ and $s: A^* \twoheadrightarrow M$ is such that $j \circ s = t \circ f$ where $j: M \hookrightarrow N$ is an embedding, and $f:A^* \rightarrow B^*$ is in $\D$, then $s \in \vv_A$. 
\end{description}
A class $\vv_A$ of finite stamps over $A$ is called a \emph{local variety of finite stamps} if it satisfies {\bf M1)}, {\bf M2)} and {\bf M3)}. 
\end{definition}

\begin{remark}
There is a canonical bijective correspondence between pseudo-varieties of finite monoids in the original sense of Eilenberg \cite{Eilenberg} and $\D$-varieties of finite stamps with $\D=\freemon$ in particular as discussed by Straubing (cf.\ \cite{Straubing}). In this sense, $\D$-varieties of finite stamps subsume Eilenberg's pseudo-varieties of finite monoids, as in the case of $\D$-varieties of regular languages. 
\end{remark}

There is another good reason to choose Straubing's variety theorem as our starting point: That is, as formally recalled below, while Straubing's variety theorem itself \cite{Straubing} concerns a bijective correspondence between ($\D$-varieties of) regular languages and finite stamps, this correspondence can be naturally extended so that suitable classes of DFAs are involved as well (through taking recognizing languages and transformation monoids of DFAs); this extension was originally investigated by Chaubard et al.\ \cite{Chaubard_Pin_Straubing} in their study on wreath products of $\D$-varieties. Since in Eilenberg's variety theory classification of DFAs has been of central concern in itself and also fruitful auxiliary step in classification of regular languages and finite monoids, it is essential to involve the variety theorem of Chaubard et al.\ \cite{Chaubard_Pin_Straubing} in the reinterpretation of Eilenberg theory. Starting from Straubing's variant of Eilenberg's variety theorem makes it straightforward to do this in a transparent way. 

In the variety theorem of Chaubard et al.\ \cite{Chaubard_Pin_Straubing}, the concept of \emph{$\D$-varieties of finite actions} (= DFAs in our terminology) was introduced so that they correspond to those of regular languages and finite stamps. Formally, finite actions are defined as follows: 

\begin{definition}[Finite actions]
 A \emph{finite action} over an alphabet $A$ is a map $s: S \times A^* \rightarrow S$ such that $S$ is a finite set; and for every $u, v \in A^*$ and $\xi \in S$, we have $s (s(\xi,u), v) = s (\xi, uv)$ and $s(\xi,\varepsilon) = \xi$. 
\end{definition}
\noindent
As in the case of DFAs, we write $s(\xi,u) = \xi \cdot u$ and call the set $S$ as the \emph{set of states} (or \emph{state set}) and the map $s: S \times A^* \rightarrow S$ as the \emph{transition function}. 

Given two finite actions $s: S \times A^* \rightarrow S, t: T \times A^* \rightarrow T$ over the same alphabet $A$, the \emph{product} of $s$ and $t$ is the finite action whose set of states is $S \times T$; and the transition function is given by $S \times T \times A^* \ni (\xi, \eta, u) \mapsto (\xi \cdot u, \eta \cdot u) \in S \times T$. We denote this finite action by $s \times t$. 
Also, we say that $t$ is a \emph{subaction} of $s$ if $T \subseteq S$ and $\xi \cdot u \in T$ for every $\xi \in T$ and $u \in A^*$; we say that $t$ is a \emph{quotient} of $s$ if there exists a surjection $p_*: S \twoheadrightarrow T$ such that $ p_* (\xi \cdot u) = (p_* \xi) \cdot u$ for every $\xi \in S$ and $u \in A^*$; and also, for a finite action $s: S \times B^* \rightarrow S$ and a homomorphism $f: A^* \rightarrow B^*$, we denote by $f^*s$ the finite action $s \circ (\id_S \times f): S \times A^* \rightarrow S \times B^* \rightarrow S$. Finally, we say that a finite action $s: S \times A^* \rightarrow S$ is \emph{trivial} if $\xi \cdot u = \xi$ for every $u \in A^*$ and $\xi \in S$. 

\begin{definition}[$\D$-variety of finite actions]
A class $\va$ of finite actions is called a \emph{$\D$-variety of finite actions} if, denoting $\va_A := \{s: A^* \times S \rightarrow S \mid s \in \va\}$, the following properties hold: 
\begin{description}
  \item[D1)] all trivial finite actions are in $\va$;
  \item[D2)] if $s:S \times A^* \rightarrow S \in \va_A$ and $t: T \times A^* \rightarrow T \in \va_A$, then $s \times t \in \va_A$;
  \item[D3)] if $s: S \times  A^* \rightarrow S \in \va_A$ and $t$ is a subaction of $s$, then $t \in \va_A$; 
  \item[D4)] if $s: S \times  A^* \rightarrow S \in \va_A$ and $t$ is a quotient of $s$, then $t \in \va_A$; 
  \item[D5)] if $s: S \times B^* \rightarrow S \in \va_B$ and $f: A^* \rightarrow B^* \in \D$, then $f^* s: S \times Y \rightarrow S \in \va_A$. 
\end{description}
A class $\va_A$ of finite actions over $A$ is called a \emph{local variety of finite actions} if it satisfies {\bf D1)}, {\bf D2)}, {\bf D3)} and {\bf D4)}. 
\end{definition}

Now the variety theorems of Straubing \cite{Straubing} and Chaubard et al.\ \cite{Chaubard_Pin_Straubing} are stated as follows: Let $\varreg, \varmon, \vardfa$ denote respectively the lattices consisting of (i) $\D$-varieties of regular languages; (ii) those of finite stamps; and (iii) those of finite actions, where the order is given by the inclusion of $\D$-varieties. 

\begin{theorem}[Straubing, \cite{Straubing}; Chaubard, Pin, Straubing, \cite{Chaubard_Pin_Straubing}]
\label{the classical variety theorems}
 There are canonical isomorphisms $\varreg \simeq \varmon \simeq \vardfa$ of lattices.
\end{theorem}

\begin{remark}
When $\D=\freemon$, the isomorphism $\varreg \simeq \varmon$ implies the original Eilenberg variety theorem. 
\end{remark}

\subsection{The variety theorems via duality theorems}
\label{section var two}
\noindent
We proceed to our proof of the isomorphism $\varreg \simeq \varmon \simeq \vardfa$ that explicitly uses Theorem 8.4.10 \cite{Rhodes_Steinberg} and Theorem \ref{duality}, \S \ref{s4}. For this purpose, we construct distinct but isomorphic lattices $\varreg'$ ($\simeq \varreg$), $\varmon'$ ($\simeq \varmon$) and $\vardfa'$ ($\simeq \vardfa$) that consist of certain functors from $\D$ to $\bialg^{op}$, $\profmon$, $\semigalois^{op}$ respectively. (Here $\bialg$ denotes the category of bialgebras over $\field_2$ and bialgebra homomorphisms, cf.\ \cite{Rhodes_Steinberg}.) The isomorphism $\varreg' \simeq \varmon' \simeq \vardfa'$ is directly proved by the equivalence $\bialg^{op} \simeq \profmon \simeq \semigalois^{op}$, from which the target isomorphism $\varreg \simeq \varmon \simeq \vardfa$ follows; one can also see that this construction of isomorphism coincides with the original one given in \cite{Straubing,Chaubard_Pin_Straubing} (cf.\ Remark \ref{application}).

Here we shall ommit the detail of proofs and just sketch necessary arguments in order to make it easy to overview quickly the essential point: The detailed proofs will just bother the reader since they are entirely straightforward. The only point to which the reader needs to pay attention is which structures replace the classical structures--- $\D$-varieties, and how; in particular, among others, local varieties of finite actions are equivalently replaced with semi-galois categories with finitely generated fundamental monoids (cf.\ Lemma \ref{char of local varieties of finite actions}). 

Now we start our argument, beginning with local varieties. Let $A$ be an alphabet. Denote by $\varreg_A, \varmon_A$ and $\vardfa_A$ the lattices of local varieties of regular languages, finite stamps, and finite actions over $A$ respectively, where the order is given by the inclusion of local varieties. These lattices are canonically isomorphic to certain lattices $\varreg_A', \varmon_A', \vardfa_A'$ made from bialgebras over $\field_2$, profinite monoids and semi-galois categories; and using this fact, the isomorphism $\varreg \simeq \varmon \simeq \vardfa$ is proved as sketched just above. 

We start with the most straightforward one, that is, the lattice $\varmon_A$. Let $\varmon_A'$ be the lattice consisting of (isomorphism classes of) quotients $\pi: \widehat{A^*} \twoheadrightarrow M$ of the free profinite monoid $\widehat{A^*}$ in $\profmon$. Here, two quotients $\pi: \widehat{A^*} \twoheadrightarrow M, \pi': \widehat{A^*} \twoheadrightarrow M'$ are ordered $\pi \leq \pi'$ if and only if there is a surjective homomorphism $\lambda: M' \twoheadrightarrow M$ such that $\pi = \lambda \circ \pi'$. Then:

\begin{proposition}[More generally, see Chen and Urbat \cite{Chen_Urbat}]
\label{characterization of varmon}
 There is a canonical isomorphism $\varmon_A \simeq \varmon_A'$ of lattices. 
\end{proposition}

\noindent
The isomorphism $\varmon_A \rightarrow \varmon_A'$ is given by taking inverse limits of local varieties $\va_A$ of finite stamps over $A$. (We regard a local variety of finite stamps $\va_A$ as an inverse system of finite monoids by ordering $s \leq s'$ for $s:A^* \twoheadrightarrow N_s, s': A^* \twoheadrightarrow N_{s'} \in \va_A$ if and only if there exists a surjective homomorphism $\lambda: N_{s'} \twoheadrightarrow N_s$ such that $s = \lambda \circ s'$.) 

A corresponding representation for the lattice $\varreg_A$ was given by Rhodes et al.\ \cite{Rhodes_Steinberg}. We recall only two necessary facts from their work; for more detail, the reader is referred to \S 8.4 \cite{Rhodes_Steinberg}. Let $\reg(A)$ be the Boolean algebra of all regular languages over $A$. Then, they pointed out the facts that:
\begin{enumerate}
 \item \emph{$\reg(A)$ canonically admits a structure of bialgebra over $\field_2$} (e.g.\ \cite{Abe}); and \vspace{0.1cm}
 \item \emph{A class $\V_A \subseteq \reg(A)$ of regular languages over $A$ is a local variety of regular languages if and only if $\V_A$ is a sub-bialgebra of $\reg(A)$.}
\end{enumerate}
Generally, Boolean algebras are equivalent to Boolean rings whose summation is given by the symmetric difference $L \oplus R := (L \backslash R) \cup (R \backslash L)$; and thus, can be regarded also as vector spaces over $\field_2$. The bialgebra structure of $\reg(A)$ is the one with respect to this vector-space structure. The above two facts means that (1) $\reg(A)$ is an object of $\bialg$; and (2) if we denote by $\varreg_A'$ the lattice consisting of (isomorphism classes of) sub-bialgebras of $\reg(A)$ in $\bialg$, then it represents the target lattice $\varreg_A$\footnote{To represent local varieties of regular languages by abstract algebraic structures, one can also use comonoids in the category $\bool$ of Boolean algebras as they are equivalent to bialgebras over $\field_2$, cf.\ \S 8.4 \cite{Rhodes_Steinberg}.}:
\begin{proposition}[Rhodes and Steinberg, \S 8.4, \cite{Rhodes_Steinberg}]
\label{characterization of varreg}
 There is a canonical isomorphism $\varreg_A \simeq \varreg_A'$ of lattices. 
\end{proposition}

Finally, let $\va_A$ be any class of finite actions over an alphabet $A$; and define a category $\dfa(\va_A)$ to be the full subcategory of $A\mathchar`-\dfa$ (cf.\ \S \ref{s2}) whose objects are finite actions that belong to $\va_A$. (Finite actions are identified with DFAs.) Also, denote by $\F_{\va_A}: \dfa(\va_A) \rightarrow \fsets$ the restriction of the forgetful functor $\F_A: A\mathchar`-\dfa \rightarrow \fsets$ to the subcategory $\dfa(\va_A) \subseteq A\mathchar`-\dfa$. Then:
\begin{lemma}
\label{local variety of finite actions is semigalois}
 A class $\va_A$ of finite actions over $A$ is a local variety of finite actions over $A$ if and only if $\langle \dfa(\va_A), \F_{\va_A} \rangle$ is a semi-galois subcategory of $\langle A\mathchar`-\dfa, \F_A \rangle$. 
\end{lemma}

\begin{remark}
 Here, by a \emph{semi-galois subcategory} (or by an \emph{embedding}), we mean an arrow of semi-galois categories $[A,\sigma]: \sgc \hookrightarrow \langle \C',\F' \rangle$ such that $A: \C \rightarrow \C'$ is fully faithful. It is not difficult to see that this class of arrows among semi-galois categories exactly corresponds to the class of \emph{surjective} homomorphisms $\pi_1(\C',\F') \twoheadrightarrow \pi_1(\C,\F)$ among profinite monoids under the duality $\semigalois^{op} \simeq \profmon$. In what follows, \emph{subobjects} $\sgc$ of a semi-galois category $\langle \C',\F' \rangle$ mean those which have embeddings $\sgc \hookrightarrow \langle \C',\F' \rangle$. (Thus our class of subobjects in $\semigalois$ is more restrictive than the usual meaning of ``subobjects'' in a category; more formally, our subobjects are the effective monomorphisms in $\semigalois$.)
\end{remark}

In addition to Lemma \ref{local variety of finite actions is semigalois}, we have a fully abstract characterization of those semi-galois categories which appear from some local varieties  $\va_A$ of finite actions.

\begin{lemma}
\label{char of local varieties of finite actions}
 A semi-galois category $\sgc$ is equivalent to $\langle \dfa(\va_A), \F_{\va_A} \rangle$ for some local variety $\va_A$ over an alphabet $A$ if and only if $\pi_1(\C,\F)$ is topologically generated by $|A|$ elements. 
\end{lemma}

Now, we denote by $\vardfa'_A$ the lattice of (isomorphism classes of) subobjects of the semi-galois category $\langle A\mathchar`-\dfa, \F_A \rangle$ in $\semigalois$. By mapping $\va_A \mapsto \langle \dfa(\va_A), \F_{\va_A} \rangle$, we get:
\begin{proposition}
\label{characterization of vardfa}
 There is a canonical isomorphism $\vardfa_A \simeq \vardfa_A'$ of lattices. 
\end{proposition}

To proceed to the proof of $\varreg \simeq \varmon \simeq \vardfa$, we recall the idea of Rhodes et al.\ \cite{Rhodes_Steinberg} of regarding varieties of regular languages (= $\freemon$-varieties) as suitable functors from $\freemon$ to $\bialg^{op}$. Since this reinterpretation can be applied to arbitrary $\D$-varieties of regular languages, it is rephrased in our terminologies here.

Let $\V$ be a $\D$-variety of regular languages, and put $\V_A := \{ L \in \reg(A) \mid L \in \V \}$ for each alphabet $A$. Then, by definition, each $\V_A$ is clearly a local variety of regular languages over $A$ (cf.\ Definition 8). Thus, by Proposition \ref{characterization of varreg}, each $\V_A$ is a sub-bialgebra of the bialgebra $\reg(A)$; in particular, $\V_A \in \bialg$. Moreover, the assignment $\D \ni A^* \mapsto \V_A \in \bialg$ on objects extends to a functor, denoted $\V: \D \rightarrow \bialg^{op}$, by assigning to each $f: A^* \rightarrow B^* \in \D$ the inverse map $f^{-1}: \V_B \rightarrow \V_A$. (Note that this is well-defined because $\V$ satisfies the axiom ${\bf R4}$, namely, for each $L \in \V_B$ the inverse image $f^{-1}L \in \reg(A)$ belongs to $\V_A$.) Thus, each $\D$-variety $\V$ of regular languages defines a functor $\V: \D \rightarrow \bialg^{op}$ in a canonical way. 

In particular, let $\reg$ be the $\D$-variety of regular languages such that $\reg_A = \reg(A)$, i.e.\ the $\D$-variety consisting of all regular languages. Then this $\D$-variety of regular languages also defines a functor denoted $\reg: \D \rightarrow \bialg^{op}$; and other functors $\V: \D \rightarrow \bialg^{op}$ induced from $\D$-varieties are all subfunctors of this functor $\reg: \D \rightarrow \bialg^{op}$. Conversely, every subfunctors of the functor $\reg: \D \rightarrow \bialg^{op}$ arise in this way; in other words, $\D$-varieties of regular languages bijectively correspond to (isomorphism classes of) subfunctors of $\reg: \D \rightarrow \bialg^{op}$. In summary, if we denote by $\varreg'$ the lattice consisting of (isomorphism classes of) subfunctors of $\reg: \D \rightarrow \bialg^{op}$, then this argument concludes that:

\begin{theorem}[Rhodes and Steinberg, \cite{Rhodes_Steinberg}] 
There is a canonical isomorphism $\varreg \simeq \varreg'$ of lattices.
\end{theorem}

By the same construction, this argument applies to the lattices $\varmon$ and $\vardfa$ as well, which we now sketch briefly. First, let $FC: \D \rightarrow \profmon$ be the functor that assigns to each $A^*$ its free profinite completion $\widehat{A^*}$; and to each arrow $f: A^* \rightarrow B^* \in \D$ the canonically induced homomorphism $\hat{f}: \widehat{A^*} \rightarrow \widehat{B^*}$. Now, denote by $\varmon'$ the lattice consisting of (isomorphism classes of) quotients of the functor $FC: \D \rightarrow \profmon$ (that is ordered in the same way as $\varmon_A'$). Using Proposition \ref{characterization of varmon} instead of Proposition \ref{characterization of varreg} in the above argument, we obtain:
\begin{theorem}
 There is a canonical isomorphism $\varmon \simeq \varmon'$ of lattices.
\end{theorem}

Similarly, let $DFA: \D \rightarrow \semigalois^{op}$ denote the functor that assigns to each $A^*$ the semi-galois category $\langle A\mathchar`-\dfa, \F_A \rangle$; and to each $f: A^* \rightarrow B^* \in \D$ the arrow $\rep(\hat{f}): \langle B\mathchar`-\dfa, \F_B \rangle \rightarrow \langle A\mathchar`-\dfa,\F_A \rangle$ in $\semigalois$. (Here, we identify $\langle A\mathchar`-\dfa, \F_A \rangle$ with $\langle \widehat{A^*}\mathchar`-\fsets, \F_{\widehat{A^*}} \rangle$; and the arrow $\rep(\hat{f})$ is the one defined in \S \ref{s4}.) Denote by $\vardfa'$ the lattice consisting of (isomorphism classes of) subfunctors of $DFA: \D \rightarrow \semigalois^{op}$. Using Proposition \ref{characterization of vardfa} instead of Proposition \ref{characterization of varreg} in the above argument, we obtain:
\begin{theorem}
 There is a canonical isomorphism $\vardfa \simeq \vardfa'$ of lattices.
\end{theorem}

Finally, in order to deduce the target isomorphisms $\varreg \simeq \varmon \simeq \vardfa$, it suffices to prove the isomorphisms $\varreg' \simeq \varmon' \simeq \vardfa'$. Because the latter lattices $\varreg', \varmon', \vardfa'$ consist of functors from $\D$ to $\bialg^{op}, \profmon$ and $\semigalois^{op}$ respectively, their isomorphisms can be proved as direct consequences from the dualities $\bialg^{op} \simeq \profmon$ (Theorem 8.4.10, \cite{Rhodes_Steinberg}) of Rhodes et al.; and $\profmon \simeq \semigalois^{op}$ (Theorem \ref{duality}, \S \ref{s4}); and also the fact that the bialgebra $\reg(A) \in \bialg$ corresponds to $\widehat{A^*} \in \profmon$ under the duality $\bialg^{op} \simeq \profmon$ (cf.\ \S 8.4, \cite{Rhodes_Steinberg}); and that $\widehat{A^*} \in \profmon$ corresponds to $\langle A\mathchar`-\dfa,\F_A \rangle \in \semigalois$ under the duality $\profmon \simeq \semigalois^{op}$. This argument completes the proof of the isomorphisms $\varreg \simeq \varmon \simeq \vardfa$. 

\begin{remark}
The isomorphism $\varreg \simeq \varmon \simeq \vardfa$ constructed here coincides with the original constructions given in \cite{Straubing,Chaubard_Pin_Straubing} that use syntactic monoids of regular languages, transformation monoids of DFAs and recognized languages of DFAs. 
\end{remark}
\begin{remark}
\label{application}
In more elementary words, the above remark implies the following fact: Let $\V \in \varreg$ be a ($\freemon$-) variety of regular languages and denote by $\C_{\V_A}$ the full subcategory of $A\mathchar`-\dfa$ whose objects are DFAs accepting languages in $\V_A$ (with respect to every initial and final states). Also let $\F_{\V_A}: \C_{\V_A} \rightarrow \fsets$ be the restriction of the fiber functor $\F_A: A\mathchar`-\dfa \rightarrow \fsets$ onto $\C_{\V_A} \subseteq A\mathchar`-\dfa$. Then (I) $\langle \C_{\V_A}, \F_{\V_A} \rangle$ forms a semi-galois category; and (II) its fundamental monoid $\pi_1(\C_{\V_A},\F_{\V_A})$ is isomorphic to the relatively free profinite monoid $\overline{\Omega}_A \vv$ with respect to the pseudo-variety $\vv$ of finite monoids that corresponds to $\V$ under the isomorphism $\varreg \simeq \varmon$ of the variety theorem. 
\end{remark}

\section{Topos Representation}
\label{s6}
\noindent
As stated in Lemma \ref{char of local varieties of finite actions}, the class of semi-galois categories (with finitely generated fundamental monoids) is isomorphic to that of local varieties of finite actions; and this isomorphism was used in the reinterpretation of the variety theorems \cite{Straubing,Chaubard_Pin_Straubing} so that these theorems can be seen as a natural consequence of the duality theorems $\bialg^{op} \simeq \profmon \simeq \semigalois^{op}$ of abstract algebras. As discussed in \S 6, \cite{Uramoto16}, this abstract interpretation of the variety theorems gives us a reason to get interested in the general structure of semi-galois categories; in this relation, we give here a simple specification of the class of those topoi which are equivalent to semi-galois categories.

Technically speaking, we study a semi-galois categorical fragment of the duality between \emph{pretopoi} and \emph{coherent topoi} (cf.\ \S \ref{s6s1}). Since the class of semi-galois categories is a proper subclass of pretopoi, the class of coherent topoi dual to semi-galois categories must constitute a specific subclass of coherent topoi. In fact, by the fact that semi-galois categories are always of the form $\Cl_f M$ for some profinite monoids $M$, it is not difficult to see that the coherent topoi dual to semi-galois categories are exactly the classifying topoi $\Cl M$ of profinite monoids $M$. In this section, we provide a purely topos-theoretic characterization of this class of coherent topoi (i.e.\ the class of classifying topoi $\Cl M$ of profinite monoids $M$), independent of the concept of profinite monoids, as an application of the duality theorem between profinite monoids and semi-galois categories. To be specific, we see that a topos $\E$ is equivalent to $\Cl M$ for some profinite monoid $M$ if and only if $\E$ is (i) \emph{coherent}, (ii) \emph{noetherian}, and (iii) \emph{has a surjective coherent point} $p: \sets \rightarrow \E$ (\S \ref{s6s2}). Together with the result in \S \ref{s5}, this class of topoi is in a bijective correspondence with the class of local varieties of regular languages. 
\subsection{Coherent topoi and pretopoi}
\label{s6s1}
\noindent
For general references on the material in this section, the reader is refered to e.g.\ \cite{Johnstone,Johnstone_elephant,MacLane_Moerdijk}. For the sake of reader's convenience, we summarize here the fundamental properties and constructions concerning coherent topoi and pretopoi that we need in this paper. The reader who is familiar with coherent topoi can skip this subsection and directly go to the next subsection (\S \ref{s6s2}). 

Firstly recall that a Grothendieck topology $J$ on a category $\C$ is said to be \emph{of finite type} if $J$ is generated by finite covering families. To be more precise, for each $X \in \C$, every covering sieve $S \in J(X)$ contains a finite number of arrows $f_i: Y_i \rightarrow X$ ($i=1,\cdots,n$) that \emph{generate} $S$, i.e.\ every $g: Z \rightarrow X \in S$ can be written as $g = f_i \circ g'$ for some $i \in \{1,\cdots,n\}$ and $g': Z \rightarrow Y_i$ (cf.\ \cite{MacLane_Moerdijk}). \emph{Coherent topoi} are those topoi of sheaves over categories with finite limits and Grothendieck topologies of finite type. Formally, they are defined as follows:


\begin{definition}[coherent topos]
 A topos $\E$ is called a \emph{coherent topos} if there exist (i) a category $\C$ with finite limits and (ii) a Grothendieck topology $J$ on $\C$ of finite type such that $\E$ is equivalent to the sheaf topos $\Sh(\C,J)$ with respect to the site $(\C,J)$. 
\end{definition}

As we will recall soon, coherent topoi are dual in a certain precise sense to the class of categories so called \emph{pretopoi}. This correspondence from coherent topoi to pretopoi is given by taking \emph{coherent objects}, i.e.\ \emph{compact} and \emph{stable} objects. To define these concepts, recall that a family $\{f_\lambda: Y_\lambda \rightarrow X\}$ of arrows is called \emph{jointly epic} if, for every parallel arrows $g, h: X \rightarrow Z$, the equalities $g \circ f_\lambda = h \circ f_\lambda: Y_\lambda \rightarrow X \rightarrow Z$ for all $\lambda$ imply the identity $g=h$. 

\begin{definition}[compact object]
 An object $X \in \E$ is called \emph{compact} if, for every jointly epic family $\{f_\lambda: Y_\lambda \rightarrow X\}$ of arrows, some finite subfamily $\{f_{\lambda_1}, \cdots, f_{\lambda_n}\}$ of it is already jointly epic. 
\end{definition}

\begin{definition}[stable object]
 An object $X \in \E$ is called \emph{stable} if, for every arrows $f: Y \rightarrow X$ and $g: Z \rightarrow X$ from compact objects $Y, Z \in \E$, their pullback $Y \times_X Z$ is also compact. 
\end{definition}

\begin{definition}[coherent object]
 An object $X \in \E$ is \emph{coherent} if it is both compact and stable. 
\end{definition}

Given a topos $\E$ in general, we denote by $\E_{coh}$ the full subcategory of $\E$ that consists of coherent objects. In the case where $\E$ is a coherent topos, it is known that $\E_{coh} =: \C$ satisfies the following properties: 
\begin{description}
 \item[$P_1$)] $\C$ has finite limits;
 \item[$P_2$)] $\C$ has finite coproducts, which are \emph{disjoint} and \emph{universal};
 \item[$P_3$)] $\C$ has coequalizers of \emph{equivalence relations}, which are universal;
 \item[$P_4$)] every equivalence relation in $\C$ is \emph{effective};
 \item[$P_5$)] every epimorphism in $\C$ is a coequalizer.
\end{description}
\noindent
(See \cite{Johnstone,Johnstone_elephant} for italic terminologies here.) 

\begin{definition}[pretopos]
A category $\C$ is called a \emph{pretopos} if it satisfies these five axioms.  
\end{definition}

\begin{remark}[semi-galois categories are pretopoi]
 Let $\C$ be a semi-galois category; by definition, it has a fiber functor $\F: \C \rightarrow \fsets$ so that $\langle \C,\F \rangle$ is equivalent to $\langle \Cl_f M, \F_M \rangle$ with $M = \pi_1(\C,\F)$. Then, it is straightforward to see that $\C = \Cl_f M$ satisfies the above axioms, i.e.\ forms a pretopos. 
\end{remark}

As we briefly mentioned above, coherent topoi are dual to pretopoi under the construction $\E \mapsto \E_{coh}$ of categories of coherent objects. The inverse correspondence from pretopoi to coherent topoi is then given by taking sheaf topoi $\Sh(\C,J_\C)$ over pretopoi $\C$ with respect to the \emph{precanonical topologies} $J_\C$. 

\begin{definition}[precanonical topology]
 Let $\C$ be a pretopos. The \emph{precanonical topology} on $\C$ is defined as the Grothendieck topology $J_\C$ that consists of those sieves $S \in J_\C(X)$ for each $X \in \C$ which are generated by finite jointly epimorphic families of arrows. 
\end{definition}

\begin{remark}
In particular, note that the precanonical topology $J_\C$ on a pretopos $\C$ is clearly of finite type; thus by definition the sheaf topos $\Sh(\C,J_\C)$ is a coherent topos. In what follows, we shall denote $\Sh(\C,J_\C)$ simply by $\Sh(\C)$, whence the Grothendieck topology on $\C$ should be understood as the precanonical topology $J_\C$. 
\end{remark}

Now we have the correspondences from coherent topoi $\E$ to pretopoi $\E_{coh}$; and from pretopoi $\C$ to coherent topoi $\Sh(\C)$. A fundamental fact about coherent topoi and pretopoi is that these correspondences $\E \mapsto \E_{coh}$ and $\C \mapsto \Sh(\C)$ are in fact mutually inverses up to equivalence: 

\begin{theorem}[e.g.\ \cite{Johnstone}]
 Let $\E$ be a coherent topos and $\C$ be a pretopos. Then:
 \begin{enumerate}
  \item $\E$ is equivalent to $\Sh(\E_{coh})$; and
  \item $\C$ is equivalent to $(\Sh(\C))_{coh}$. 
 \end{enumerate}
\end{theorem}
\subsection{Classifying topoi of profinite monoids}
\label{s6s2}
\noindent
Because semi-galois categories are always pretopoi, the class of semi-galois categories must correspond to a specific subclass of coherent topoi under the correspondence between pretopoi and coherent topoi. In fact, as we will see later, semi-galois categories correspond exactly to coherent topoi of the form $\Cl M$ for some profinite monoids $M$. In this section we give a simple characterization of the class of coherent topoi $\E$ in the form $\Cl M$ for some profinite monoids $M$ in a purely topos-theoretic terminology, not explicitly mentioning to the concept of profinite monoids. This characterization of topoi $\Cl M$ is in a sense a topos-theoretic translation of the duality theorem between profinite monoids and semi-galois categories; and formally is stated as follows.

\begin{theorem}
\label{characterization of BM}
 Let $\E$ be a topos. Then we have the following:
 \begin{enumerate}
  \item the topos $\E$ is equivalent to $\Cl M$ for some profinite monoid $M$ if and only if $\E$ is (i) coherent, (ii) noetherian and (iii) has a surjective coherent point $p: \sets \rightarrow \E$;
  \item when this is the case, the inverse image functor $p^*: \E \rightarrow \sets$ restricts to $p^*: \E_{coh} \rightarrow \fsets$, and the pair $\langle \E_{coh}, p^* \rangle$ forms a semi-galois category; 
  \item the profinite monoid $M$ such that $\E \simeq \Cl M$ can be taken as $M = \pi_1(\E_{coh},p^*)$.
 \end{enumerate}
\end{theorem}

\begin{remark}[Noether-ness]
\label{Noether-ness}
 Here, a coherent topos $\E$ is called \emph{noetherian} if, for every coherent object $X \in \E_{coh}$, the lattice of subobjects of $X$ satisfies the ascending chain condition. This condition is in fact equivalent to saying that every compact object in $\E$ is coherent \cite{Johnstone}. We will use this remark in our proof of the if part of Theorem \ref{characterization of BM}.
\end{remark}

We first prove the only-if part of the claim 1. 

\begin{lemma}
\label{coherent objects in BM}
 The coherent objects in $\Cl M$ are exactly the finite $M$-sets. That is, $(\Cl M)_{coh} = \Cl_f M$. 
\end{lemma}
\begin{proof}
 It is clear that finite $M$-sets are coherent. So, conversely, we see that a coherent object $X \in \Cl M$ is finite. Let $X \in \Cl M$ be coherent. For each element $\xi \in S_X$, its orbit $\xi \cdot M \subseteq S_X$ defines a subobject $X_\xi \hookrightarrow X$. Since the $M$-action is continuous (with $S_X$ discrete) and $M$ is compact, it follows that $\xi \cdot M$ is a finite set, i.e.\ $X_\xi \in \Cl_f M$. By the coherence of $X$ and that $X = \bigcup X_\xi$, we have $X = X_{\xi_1} \cup X_{\xi_2} \cup \cdots \cup X_{\xi_n}$ for some finite points $\xi_i \in S_X$. Since $X_{\xi_i} \in \Cl_f M$, one obtains $X \in \Cl_f M$. 
\end{proof}

\begin{lemma}
\label{coherent surjective point}
 The topos $\Cl M$ has a coherent surjective point $p: \sets \rightarrow \Cl M$. 
\end{lemma}
\begin{proof}
 For later use we explicitly construct such a point $p$. On one hand, the inverse image functor $p^*: \Cl M \rightarrow \sets$ assigns to each $X \in \Cl M$ the underlying set $S_X$; and to each $f: X \rightarrow Y$ the corresponding map $f_*: S_X \rightarrow S_Y$. On the other hand, the direct image functor $p_*: \sets \rightarrow \Cl M$ assigns to each $S \in \sets$ the $M$-set $S^M$ consisting of continuous maps $\xi: M \rightarrow S$ with respect to the discrete topology on $S$; and to each map $f: S \rightarrow T$ the composition $S^M \ni \xi \mapsto f \circ \xi \in T^M$. The set $S^M$ is indeed an $M$-set, that is, equipped with the following right $M$-action:
\begin{eqnarray*}
 S^M \times M & \longrightarrow & S^M \\
 (\xi, m) & \longmapsto & \xi \cdot m : n \mapsto \xi (mn) 
\end{eqnarray*}
It is straightforward to see that $p^*: \Cl M \rightarrow \sets$ is a left adjoint to $p_*: \sets \rightarrow \Cl M$ and $p^*$ is left exact. That is, $p^*: \Cl M \rightarrow \sets$ and $p_*: \sets \rightarrow \Cl M$ define a point $p: \sets \rightarrow \Cl M$. By definition, $p^*$ reflects isomorphisms (i.e.\ $p$ is surjective), and by Lemma \ref{coherent objects in BM} and by $\sets_{coh} = \fsets$, $p^*$ preserves coherent objects (i.e.\ $p$ is coherent). 
\end{proof}

\begin{lemma}
\label{noetherian}
 The topos $\Cl M$ is noetherian. 
\end{lemma}
\begin{proof}
 Since coherent $M$-sets are finite by Lemma \ref{coherent objects in BM}, the lattice $\mathrm{Sub}(X)$ of subobjects of coherent $M$-sets $X$ is finite; and thus, clearly satisfy the ascending chain condition. That is, $\Cl M$ is noetherian. 
\end{proof}

\begin{theorem}
\label{coherence of BM}
 The topos $\Cl M$ is coherent, i.e.\ equivalent to the sheaf topos $\Sh(\Cl_f M, J_M)$ over the semi-galois category $\Cl_f M$ with respect to the precanonical topology $J_M$.
\end{theorem}

\begin{notation}
In what follows, given a semi-galois category $\sgc$ in general, we denote simply by $\Sh(\C)$ the sheaf topos over $\C$ with respect to the precanonical topology on $\C$. In particular, we shall denote $\Sh(\Cl_f M,J_M)$ by $\Sh(\Cl_f M)$. 
\end{notation}


\begin{proof}
This will be proved easily by using some axiomatic characterization of coherent topoi. Here, just in order to see the correspondence $\Cl M \leftrightarrow \Sh(\Cl_f M)$, we explicitly construct an equivalence. First the functor $\Phi: \Cl M \rightarrow \Sh(\Cl_f M)$ is defined naturally by $\Cl M \ni Y \mapsto \Hom_{\Cl M}(\ast,Y) \in \Sh(\Cl_f M)$. (Here notice that the presheaf $\Phi(Y) = \Hom_{\Cl M}(\ast, Y): \Cl_f M^{op} \rightarrow \sets$ defines actually a sheaf on $\Cl_f M$ with respect to the precanonical topology.) 

This functor $\Phi: \Cl M \rightarrow \Sh(\Cl_f M)$ is \emph{fully faithful}. To see the faithfulness of $\Phi$, let $f, g: Y \rightarrow Y'$ be $M$-equivariant maps between $M$-sets $Y,Y' \in \Cl M$ such that $\Phi(f) = \Phi(g)$. By $\Phi(f) = \Phi(g)$, it follows that $f |_{Y_\eta} = g |_{Y_\eta}$--- the restrictions of $f,g$ onto $Y_\eta \subseteq Y$--- for each $\eta \in S_Y$. But, as in the discussion at Lemma \ref{coherent objects in BM}, one also has $Y = \bigcup Y_\eta$ with $Y_\eta \in \Cl_f M$ for every $\eta \in S_Y$. Therefore, this implies $f = g$ and that $\Phi$ is faithful. To see the fullness of $\Phi$, let $\sigma: \Phi(Y) \Rightarrow \Phi(Y')$ be an arrow in $\Sh(\Cl_f M)$. Notice that we have a canonical isomorphism of sets $S_Y \simeq \Hom_{\Cl M} (M, Y) = \bigcup \Hom_{\Cl M} (M_\lambda, Y)$, where $M \twoheadrightarrow M_\lambda$ range over all finite quotients of $M$; and under this isomorphism, the components $\sigma_{M_\lambda}: \Hom_{\Cl M}(M_\lambda, Y) \rightarrow \Hom_{\Cl M}(M_{\lambda}, Y')$ of $\sigma: \Phi(Y) \Rightarrow \Phi(Y')$ define a map $\bar{\sigma}: S_Y \rightarrow S_Y'$, which is in fact $M$-equivariant. By construction, one can see that $\Phi(\bar{\sigma}) = \sigma: \Phi(Y) \Rightarrow \Phi(Y')$. Thus, $\Phi$ is full. 

To complete the proof, we need to show that $\Phi$ is also \emph{essentially surjective}; this is done by constructing an assignment $\Psi: \Sh(\Cl_f M) \rightarrow \Cl M$ on objects so that for each $P \in \Sh (\Cl_f M)$, one has an isomorphism $P \simeq \Phi \circ \Psi (P)$. 

For this purpose, take $P \in \Sh(\Cl_f M)$ arbitrarily. Then we define an $M$-set $\Psi(P) \in \Cl M$ as follows. The underlying set of $\Psi(P)$ is defined as $\colim_\lambda P(M_\lambda)$, where $\pi_\lambda: M \twoheadrightarrow M_\lambda$ range over all finite quotients of $M$ (and naturally form an inverse system in $\Cl_f M$ regarding $M_\lambda$ as finite $M$-sets). In what follows, the colimit diagram for $\colim_\lambda P(M_\lambda)$ is denoted by $u_\lambda: P(M_\lambda) \rightarrow \colim_\lambda P(M_\lambda)$. Here notice that since $P$ is a sheaf on $\Cl_f M$ with respect to $J_M$ and each $M_\lambda \twoheadrightarrow M_\mu$ is an epimorphism in $\Cl_f M$, the map $u_\lambda: P(M_\lambda) \rightarrow \colim P(M_\lambda)$ is injective. Thus, we can regard as $\Psi(P) = \bigcup P(M_\lambda)$. 

We prove the target isomorphism $P \simeq \Phi \circ \Psi (P)$ based on two sub-lemmas: 
\begin{lemma}
\label{fixpoint}
 Let $X \in \Cl M$ and $\pi_\lambda: M \twoheadrightarrow M_\lambda$ be a finite quotient. Also consider the congruence $\ker(\pi_\lambda) := \{(m,m') \in M \times M \mid \pi_\lambda(m) = \pi_\lambda(m')\}$. Moreover, define a subobject $X_\lambda \hookrightarrow X$ in $\Cl M$ so that the underlying set of $X_\lambda$ is given as follows:
\begin{eqnarray}
  S_\lambda &:=& \bigl\{ \xi \in S_X \mid \xi \cdot m = \xi \cdot m' \hspace{0.2cm} \forall (m,m') \in \ker(\pi_\lambda) \bigr\}. 
\end{eqnarray}
Then we have a canonical isomorphism of sets:
\begin{eqnarray}
 S_\lambda & \simeq & \Hom_{\Cl M} (M_\lambda, X).
\end{eqnarray}
\end{lemma}

\begin{notation}
 In what follows, for a finite quotient $\pi_\lambda: M \twoheadrightarrow M_\lambda$ and an element $m \in M$, we denote by $m_\lambda$ the element $\pi_\lambda(m) \in M_\lambda$. 
\end{notation}

\begin{proof}
 We first need to see that $S_\lambda \subseteq S_X$ is indeed closed under $M$-action, i.e.\ defines a subobject $X_\lambda \hookrightarrow X$ of $M$-sets. Let $\xi \in S_\lambda$ and $n \in M$. For each $(m,m') \in \ker(\pi_\lambda)$, we have $(nm, nm') \in \ker(\pi_\lambda)$ because $\ker(\pi_\lambda)$ is a (left) congruence. Thus, since $\xi \in S_\lambda$, we have $(\xi\cdot n) \cdot m = \xi \cdot nm = \xi \cdot nm' = (\xi \cdot n) \cdot m$. This means that, because $(m,m') \in \ker(\pi_\lambda)$ is now arbitrary, we proved $\xi \cdot n \in S_\lambda$. That is, the set $S_\lambda \subseteq S_X$ is closed under the $M$-action. 

The target isomorphism $\phi: S_\lambda \rightarrow \Hom_{\Cl M} (M_\lambda,X)$ is given as follows. For each $\xi \in S_\lambda$, define a map $\phi(\xi): M_\lambda \ni m_\lambda \mapsto \xi \cdot m \in X$. First note that $\phi(\xi): M_\lambda \rightarrow X$ is well-defined as a map of sets: That is, if $m_\lambda = m'_\lambda$ for $m,m' \in M$ (i.e.\ $(m,m') \in \ker(\pi_\lambda)$), we have $\xi \cdot m = \xi \cdot m'$ by $\xi \in S_\lambda$. Secondly, this map $\phi(\xi): M_\lambda \rightarrow X$ is $M$-equivariant: 
\begin{eqnarray*}
  \phi(\xi) (m_\lambda \cdot n) & = & \phi(\xi) ((mn)_\lambda) \\
  & = & \xi \cdot mn \\
  & = & (\xi \cdot m) \cdot n \\
  & = & (\phi(\xi) (m_\lambda) ) \cdot n. 
\end{eqnarray*}
Therefore, the map $\phi: S_\lambda \rightarrow \Hom_{\Cl M} (M_\lambda, X)$ is well-defined. Finally, this map $\phi$ is in fact bijective. For its inverse $\psi: \Hom_{\Cl M}(M_\lambda, X) \rightarrow S_\lambda$ is gien as follows: For each $f \in \Hom_{\Cl M}(M_\lambda,X)$, set $\psi(f) := f(1_\lambda) \in S_X$, where $1_\lambda = \pi_\lambda(1)$ is the identity of $M_\lambda$. Then $\psi(f)$ is indeed in $S_\lambda$. It is straightforward to see that $\psi$ is the inverse of $\phi$. 
\end{proof}

\begin{lemma}
\label{galois case of essential surjectivity}
 Let $P \in \Sh(\Cl_f M)$. Then $P(M_\lambda)$ is isomorphic to $\Psi(P)_\lambda$ under $u_\lambda: P(M_\lambda) \rightarrow \bigcup P(M_\lambda) = \Psi(P)$. More explicitly, $P(M_\lambda) \subseteq \Psi(P)$ is characterized as the following subset:
\begin{eqnarray}
\label{difficult part}
 P(M_\lambda) &=& \biggl\{ [s,\mu] \in \Psi(P) \mid [s,\mu] \cdot m = [s, \mu] \cdot m' \hspace{0.2cm} \forall (m,m') \in \ker(\pi_\lambda) \biggr\}. 
\end{eqnarray}
That is, combining with Lemma \ref{fixpoint}, we have a canonical isomorphism:
\begin{eqnarray*}
   P(M_\lambda) & \simeq & (\Phi \circ \Psi (P)) (M_\lambda).
\end{eqnarray*}
\end{lemma}

\begin{proof}
 The difficult part of this proof is to see that $\Psi(P)_\lambda$ (i.e.\ the right-hand side of the equation (\ref{difficult part})) is in the image of $u_\lambda: P(M_\lambda) \hookrightarrow \Psi(P)$; the inverse inclusion is easy. We use the similar technique to the one proving Theorem 1, \S 9, Chap.\ III \cite{MacLane_Moerdijk}. 

Let $[s,\mu] \in \Psi(P)_\lambda$; we want to prove that $[s,\mu] = [t, \lambda]$ for some $t \in P(M_\lambda)$. Since the inverse system consisting of finite quotients $\pi_\lambda: M \twoheadrightarrow M_\lambda$ is cofiltered, one may assume that we have $\rho^\mu_\lambda:M_\mu \twoheadrightarrow M_\lambda$ such that $\pi_\lambda = \rho^\mu_\lambda \circ \pi_\mu$. Also, since $P$ is now a sheaf on $\Cl_f M$ with respect to $J_M$ and $\rho^\mu_\lambda: M_\mu \twoheadrightarrow M_\lambda$ is a (single) covering of $M_\lambda \in \Cl_f M$, it suffices to see that, for parallel arrows $f, g: Z \rightarrow M_\mu$ from each $Z \in \Cl_f M$ such that $\rho^\mu_\lambda \circ f = \rho^\mu_\lambda \circ g$, we have $P(f)(s) = P(g)(s)$. (That is, the single element $s \in P(M_\mu)$ defines a matching family on $M_\lambda$ with respect to the covering $\rho^\mu_\lambda: M_\mu \twoheadrightarrow M_\lambda$.) Moreover, using Proposition \ref{cofilteredness of galois objects}, \S \ref{s3s2} (claiming that galois objects are cofinal in $\Cl_f M$) and by the fact that $P$ is a sheaf, it is sufficient to consider the case where $Z$ is a galois object in $\Cl_f M$, i.e.\ of the form $M_\kappa$. 

Let $f, g: M_\kappa \rightarrow M_\mu$ be $M$-equivariant maps such that $\rho^\mu_\lambda \circ f = \rho^\mu_\lambda \circ g$. Since $f, g$ are $M$-equivariant, they are determined by $m_\mu:=f(1_\kappa)$ and $m'_\mu:= g(1_\kappa)$ (for some $m, m' \in M$). Then note that $\rho^\mu_\lambda \circ f= \rho^\mu_\lambda \circ g$ implies that $(m,m') \in \ker(\pi_\lambda)$; and thus we have $[s,\mu] \cdot m = [s,\mu] \cdot m'$ by the assumption that $[s,\mu] \in \Psi(P)_\lambda$. This means that $[P(f)(s), \kappa] = [P(g)(s), \kappa]$. Since $u_\kappa: P(M_\kappa) \ni t \mapsto [t, \kappa] \in \Psi(P)$ is injective, one gets the desired equality $P(f)(s) = P(g)(s)$. 

Finally, we saw that $s \in P(M_\mu)$ defines a matching family on $M_\lambda$ with respect to the single covering $\rho^\mu_\lambda: M_\mu \twoheadrightarrow M_\lambda$; and since $P$ is a sheaf, this matching family $s$ has an amalgamation $t \in P(M_\lambda)$. This means that we have $[s,\mu] = [t, \lambda]$ with $t \in P(M_\lambda)$. Therefore, we can now conclude that $\Psi(P)_\lambda$ is in the image of $u_\lambda : P(M_\lambda) \hookrightarrow \Psi(P)$. This completes the proof. 
\end{proof}

We return to the proof of Theorem \ref{coherence of BM}. To complete the proof, we need to show a natural isomorphism $P \simeq \Phi \circ \Psi (P)$. By Lemma \ref{galois case of essential surjectivity}, we know that these two sheaves $P$ and $P':= \Phi \circ \Psi(P): \Cl_f M^{op} \rightarrow \sets$ are (naturally) isomorphic on galois objects $M_\lambda$ in $\Cl_f M$. However, since galois objects are cofinal in $\Cl_f M$ (Proposition \ref{cofilteredness of galois objects}) and since $P, P'$ are sheaves on $\Cl_f M$ with respect to $J_M$, it is straightforward to see that $P$ and $P'$ are naturally isomorphic. Consequently, the functor $\Phi: \Cl M \rightarrow \Sh(\Cl_f M)$ is fully faithful and essentially surjective, i.e.\ an euqivalence of categories, with $\Psi: \Sh(\Cl_f M) \rightarrow \Cl M$ giving its inverse. This concludes that $\Cl M$ is a coherent topos, with the semi-galois category $\Cl_f M \subseteq \Cl M$ being its defining pretopos. This completes the proof of Theorem \ref{coherence of BM}. 
\end{proof}

So far we have proved the only-if part of the claim 1 of Theorem \ref{characterization of BM}: That is, the topos $\Cl M$ is indeed (i) coherent, (ii) noetherian, and (iii) has a surjective coherent point $p: \sets \rightarrow \Cl M$. The rest of this section is devoted to a proof of its inverse, which heavily relies on the duality theorem between profinite monoids and semi-galois categories: 

\begin{theorem}
\label{the if part}
 Let $\E$ be a topos that is (i) coherent, (ii) noetherian, and (iii) has a surjective coherent point $p: \sets \rightarrow \E$. Then, the pair $\langle \E_{coh}, p^* \rangle$ forms a semi-galois category; and if we put $M_p :=\pi_1(\E_{coh}, p^*)$, then $\E$ is equivalent to $\Cl M_p$. 
\end{theorem}

We start from showing the first claim: 

\begin{lemma}
\label{coherent objects form a semigalois category}
 The pair $\langle \E_{coh}, p^* \rangle$ forms a semi-galois category. 
\end{lemma}
\begin{proof}
 Firstly recall that, since $\sets_{coh}=\fsets$ and the point $p: \sets \rightarrow \E$ is now coherent, the inverse image functor $p^*: \E \rightarrow \sets$ restricts to coherent objects $p^*: \E_{coh} \rightarrow \fsets$. We see that $\langle \E_{coh}, p^* \rangle$ satisfies the axiom of semi-galois categories. 

By general facts about coherent objects in a coherent topos, $\E_{coh}$ has finite limits and admits epi-mono factorizations of arrows. Also, since $p: \sets \rightarrow \E$ is a geometric morphism, $p^*: \E_{coh} \rightarrow \fsets$ is exact. Furthermore, since $p$ is surjective, $p^*$ reflects isomorphisms. Therefore, it is sufficient to prove only that $\E_{coh}$ has finite pushouts. 

To see this, we use the assumption that $\E$ is noetherian, whence the compact objects are the coherent objects (cf.\ Remark \ref{Noether-ness}). Let $X, Y, Z \in \E_{coh}$ be coherent objects in $\E$; and consider the following pushout diagram:
\begin{eqnarray*}
  \xymatrix{
    X \ar[r]^f \ar[d]_g & Z \ar[d]^k \\
    Y \ar[r]_h          & Y \sqcup_X Z 
}.
\end{eqnarray*}
Since $Y, Z$ are compact, so is the pushout $Y \sqcup_X Z$ because $Y \sqcup_X Z$ is an epimorphic image of the compact object $Y \sqcup Z$. However, since $\E$ is now noetherian so that every compact object is coherent, it follows that $Y \sqcup_X Z$ is coherent. This means that $\E_{coh}$ is closed under finite pushouts. Therefore $\langle \E_{coh}, p^* \rangle$ forms a semi-galois category. 
\end{proof}

\begin{proof}
Now we complete the proof of Theorem \ref{the if part}. Let $M_p$ be the fundamental monoid $\pi_1(\E_{coh},p^*)$ of this semi-galois category $\langle \E_{coh},p^* \rangle$. Since $\E$ is a coherent topos, one generally has an equivalence $\E \simeq \Sh(\E_{coh})$ of topoi. However, by the duality theorem between profinite monoids and semi-galois categories, we have an equivalence $\E_{coh} \simeq \Cl_f M_p$. Now, applying the representation $\Cl M \simeq \Sh(\Cl_f M)$ for the profinite monoid $M_p$, we consequently obtain the desired representation of topoi, $\E \simeq \Cl M_p$. Summary:
\begin{eqnarray*}
   \E & \simeq & \Sh(\E_{coh})       \hspace{1.3cm} \textrm{(see \cite{Johnstone_elephant}.)}\\
      & \simeq & \Sh(\Cl_f M_p) \hspace{0.5cm} \textrm{($\because$ Theorem \ref{duality}, \S \ref{s4} and Lemma \ref{coherent objects form a semigalois category} above.)}\\
      & \simeq & \Cl M_p. \hspace{1.3cm} \textrm{($\because$ Theorem \ref{coherence of BM}.)}
\end{eqnarray*}
This complete the proof.
\end{proof}

\section{Discussions}
\label{s7}
\noindent
For a reconsideration of Eilenberg's variety theory, this paper put a special emphasis on the consideration of semi-galois categories. The first three sections (\S \ref{s2} -- \S \ref{s4}), for this aim, were devoted to basic studies on the general structure of semi-galois categories. Returning to Eilenberg theory, we discussed in \S \ref{s5} that semi-galois categories are essentially equivalent structures to local varieties of finite actions (Lemma \ref{char of local varieties of finite actions}); this axiomatization of classical structures in Eilenberg theory yielded a rather conceptual description of the variety theorems \cite{Straubing,Chaubard_Pin_Straubing} as a natural consequence of duality theorems betweem abstract algebras; we must, however, discuss now what this axiomatization of the theory is actually intended for. 

In the presentation made at LICS'16, we have posed roughly two types of problems concerning this axiomatization of Eilenberg's theory: One is \emph{logical} and the other is \emph{geometric}. Also, since our formulation of the theory is slightly different from those developed in \cite{Adamek_general,Chen_Urbat,Bojanczyk,Adamek_category}, our result implicitly indicates some new directions to extend Eilenberg's variety theory. We now discuss these problems in the rest of this paper in order to clarify some background motivation of the current study. However, these problems remain open in this paper; some of them will be pursued in future work.

\paragraph{Logical problem}
In the development of Eilenberg theory, B\"uchi's monadic second-order logic MSO[$<$] over finite words has been of central concern; indeed, historically speaking, Eilenberg's theory has been studied, in some sense, so as to develop a finite-semigroup-based method of deciding the definability of regular languages by several fragments of MSO[$<$] (cf.\ Remark \ref{why variety}, \S \ref{section var one}). To the best of our knowledge, however, we do not yet have a satisfactory conceptual understanding of this relationship between Eilenberg's variety theory and MSO[$<$] (or other logics) over finite words. The logical problem of our concern is about this problem in particular. 

In \S \ref{s5} and \S \ref{s6}, we saw that local varieties of regular languages (classes of regualr languages typically defined by logics over finite words such as FO[$<$] and MSO[$<$]) bijectively correspond--- via local varieties of finite actions--- to semi-galois categories $\sgc$; and hence (cf.\ \S \ref{s6}), also to coherent noetherian topoi $(\E,p)$ with coherent surjective points $p: \sets \rightarrow \E$. In particular, the latter structures have several logical descriptions. For example, being coherent topoi, such topoi $\E$ can be realized as classifying topoi of some \emph{coherent theories} \cite{Johnstone}; in other words, one can say that local varieties of regular languages correspond to coherent theories in this way. Nevertheless, coherent theories are logics of many sorts, while MSO[$<$] or other variants over finite words are typically logics of single sort. Therefore, our axiomatic description of local varieties of regular languages (by semi-galois categories and coherent noetherian topoi) themselves does not answer to our problem directly. In view of the relationship between coherent logics and coherent topoi, a natural question may be to design a natural class of logics over finite words so that they describe exactly the local varieties of regular languages. Concerning this problem, nothing is achieved in this paper; but the axiomatic nature of our formulation of Eilenberg theory may be an auxiliary step for this logical problem.

\paragraph{Geometric problem}
Unlike the case of galois categories, we did not consider here concerte geometric constructions of semi-galois categories; but we regard this problem as of particular importance because such constructions will give a concerte meaning (or \emph{semantics}) of Eilenberg's variety theory, hence, a new nuance of classifying hierarchies of regular languages; and conversely, a new insight on number theory as well. 

Since semi-galois categories are seamless extension of galois categories, it seems reasonable to expect that typical constructions of galois categories--- e.g.\ those $\Et(S)$ consisting of \emph{finite {\'e}tale coverings over connected schemes} $S$ (cf.\ \cite{Lenstra}) in particular--- can be extended so that the resulting category forms a semi-galois category rather than a galois category. A natural idea is to relax some condition of this construction; for instance, to drop the {\'e}taleness of coverings, although this naive idea itself is not successful at least to our current understanding. Concerning this problem, it may be meaningful to remark that we cannot prove yet the uniqueness of fiber functors on a semi-galois category, while so are those on a galois category; in fact, we vaguely consider that they are not unique, even up to isomorphism. In the case of galois categories, say $\Et(S)$, the uniqueness of fiber functors corresponds to the fact that the fundamental group $\pi_1(S,\xi)$ does not depend on the choice of the geometric point $\xi$ of $S$ up to isomorphism; in particular, in the case where $S$ is the spectrum $\spec(k)$ of a field $k$, this concerns the fact that the separable closure $k^{sep}$ of $k$ is unique up to isomorphism. The non-uniqueness of fiber functors on a semi-galois category may be a hint of what situation yields semi-galois categories; this problem represents one of the most important future direction of the current work.

\paragraph{Extension of the theory}
As mentioned in the introduction (\S \ref{s1s1}), one of leading motivations to review Eilenberg's theory is to seek a right direction to extend the theory to more subtle situations beyond regular languages. Recently, several authors \cite{Adamek_general,Adamek_category,Chen_Urbat,Bojanczyk} have being tackling on this problem with their respective generic frameworks. 

The result developed in this paper may provide yet another direction for this problem. For instance, in the precise sense described in \S \ref{s5} and \S \ref{s6}, the class of (i) local varieties of regular languages corresponds to that of (ii) coherent noetherian topoi $\langle \E,p \rangle$. In a sense, this gives a topos-based classification of language classes known as local varieties of regular languages. In view of this, it may be possible to ask e.g.\ what class of topoi appear in place of (ii) if we replace (i) with other classes of languages beyond regular languages. If this line of question could make sense in some way, the solution to the question may yield a rather systematic classification of language hierarchies by those of topoi--- to which we can apply several description from the logical and geometric point of view--- seamlessly extending Eilenberg's variety theory. Although this is just an informal (and very naive) description of future direction, it may be worthwile seeking some topos-theoretic extension of Eilenberg's theory in this sense. 

\bibliographystyle{abbrvnat}
\bibliography{preprint_semigalois}

\begin{thebibliography}{39}
\providecommand{\natexlab}[1]{#1}
\providecommand{\url}[1]{\texttt{#1}}
\expandafter\ifx\csname urlstyle\endcsname\relax
  \providecommand{\doi}[1]{doi: #1}\else
  \providecommand{\doi}{doi: \begingroup \urlstyle{rm}\Url}\fi

\bibitem[Abe(1980)]{Abe}
E.~Abe.
\newblock \emph{Hopf Algebras}.
\newblock Number~74 in Cambridge Tracks in Mathematics. Cambridge University
  Press, Cambridge New York, 1980.

\bibitem[Ad{\'{a}}mek et~al.(2014)Ad{\'{a}}mek, Milius, Myers, and
  Urbat]{Adamek_general}
J.~Ad{\'{a}}mek, S.~Milius, R.~Myers, and H.~Urbat.
\newblock Generalized {E}ilenberg variety theorem {I:} local varieties of
  languages.
\newblock In \emph{Foundations of Software Science and Computation Structures},
  pages 366--380, 2014.

\bibitem[Ad{\'{a}}mek et~al.(2015)Ad{\'{a}}mek, Myers, Urbat, and
  Milius]{Adamek_category}
J.~Ad{\'{a}}mek, R.~Myers, H.~Urbat, and S.~Milius.
\newblock Varieties of languages in a category.
\newblock In \emph{30th Annual {ACM/IEEE} Symposium on Logic in Computer
  Science, {LICS} 2015, Kyoto, Japan, July 6--10, 2015}, pages 414--425, 2015.

\bibitem[Almeida(1990)]{Almeida}
J.~Almeida.
\newblock Implicit operations on finite $\mathcal{J}$-trivial semigroups and a
  conjecture of i. simon.
\newblock \emph{J. Pure and Applied Algebra}, 69:\penalty0 205--218, 1990.

\bibitem[Almeida(1994)]{Almeida94}
J.~Almeida.
\newblock \emph{{Finite Semigroups and Universal Algebra}}.
\newblock World Scientific Publishing Co. Inc., River Edge, NJ, 1994.

\bibitem[Boja{\' n}czyk(2015)]{Bojanczyk}
M.~Boja{\' n}czyk.
\newblock Recognizable languages over monads.
\newblock In \emph{19th International Conference Development in Language Theory
  (DLT 2015)}, volume 9168 of \emph{Lecture Notes in Computer Science}, pages
  1--13. Springer, 2015.

\bibitem[Bruguieres(2013)]{Bruguieres}
A.~Bruguieres.
\newblock {G}alois-{G}rothendieck duality, {T}annaka duality and {H}opf
  (co)monads, 2013.
\newblock \url{http://www.math.univ-montp2.fr/~bruguieres/recherche.html}.

\bibitem[Bruguieres(2016)]{Bruguieres_personal}
A.~Bruguieres, February 2016.
\newblock personal communication.

\bibitem[B{\"u}chi(1960)]{Buchi}
R.~B{\"u}chi.
\newblock Weak second-order arithmetic and finite automata.
\newblock \emph{Zeitschrift fur mathematische Logik und Grundlagen der
  Mathematik}, pages 66--92, 1960.

\bibitem[Chaubard et~al.(2006)Chaubard, Pin, and
  Straubing]{Chaubard_Pin_Straubing}
L.~Chaubard, J.-E. Pin, and H.~Straubing.
\newblock Actions, wreath products of {C}-varieties and concatenation products.
\newblock \emph{Theoret. Comput. Sci.}, 356:\penalty0 73--89, 2006.

\bibitem[Chen and Urbat(2015)]{Chen_Urbat}
L.-T. Chen and H.~Urbat.
\newblock A fibrational approach to automata theory.
\newblock In \emph{6th International Conference on Algebra and Coalgebra in
  Computer Science (CALCO 2015)}, volume~35, pages 50--65, 2015.

\bibitem[Chen and Urbat(2016)]{Chen_Urbat_product}
L.-T. Chen and H.~Urbat.
\newblock Sch{\"u}tzenberger products in a category.
\newblock In \emph{20th International Conference on Developments in Language
  Theory}, pages 89--101. Springer, 2016.

\bibitem[Chen et~al.(2016)Chen, Ad{\`a}mek, Millius, and
  Urbat]{Millius_profinite_monad}
L.-T. Chen, J.~Ad{\`a}mek, S.~Millius, and H.~Urbat.
\newblock Profinite monads, profinite equations, and {R}eiterman's theorem.
\newblock In \emph{FoSSaCS'16}, 2016.

\bibitem[Deligne and Milne(1982)]{Deligne_Milne}
P.~Deligne and J.~Milne.
\newblock Tannakian categories, in {H}odge cycles, motives, and {S}himura
  varieties.
\newblock 900:\penalty0 101--228, 1982.

\bibitem[Diekert et~al.(2008)Diekert, Gastin, and Kufleitner]{survey_logic}
V.~Diekert, P.~Gastin, and M.~Kufleitner.
\newblock A survey on small fragments of first-order logic over finite words.
\newblock \emph{Int. J. Found. Compt. Sci.}, 19:\penalty0 513--548, 2008.

\bibitem[Eilenberg(1976)]{Eilenberg}
S.~Eilenberg.
\newblock \emph{Automata, Languages and Machines, vol. B}.
\newblock Academic Press, New York, 1976.

\bibitem[Gehrke et~al.(2008)Gehrke, Grigorieff, and Pin]{Gehrke_Grigorieff_Pin}
M.~Gehrke, S.~Grigorieff, and J.~Pin.
\newblock Duality and equational theory of regular languages.
\newblock In \emph{Automata, Languages and Programming, 35th International
  Colloquium, {ICALP} 2008}, pages 246--257, 2008.

\bibitem[Gehrke et~al.(2016)Gehrke, Krebs, and Pin]{Gehrke_ultrafilter}
M.~Gehrke, A.~Krebs, and J.~Pin.
\newblock Ultrafilters on words for a fragment of logic.
\newblock \emph{Theoretical Computer Science}, pages 37--58, 2016.

\bibitem[Grothendieck and Raynaud(2002)]{SGA}
A.~Grothendieck and M.~Raynaud.
\newblock Rev{\^{e}}tements {\`{e}}tales et groupe fondamental ({SGA} 1).
\newblock Available from: \url{http://arxiv.org/abs/math/0206203}, 2002.

\bibitem[Hopcroft et~al.(2006)Hopcroft, Motwani, and Ullman]{Hopcroft_Ulman}
J.~Hopcroft, R.~Motwani, and J.~Ullman.
\newblock \emph{Introduction to Automata Theory, Languages, and Computation}.
\newblock Addison Wesley, 2006.

\bibitem[Johnstone(2014)]{Johnstone}
P.~Johnstone.
\newblock \emph{Topos Theory}.
\newblock Dover Publications, 2014.

\bibitem[Johnstone(2003)]{Johnstone_elephant}
P.~Johnstone.
\newblock \emph{Sketches of An Elephant: A Topos Theory Compendium Vol. 2}.
\newblock Clarendon Press, 2003.

\bibitem[Lenstra()]{Lenstra}
H.~Lenstra.
\newblock Galois theory for schemes.

\bibitem[MacLane(1978)]{MacLane}
S.~MacLane.
\newblock \emph{Categories for the working mathematician}.
\newblock Springer-Verlag New York, 1978.

\bibitem[MacLane and Moerdijk(1992)]{MacLane_Moerdijk}
S.~MacLane and I.~Moerdijk.
\newblock \emph{Sheaves in Geometry and Logic}.
\newblock Springer New York, 1992.

\bibitem[Moerdijk and Vermeulen(2000)]{Moerdijk_proper}
I.~Moerdijk and J.~Vermeulen.
\newblock \emph{Proper maps of toposes}.
\newblock American Mathematical Society, 2000.

\bibitem[Pin(1995)]{Pin_var}
J.~Pin.
\newblock A variety theorem without complementation, 1995.

\bibitem[Pin()]{Pin}
J.-E. Pin.
\newblock Mathematical foundation of automata theory.
\newblock Available from:
  \url{http://www.liafa.jussieu.fr/~jep/PDF/MPRI/MPRI.pdf}.

\bibitem[Pippenger(1997)]{Pippenger}
N.~Pippenger.
\newblock {Regular Languages and {S}tone Duality}.
\newblock \emph{Theory Comput. Syst.}, 30\penalty0 (2):\penalty0 121--134,
  1997.

\bibitem[Pol{\' a}k(2001)]{Polak}
L.~Pol{\' a}k.
\newblock Syntactic semirings of a language.
\newblock In \emph{International Symposium on Mathematical Foundations of
  Computer Science (MFCS)}, volume 2136 of \emph{Lecture Notes in Computer
  Science}, pages 611--620. Springer, 2001.

\bibitem[Priestley(1970)]{Priestley}
H.~Priestley.
\newblock Representation of distributive lattices by means of ordered {S}tone
  spaces.
\newblock \emph{Bull. London Math. Soc.}, 2:\penalty0 186--190, 1970.

\bibitem[Reiterman(1982)]{Reiterman}
J.~Reiterman.
\newblock The {B}irkhoff theorem for finite algebras.
\newblock \emph{Algebra Universalis}, 14\penalty0 (1):\penalty0 1--10, 1982.

\bibitem[Reutenauer(1980)]{Reutenauer}
C.~Reutenauer.
\newblock S{\'e}ries formelles et alg{\`e}bres syntactiques.
\newblock \emph{J. Algebra}, 66:\penalty0 448--483, 1980.

\bibitem[Rhodes and Steinberg(2008)]{Rhodes_Steinberg}
J.~Rhodes and B.~Steinberg.
\newblock \emph{The $\mathfrak{q}$-Theory of Finite Semigroups}.
\newblock Springer Monographs in Mathematics. Springer, 2008.

\bibitem[Sch{\"a}ppi(2011)]{Schappi}
D.~Sch{\"a}ppi.
\newblock The formal theory of tannaka duality.
\newblock 2011.
\newblock arXiv:1112.5213[math.CT].

\bibitem[Straubing(2002)]{Straubing}
H.~Straubing.
\newblock On logical descriptions of regular languages.
\newblock In \emph{LATIN 2002}, Lect. Notes Comp. Sci. Springer, 2002.

\bibitem[Szamuely(2009)]{Szamuely}
T.~Szamuely.
\newblock \emph{Galois Groups and Fundamental Groups}.
\newblock Cambridge University Press, 2009.

\bibitem[Tonini(2009)]{Tonini}
F.~Tonini.
\newblock Notes on {G}rothendieck-{G}alois theory, October 2009.

\bibitem[Uramoto(2016)]{Uramoto16}
T.~Uramoto.
\newblock Semi-galois categories {I}: The classical {E}ilenberg variety theory.
\newblock In \emph{Proc. of LICS'16}, pages 545--554. ACM, 2016.

\end{thebibliography}

\end{document}